\documentclass[onefignum,onetabnum]{siamart190516}

\headers{Geometry Error of Parametric Mapping For Unfitted Space-Time FEM}{F. Heimann, C. Lehrenfeld}

\title{
Geometry Error Analysis of a parametric mapping for Higher Order Unfitted Space-Time Methods
\thanks{Submitted to the editors DATE.
}}

\author{Fabian Heimann
\and Christoph Lehrenfeld\thanks{Institut f\"ur Numerische und Angewandte Mathematik, Universit\"at G\"ottingen, 
(\texttt{\{f.heimann, lehrenfeld\}@math.uni-goettingen.de}).}}

\usepackage{lipsum}
\usepackage{amsfonts}
\usepackage{amsmath}
\usepackage{amssymb}
\usepackage{nicefrac}
\usepackage{graphicx}
\usepackage[caption=false]{subfig}
\usepackage{epstopdf}
\usepackage{algorithmic}
\ifpdf
  \DeclareGraphicsExtensions{.eps,.pdf,.png,.jpg}
\else
  \DeclareGraphicsExtensions{.eps}
\fi

\usepackage[normalem]{ulem}

\newsiamremark{remark}{Remark}
\newsiamremark{assumption}{Assumption}
\newsiamremark{hypothesis}{Hypothesis}
\crefname{hypothesis}{Hypothesis}{Hypotheses}
\newsiamthm{claim}{Claim}

\usepackage{amsopn}

\usepackage[]{hyperref}
\usepackage[capitalize]{cleveref}

\providecommand{\vertiii}[1]{{\left\vert\kern-0.25ex\left\vert\kern-0.25ex\left\vert #1 
\right\vert\kern-0.25ex\right\vert\kern-0.25ex\right\vert}}

\usepackage{tikz}
\usepackage{tikz-3dplot}
\usetikzlibrary{fpu,positioning,shapes,arrows,calc, arrows.meta, patterns, intersections, decorations.pathreplacing,calligraphy}
\usetikzlibrary{shapes,arrows,positioning,fit,calc,spy, colorbrewer}

\usepackage{pgfplots}
\usepgfplotslibrary{colorbrewer,groupplots}

\pgfplotsset{
cycle list/Set1-5,
cycle multiindex* list={
mark list*\nextlist
Set1-5\nextlist
},
}

 \usepackage{shellesc}
 \usetikzlibrary{external}

\usepackage{booktabs}

\usepackage{graphics}

\usepackage{mathtools}

\providecommand{\id}{{\operatorname{id}}}
\providecommand{\idx}{\hyperlink{def:idx}{\operatorname{id}_{{x}}}}
\providecommand{\idxt}{\hyperlink{def:idx}{\operatorname{id}_{{x},t}}}
\providecommand{\Idx}{\hyperlink{def:Idx}{\operatorname{Id}_{{x}}}}
\providecommand{\Idxt}{\hyperlink{def:Idxt}{\operatorname{Id}_{{x},t}}}

\providecommand{\In}[1]{\hyperlink{def:In}{I_{#1}}}
\providecommand{\Qn}[1]{\hyperlink{def:Qn}{Q^{#1}}}

\providecommand{\Gh}{\hyperlink{def:Gh}{G_h}}
\providecommand{\GH}{\hyperlink{def:GH}{G_H}}
\providecommand{\Ghi}[1]{\hyperlink{def:Ghi}{G_{h,#1}}}
\providecommand{\PhG}{\hyperlink{def:PhG}{P_h^\Gamma}}

\providecommand{\ppi}{\hyperlink{def:pi}{\pi}}
\providecommand{\phix}{\hyperlink{def:phi}{\phi}}

\providecommand{\philin}{\hyperlink{def:philin}{\phi^{\text{lin}}}}
\providecommand{\philini}[1]{\hyperlink{def:philini}{\phi^{\text{lin}}_{#1}}}

\providecommand{\dU}{\hyperlink{def:dU}{d_U}}
\providecommand{\lphi}{\hyperlink{def:lphi}{\ell_{\!\phi}}}

\providecommand{\Sinf}[2]{C^{#1}(#2)}

\providecommand{\FE}{\hyperlink{def:FE}{FE}}
\providecommand{\FEs}{\hyperlink{def:FE}{FEs}}
\providecommand{\phih}{\hyperlink{def:phih}{\phi_h}}
\providecommand{\phihi}[1]{\hyperlink{def:phihi}{\phi_{h,#1}}}

\providecommand{\partials}{\hyperlink{def:partials}{\partial^s}}

\providecommand{\PsiG}{\hyperlink{def:PsiG}{\Psi^\Gamma}}
\providecommand{\PsiGi}[1]{\hyperlink{def:PsiGi}{\Psi^\Gamma_{#1}}}
\providecommand{\QGn}{\hyperlink{def:QGn}{Q_h^\Gamma}}
\providecommand{\QGnp}{\hyperlink{def:QGnp}{Q_{h,+}^\Gamma}}
\providecommand{\Dxt}{\hyperlink{def:Dxt}{D_{\mathbf{x},t}}}

\providecommand{\tOmega}{\hyperlink{def:tOmega}{\tilde{\Omega}}}
\providecommand{\Th}{\hyperlink{def:Th}{\mathcal{T}_h}}
\providecommand{\ThU}{\hyperlink{def:ThU}{\mathcal{T}_h^U}}
\providecommand{\QhU}{\hyperlink{def:QhU}{\mathcal{Q}_h^U}}
\providecommand{\QU}{\hyperlink{def:QU}{Q_h^U}}
\providecommand{\Wh}[1]{\hyperlink{def:Wh}{{W}_h^{#1}}}
\providecommand{\Vh}[1]{\hyperlink{def:Vh}{{V}_h^{#1}}}
\providecommand{\U}{\hyperlink{def:U}{U}}
\providecommand{\Ust}[1]{\hyperlink{def:Ust}{U_s(#1)}}
\providecommand{\Ubar}{\hyperlink{def:Ubar}{\bar{U}}}
\providecommand{\phii}[1]{\hyperlink{def:phii}{\phi_{#1}}}
\providecommand{\ti}[1]{\hyperlink{def:ti}{t_{#1}}}
\providecommand{\phidt}{\hyperlink{def:phidt}{\phi_{\Delta t}}}
\providecommand{\phiH}{\hyperlink{def:phiH}{\phi_{H}}}

\providecommand{\phidti}[1]{\hyperlink{def:phidti}{\phi_{\Delta t, #1}}}
\providecommand{\Qlin}{\hyperlink{def:Qlinn}{Q^{\text{lin}}}}
\providecommand{\EQlin}{\hyperlink{def:EQlinn}{Q^{\text{lin}}_{\mathcal{E}}}}

\providecommand{\Omlin}[1]{\hyperlink{def:Omlin}{\Omega^{\text{lin}}(#1)}}

\providecommand{\EOmlin}{\hyperlink{def:EOmlin}{\Omega^{\text{lin}}_{\mathcal{E}}}}
\providecommand{\cnphi}[1]{\hyperlink{def:cnphi}{c_{\nabla\!\phi}^{#1}}}
\providecommand{\Cnphi}[1]{\hyperlink{def:Cnphi}{C_{\nabla\!\phi}^{#1}}}
\providecommand{\tQn}{\hyperlink{def:tQn}{\tilde{Q}^n}}

\providecommand{\Hrs}[2]{\hyperlink{def:Hrs}{H_{\scalebox{0.5}[0.35]{$\square$}}^{#1,#2}}}

\renewcommand{\b}{\hyperlink{def:b}{b}}
\renewcommand{\d}{\hyperlink{def:d}{d}}
\providecommand{\bOne}{\hyperlink{def:b1}{b_1}}
\providecommand{\bTwo}{\hyperlink{def:b2}{b_2}}
\providecommand{\di}[1]{\hyperlink{def:di}{d_{#1}}}
\providecommand{\G}{\hyperlink{def:G}{G}}
\providecommand{\Gi}[1]{\hyperlink{def:Gi}{G_{#1}}}
\providecommand{\dOmega}{\hyperlink{def:dOmega}{d_{\Omega}}}

\providecommand{\Verts}{\hyperlink{def:verts}{\mathcal{V}}}
\providecommand{\Facets}{\hyperlink{def:facets}{\mathcal{F}}}

\providecommand{\ThQ}{\hyperlink{def:ThQ}{\mathcal{Q}_h^n}}
\providecommand{\ThGQ}{\hyperlink{def:ThGQ}{\mathcal{Q}_h^\Gamma}}
\providecommand{\ThGQp}{\hyperlink{def:ThGQp}{\mathcal{Q}_{h,\!+}^\Gamma}}
\providecommand{\ThGQS}{\ThGQp\!\!\setminus\ThGQ}

\providecommand{\ThG}{\hyperlink{def:ThG}{\mathcal{T}_h^\Gamma}}
\providecommand{\ThGp}{\hyperlink{def:ThGp}{\mathcal{T}_{h,+}^\Gamma}}
\providecommand{\ThbOne}{\hyperlink{def:Thb1}{\mathcal{T}^\Gamma_{b1}}}
\providecommand{\ThbTwo}{\hyperlink{def:Thb2}{\mathcal{T}^\Gamma_{b2}}}

\providecommand{\Thest}{\hyperlink{def:Thest}{\Theta_h^{\text{st}}}}
\providecommand{\Theh}{\hyperlink{def:Theh}{\Theta_h}}
\providecommand{\Psist}{\hyperlink{def:Psist}{\Psi^{\text{st}}}}
\providecommand{\Psii}{\hyperlink{def:Psi}{\Psi}}

\providecommand{\OG}{\hyperlink{def:OG}{\Omega^\Gamma}}
\providecommand{\OGp}{\hyperlink{def:OGp}{\Omega_{+}^\Gamma}}
\providecommand{\Ohi}[1]{\hyperlink{def:Oh}{\Omega_h(#1)}}

\providecommand{\OGbOne}{\hyperlink{def:OGbOne}{\Omega^\Gamma_{b1}}}
\providecommand{\OGbTwo}{\hyperlink{def:OGbTwo}{\Omega^\Gamma_{b2}}}

\providecommand{\PsihG}{\hyperlink{def:PsihG}{\Psi^\Gamma_{h}}}
\providecommand{\PsihGi}[1]{\hyperlink{def:PsihGi}{\Psi^\Gamma_{h,#1}}}
\providecommand{\PsidtG}{\hyperlink{def:PsidtG}{\Psi^\Gamma_{\Delta t}}}
\providecommand{\PsiHG}{\hyperlink{def:PsiHG}{\Psi^\Gamma_{H}}}
\providecommand{\PsidtGi}[1]{\hyperlink{def:PsidtGi}{\Psi^\Gamma_{\Delta t,#1}}}
\providecommand{\Phihst}{\hyperlink{def:Phihst}{\Phi^{\text{st}}_{h}}}

\providecommand{\ddti}[1]{\hyperlink{def:ddti}{d_{\Delta t,#1}}}
\providecommand{\ddh}{\hyperlink{def:dh}{d_{h}}}
\providecommand{\dH}{\hyperlink{def:dH}{d_{H}}}
\providecommand{\dhi}[1]{\hyperlink{def:dhi}{d_{h,#1}}}
\providecommand{\Gdti}[1]{\hyperlink{def:Gdti}{G_{\Delta t,#1}}}
\providecommand{\ET}{\hyperlink{def:ET}{\mathcal{E}_T}}
\providecommand{\elli}[1]{\hyperlink{def:elli}{\ell_{#1}}}

\providecommand{\bi}[1]{\hyperlink{def:bi}{b_{i}}}

\providecommand{\It}[1]{\hyperlink{def:It}{I^t_{#1}}}
\providecommand{\Is}[1]{\hyperlink{def:Is}{I^s_{#1}}}

\providecommand{\ThnG}[1]{\hyperref[eq:new_st_isoparam_regions]{\mathcal{T}_{h,#1}^{\Gamma}}}

\providecommand{\ThehG}{\hyperlink{def:ThehG}{\Theta_h^\Gamma}}
\providecommand{\Thehst}{\hyperlink{def:Thehst}{\Theta_h^{\text{st}}}}
\providecommand{\Qhn}{\hyperlink{def:Oh}{Q_{h}^n}}

\ifpdf
\hypersetup{
  pdftitle={Geometry Error Analysis of a parametric mapping for Higher Order Unfitted Space-Time Methods},
  pdfauthor={F. Heimann, C. Lehrenfeld}
}
\fi
 
\begin{document}

\maketitle

\begin{abstract}
In [Heimann, Lehrenfeld, Preu\ss{}, SIAM J. Sci. Comp. 45(2), 2023, B139 - B165] new geometrically unfitted space-time Finite Element methods for partial differential equations posed on moving domains of higher-order accuracy in space and time have been introduced.
For geometrically higher-order accuracy a parametric mapping on a background space-time tensor-product mesh has been used. In this paper, we concentrate on the geometrical accuracy of the approximation and derive rigorous bounds for the distance between the realized and an ideal mapping in different norms and derive results for the space-time regularity of the parametric mapping. These results are important and lay the ground for the error analysis of corresponding unfitted space-time finite element methods.
\end{abstract}



\section{Introduction}
\hypertarget{def:fe}{Finite element (FE)} simulations play an important role in engineering and other computational sciences \cite{brennerscott, hughes2012finite}. In many of those applications, complex geometries are involved. The most straightforward strategy for taking these geometries into account is to generate a mesh following the geometry. For polytopal domains, for instance, a simplicial mesh can be generated in an exact fit with the physical domain. For other domains -- imagine for instance a sphere -- a polygonal approximation can be meshed and a mapping is applied to curve the boundary elements to arrive at a higher-order accurate domain approximation \cite{CIARLET1972217, bernardi89}. 
In the last two decades, unfitted \FEs~%
have been proposed as an alternative approach to solving PDEs on complex domains \cite{burman15}. The idea here is that the \FE~simulation is based on a background grid which does not take into account the geometry of the problem. Instead, the geometry is prescribed separately, e.g. implicitly by a levelset function $\phi$. In that way, the computational task of mesh generation can be avoided, and new questions concerning e.g. numerical quadrature on the implicitly cut elements arise.

By design, the unfitted paradigm becomes particularly attractive for problems involving time-dependent domains, because in that case the mesh must otherwise be carefully transferred or regenerated. We focus on this setting in this article. Our main contribution is the development and mathematical validation of a strategy to provide a computationally feasible and higher-order accurate unfitted geometry description for time-dependent \emph{smooth} domains. This provides a crucial component for unfitted \FE~simulations, in particular space-time discretisations for bulk- or surface problems. We presented such a bulk space-time \FE~method for the convection-diffusion problem in \cite{HLP2022}, whilst a similar method for surface problems have been considered in \cite{sassreusken23}. The rigorous mathematical error analysis of both methods, which will be the topic of forthcoming papers, heavily exploits the results from this paper. In that way, we split the mathematical analysis between mere geometry considerations, which are considered in this work, and the problem and method-specific analysis sections, which will be treated in separate works.

A crucial distinction is that between second-order and higher-order accurate geometry representations in space on simplicial meshes: For a second-order accurate geometry description, for which the maximal distance between exact and discrete geometry should be of order $h^2$, where $h$ denotes the mesh size, it suffices to generate an elementwise linear interpolant of the levelset function and use the implied discrete geometry. This discrete geometry will be a simple polytope on each element, so that integration can be performed easily by tesselation\footnote{
Similar statements also hold for quadrilateral and hexahedral elements based on bi- and trilinear levelset representations, see \cite{HL_ENUMATH_2019}.} \cite{L_PHD_2015}. One option for obtaining a higher-order accurate geometry description suggests improving the quality of the second-order approximation using an isoparametric mapping \cite{L_CMAME_2016, LR_IMAJNA_2018}. In this paper, we want to build on this idea, specifically, the numerical analysis in \cite{LR_IMAJNA_2018}, and extend it to the space-time setting. We have seen numerically in \cite{HLP2022} that the generalisation to tensor-product space-time levelset representations is feasible.
Compared to \cite{HLP2022}, where this statement is validated by numerical experiments, we provide a rigorous numerical analysis in this paper.
\subsection{Approaches in the literature} \label{ssec:lit}
We would like to relate our methods to the literature on \emph{unfitted} discretisations.

First, the issue of numerical integration can be numerically solved with special integration techniques, cf. e.g. \cite{saye15,friesomerovic16}. To our best knowledge, whilst the methods make reasonable approximations of relevant quantities, their well-functioning has not yet been mathematically analysed rigorously. Some of these methods come with other restrictions such as quadrilateral (or hexahedral) elements or the necessity of local refinements within one element, which can potentially be computationally expensive for high-order accuracy demands. 
Similar integration strategies are also applied in other approaches in the literature. In particular, in \cite{BADIA202360}, a variant space-time method is presented and analysed, which introduces a cell agglomeration technique. The cell-agglomeration technique is an alternative to the so-called Ghost-penalty stabilisation which is often used in methods such as \cite{HLP2022} in order to ensure stability in the presence of small cuts. The numerical analysis in \cite{BADIA202360} applies in both cases. In terms of geometries in the numerical analysis, \cite{BADIA202360} does not take into account the inexactness of discrete geometries, but the exact domains are used. We aim to follow the numerical implementation more closely in the sense of taking into account the discrete geometry in the detailed numerical analysis. We note that for the DG-in-time method of \cite{HLP2022}, an analysis assuming the exact handling of geometries has also been presented in \cite{preuss18}, and an analysis involving a semi-discrete geometry in \cite{heimann20}.

Further applications of the unfitted space-time methodology been investigated for an osmotic cell swelling problem in \cite{wendler22}, the heat equation on two overlapping meshes in \cite{larson2023spacetime}, Navier-Stokes equations in \cite{anselmann2021cutfemNSE,AB21},
two-phase flow problems in \cite{LR_SINUM_2013, L_PHD_2015, FRACHON201977, FRACHON2023111734} and coupled surface-bulk problems in \cite{Z18, HLZ16}.

Apart from the space-time discretisations on which we focus here, time-stepping methods provide an alternative approach to solving the same physical problems. 
In particular, those methods have been studied for a bulk convection-diffusion problem in low- \cite{LO_ESAIM_2019} and high-order \cite{LL_ARXIV_2021}, the Stokes-like problems \cite{IMAJNA_vWRL_2021, burman2020eulerian, neilan2023eulerian}. Moreover, Trace-FEM discretisations for surface problems have been studied \cite{LOX, doi:10.1137/21M1403126, olshanskii2023eulerian}.

\subsection{The spatial case of the unfitted isoparametric FEM} \label{ssec:spatial}
In this section, we outline the main concepts in the construction and analysis of the isoparametric FEM for the \emph{spatial} setting, as the construction and analysis of the \emph{space-time} version share a similar structure. 
We refer to \cite{L_CMAME_2016} and \cite{LR_IMAJNA_2018} for details on the \emph{spatial} setting. 

Assume we want to approximate the potentially complicated physical domain $\Omega$, which is embedded in some background domain $\tilde \Omega$ equipped with a simplicial\footnote{We restrict to simplicial meshes for ease of presentation. Extensions to quadrilateral and hexahedral meshes, following \cite{HL_ENUMATH_2019}, are however easily possible.} mesh $\mathcal{T}_h$, $\bigcup_{T\in \mathcal{T}_h} \overline{T} = \overline{\tilde \Omega}$. Further, assume that a levelset function $\phi(x)$ encodes $\Omega$ s.t.
\begin{equation}
 \Omega = \{ x \in \tilde \Omega \, | \, \phi(x) < 0\},
\end{equation}
and that an $\mathcal{T}_h$-elementwise linear interpolation of $\phi$, denoted by $\phi^{\text{lin}}$ is given. On 
\begin{equation}
 \Omega^{\text{lin}} := \{ x \in \tilde \Omega \, | \, \phi^{\text{lin}}(x) < 0\}
\end{equation}
the task of numerical integration 
is computationally feasible as the local polygonal part of $\Omega^{\text{lin}}$ on each respective $T \in \mathcal{T}_h$ can be subdivided in few simplices\footnote{The detailed distinction of cases for 2d and 3d is not complicated, but too long to be given in detail here. We refer the reader to \cite{L_PHD_2015}.}, cf. l.h.s. of \cref{fig:isossketch} for an example. This allows for the construction of high-order accurate quadrature rules on $\Omega^{\text{lin}}$ with \emph{positive quadrature weights}. 
The isoparametric mapping now aims at providing a mesh deformation $\Theh \colon \tilde \Omega \to \tilde \Omega$, such that the image of $\Omega^{\text{lin}}$ under this mapping is a computationally feasible and higher-order accurate approximation of $\Omega$, cf. r.h.s. of \cref{fig:isossketch} for an example.
For this to work the domain has to be sufficiently well captured by $\Omega^{\text{lin}}$, i.e. one requires sufficient mesh resolution. 

\begin{figure}
  \begin{center}
    \includegraphics[width=0.65\textwidth]{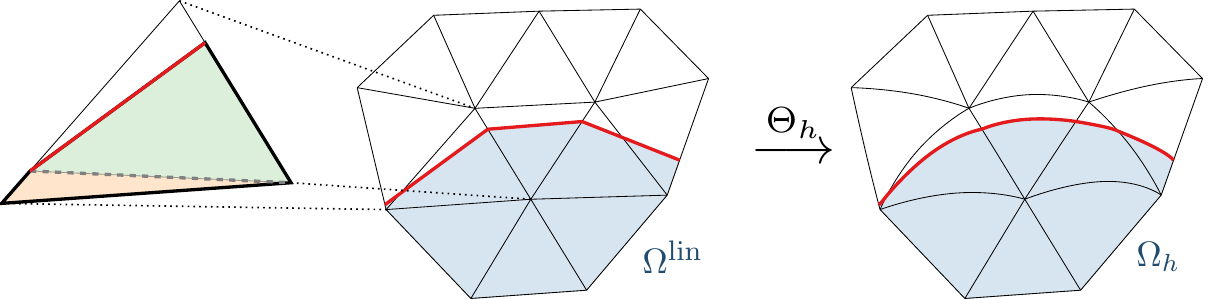}
  \end{center}
  \vspace*{-0.2cm}
  \caption{Sketch of a simple subdivision for quadrature on $T \cap \Omega^{\text{lin}}$, the discrete geometry $\Omega^{\text{lin}}$ and the mapped geometry $\Omega^h$ obtained from the mapping $\Theta_h$. Reproduction from \cite{HLP2022}.}
  \vspace*{-0.25cm} \label{fig:isossketch}
\end{figure}

The analysis of this mapping starts with the construction of a theoretical and element-local counterpart of $\Theh$, i.e. a mapping $\Psi^\Gamma\colon \bigcup \{ T \in \mathcal{T}_h \, | \, T \cap \Omega^{\text{lin}} \neq \varnothing \} \to \tilde \Omega $ with the property $\Psi^\Gamma(\partial \Omega^{\text{lin}}) = \partial \Omega$. So, the criterion of computational feasibility is neglected for the moment to construct an ``ideal'' mapping. An appropriate search direction $G$ is fixed and for all points within a cut element, $x \in \{ T \in \mathcal{T}_h \, | \, T \cap \Omega^{\text{lin}} \neq \varnothing \}$, $\Psi^\Gamma(x)$ is defined to be the point which is the closest one along the search direction to reproduce the value of $\phi^{\text{lin}}(x)$ now at $\phi(\Psi^\Gamma(x))$. For interpolation results within the isoparametric unfitted FEM exploiting this mapping, the boundedness in relevant norms of the extension of $\Psi^\Gamma$ onto all of $\tilde \Omega$, which is called $\Psi$ will be central. In \cite{LR_IMAJNA_2018}, a \FE~blending procedure is suggested to arrive at $\Psi = \mathcal{E} \Psi^\Gamma$.

In the next step, 
a discrete \FE~approximation of the ideal mapping $\Psi^\Gamma$, which is called $\ThehG$, is constructed. 
In practice, to reproduce the values of a higher-order discrete approximation of $\phi$, denoted by $\phi_h$, a discrete search direction $G_h$ replaces the search direction $G$
and a subsequent application of an Oswald-type interpolator is used to obtain inter-element continuity. Finally, $\ThehG$ is extended with the same blending procedure 
to yield $\Theta_h := \mathcal{E} \ThehG$.

The desired result that the so-constructed discrete approximation is of higher-order is stated and proven in this framework as a small difference between $\Psi$ and $\Theta_h$. In order to be able to include these results in a Strang-Lemma-type analysis of a particular discretisation of a PDE problem (in the case of \cite{LR_IMAJNA_2018} a two-domain Poisson problem is considered, in the case of \cite{L_GUFEMA_2017} a fictitious domain problem and in \cite{grande2018analysis} a surface Poisson problem), the mapping $\Phi_h$ is defined as $\Phi_h := \Psi \circ (\Theta_h)^{-1}$ and bounded close to the identity. Overall, the following three components are crucial for the success of the approach (in the spatial setting as well as in the \emph{space-time} setting):
\begin{itemize}
  \item \textit{Robust quadrature rules} based on positive quadrature weights on $\Omega^{\text{lin}}$. 
  \item \textit{Geometrical accuracy}, i.e. higher-order bounds on $\| \Phi_h - \id \|$ in relevant norms.
  \item \textit{Approximability of sufficiently smooth functions on deformed meshes}, which is transferred from standard interpolation results of the uncurved mesh $\mathcal{T}_h$ if sufficiently high spatial derivatives of $\Psi$ stay uniformly bounded.
\end{itemize}
This overall structure for a rigorous mathematical analysis in the spatial case provides a blueprint for the analysis of the space-time case, as we shall see in more detail later.
\subsection{Space-Time isoparametric mapping}
Before we proceed with the analysis for the space-time case, let briefly summarize the central idea of the space-time generalisation of the isoparametric mapping. For more details, we refer the reader to our computational exposition \cite{HLP2022}.

In the moving domain setting, the physical domain is allowed to change in time, i.e. $ \Omega = \Omega (t) \subseteq \tilde \Omega$, where $\tilde \Omega$ is a domain with an unfitted simplicial mesh as before. Accordingly, the levelset function is assumed to be a space-time function, $\phi = \phi(x,t)$. Generalising the idea of a higher-order accurate approximate levelset function $\phi_h$, we assume that there exist spatially higher order $\phi_h^i$ (for $i=0,\dots,q_t$) such that
$  
 \phi_h(x,t) = \sum_{i=0}^{q_t} \ell_i(t) \cdot \phi_h^i(x),
$
where 
$\mathcal{P}^{q_t}(I_n)= \mathrm{span} \{\ell_0,\dots,\ell_{ q_t } \}$
and $\mathcal{P}^{q_t}(I_n)$ denotes the set of polynomials of order up to $q_t$ on the time interval of interest $I_n$. Denoting by $\Theta_{h}(\phi_h, \phi^{\text{lin}})$ the \emph{merely spatial} isoparametric mappings, we can obtain the space-time isoparametric mapping in the following tensor-product fashion:
  \begin{align}
    \Theta_h^n(x,t) &= {\sum}_{i=0}^{q_t} \ell_i(t) \cdot \Theta_h(\phi_h^i(x),I_h^1 \phi_h^i (x)). 
  \end{align}
In this equation, $I_h^1$ shall refer to an interpolation operator into the space of continuous elementwise linear finite element functions. In that way, the discrete polynomial basis in time enables the generalisation of the space-time mapping.

A slight difference to the merely spatial contruction $\Theta_h( \cdot, \cdot)$ as in \cite{L_CMAME_2016, LR_IMAJNA_2018} relates to the transition of the mapping from the vicinity of the domain boundary to the bulk, where no mesh deformation is wanted. In the merely spatial case, a Finite Element blending is used, which can be extended to the space-time case by defining the cut elements with respect to the overall time slice, as it was done in \cite{HLP2022}. Alternatively, a smooth blending function can be involved, which yields helpful features in the numerical analysis, as we shall see later in more detail. 

The task of numerical integration can be solved using $\Theta_h^n(\cdot, t)$ to go from the integration domain to a reference configuration based on $I_h^1 \phi_h (\dots,t)$. For this configuration (adapted) iterated integrals in time and simplex subdivision strategies yields robust and higher order accurate quadrature, cf. \cite[Section 5]{HLP2022}.

Let us stress that the construction of the space-time mapping heavily relies on the spatial problem described in \cref{ssec:spatial}. However, for the analysis of the geometry errors and properties of the resulting space-time mapping, errors and space and time do not decouple so that the generalization of the spatial results is not straightforward.

A sketch of involved domains, subdivision and quadrature points (in space) for a fixed time $t \in I_n$ is displayed in \Cref{fig:int_points_moving} for an example case.
\begin{figure}
 \begin{center}
 \href{https://www.youtube.com/watch?v=DIiRt1yhfhU}{ \includegraphics[width=0.9\textwidth]{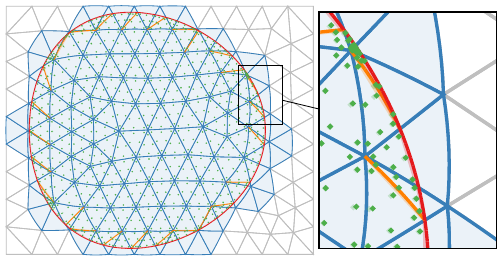} }
 \end{center}
 \vspace*{-0.5cm}
 \caption{
  Discrete regions and integration points for a test geometry (cf. \cite{HLP2022}). Within the time slice, all elements contributing to the space-time domain are shown in blue. All of these (and their neighbors) are deformed by the mapping. This yields a higher order approximation of the physical domain, shown in red. To illustrate the improved approximation properties, in light red we display the (spatially) second order accurate interface. In orange, the subdivision into simplices for the purpose of numerical integration is shown, with solid lines after deformation and in light colour before. On the so-contructed regions, numerical integration (in space) can be performed by mapping standard Gaussian simplex rules onto the geometry. These points are shown in green. Here, the FE~blending function is used. An animated version is available at 
\url{https://www.youtube.com/watch?v=DIiRt1yhfhU} 
and the version for the smooth blending at 
\url{https://www.youtube.com/watch?v=CWI-YzeET6E}.
 }
 \label{fig:int_points_moving}
 \vspace*{-0.25cm}
\end{figure}

\subsection{Structure of the paper}
In this paper, we want to essentially reproduce the aforementioned results for the spatial setting now in relation to this generalisation in space and time. As it turns out, the interplay of temporal and spatial errors plays a crucial role and we aim for high-order error bounds under mild conditions on the anisotropy of space and time resolution.
To this end,
this paper is structured as summarised in the following.\footnote{In this work a large number of symbols and definitions need to be handled by the reader. To facilitate backtracking, we have introduced hyperlinks to the respective definitions and symbols. These hyperlinks are marked in green and can be used to conveniently find the respective definitions.} See also \Cref{fig:num_analysis_overview} for an overview illustration.

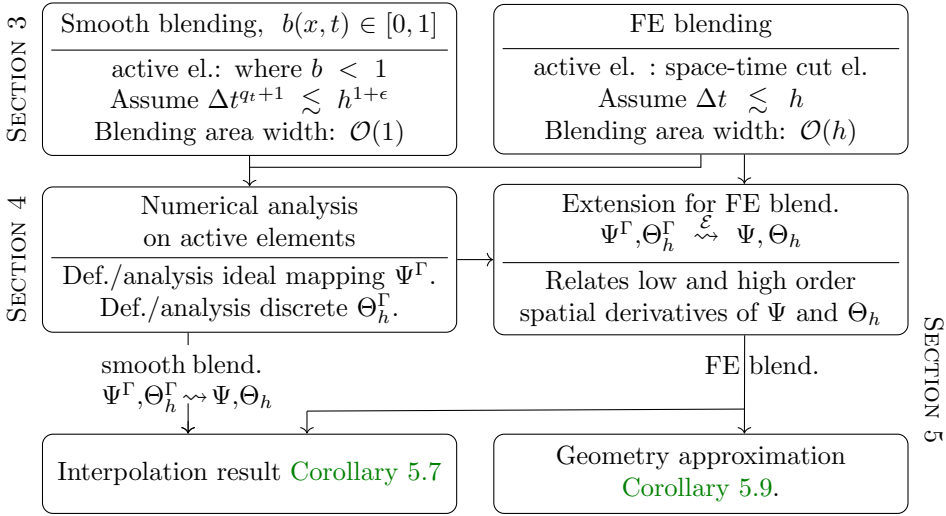
\begin{figure}
  \tikzset{external/export next=false} \centering
\begin{tikzpicture}[
    line/.style={draw, -latex'},
    block/.style={rectangle, rounded corners, draw, text width=15em, text centered, minimum height=3em} 
]

\node[block] (smblend) at (0,7) { Smooth blending,~ $b(x,t) \in [0,1]$\\[1ex] \hrule ~\\[-0.5ex]
active el.: where $b<1$ \\
Assume $\Delta t^{q_t+1} \lesssim h^{1+\epsilon}$ \\
Blending area width: $\mathcal{O}(1)$
};

\node[block] (feblend) at (6,7) { FE blending\\[1ex] \hrule ~\\[-0.5ex]
active el. $:$ space-time cut el. \\
Assume $\Delta t \lesssim h$ \\
Blending area width: $\mathcal{O}(h)$
};

\node[rotate=90] at (-3.1,7) {\textsc{Section 3}};
\node[rotate=90] at (-3.1,4.6) {\textsc{Section 4}};
\node[rotate=-90] at (9.1,3) {\textsc{Section 5}};

\node[block] (num_analysis_at_active) at (0,4.6) {Numerical analysis \\on active elements\\[1ex] \hrule ~\\[-0.5ex]
Def./analysis ideal mapping $\Psi^\Gamma$. \\
Def./analysis discrete $\Theta_h^\Gamma$.
};

\node[block] (feext) at (6,4.6) {Extension for FE blend. $\Psi^\Gamma,\! \Theta_h^\Gamma \! \stackrel{\mathcal{E}}{\leadsto}\! \Psi, \Theta_h$\\[1ex] \hrule ~\\[-0.5ex]
Relates low and high order spatial derivatives of $\Psi$ and $\Theta_h$
};

\node[block, 
] (int_result) at (0,1.75) {
Interpolation result \Cref{boundedness_result_final} 
};

\node[block, 
] (geom_result) at (6,1.75) {
Geometry approximation\\ \Cref{lemma_Phi_h_Id_diff_small}.
};

\draw[->] (smblend) to node[midway] (mid) {} (num_analysis_at_active);
\draw (feblend.270) to ($(feblend.270) + (0,-0.17)$) to ($($(feblend.270) + (0,-0.17)$)!(num_analysis_at_active.90)!($(feblend.270) + (2,-0.17)$)$);
\draw[->] (feblend.300) to ($(feext.70)!(feblend.300)!(feext.20)$);
\draw[->] ($(feext.-70)!(feblend.300)!(feext.-20)$) to node[pos=0.3] {
  ~~~~FE blend.} ($(geom_result.90)!(feblend.300)!(geom_result.85)$);

\coordinate (FE_blend_arrow1_midway) at ($ ($(feext.-70)!(feblend.300)!(feext.-20)$)!0.75! ($(geom_result.90)!(feblend.300)!(geom_result.85)$)$);

\draw (FE_blend_arrow1_midway) to ($(int_result.35) + (0,0.3)$);
\draw[->] ($(int_result.35) + (0,0.3)$) to (int_result.35);

\draw[->] (num_analysis_at_active.0) to (feext.180);

\draw[->] (num_analysis_at_active.230) to node[pos=0.5,fill=white, scale=1.0] {
  \begin{minipage}{2.3cm}smooth blend.~~~\\ $\Psi^\Gamma,\!\Theta_h^\Gamma\! \leadsto\! \Psi,\! \Theta_h$
  \end{minipage}} ($ (int_result.85)!(num_analysis_at_active.230)!(int_result.90)$);

\draw[->] ($ (num_analysis_at_active.230)!0.7!($ (int_result.85)!(num_analysis_at_active.230)!(int_result.90)$) $) to ($(geom_result.155)!($ (num_analysis_at_active.230)!0.7!($ (int_result.85)!(num_analysis_at_active.230)!(int_result.90)$) $)!(geom_result.165)$);
\end{tikzpicture}
\caption{Overview of the numerical analysis in this paper.}
\label{fig:num_analysis_overview}
 \vspace*{-0.5cm}
\end{figure}

In \cref{sec:levelset_geometry}, we introduce how the unfitted geometry in space-time is described by a levelset function and the assumptions considered for the analysis.
For the geometrical accuracy one essentially only requires a non-trivial mapping in the vicinity of the domain boundary
and would like the mapping to transition to the trivial mapping away from it. 
In a first phase, the mapping is constructed on a set of elements, referred to as the \emph{active elements} in the following, and then in a second phase, if necessary, extended on the remaining elements.
A rapid transition from active elements to the remainder of the mesh within a layer of only one element is computationally attractive but can lead to sharper gradients in the mapping than for a smooth transition which however requires to choose a larger set of \emph{active elements} for the first phase. Both options are introduced and discussed in \cref{sec:blendings}.
Several mappings for the \emph{active elements}, including an ideal and the realisable mapping, are constructed and analysed in
\cref{sec:mappings}.
%
%
The results in this section lay the ground for the results of the global mappings which are introduced and analysed in the subsequent \cref{sec:global_mappings} as the extension away from active elements. These show the geometrical accuracy and regularity bounds for the mappings.
Then, we elaborate on space-time interpolation based on the deformed meshes in \cref{sec:interpolation}. 
Finally, we end the paper with some numerical experiments in \cref{ref:numexp} and a conclusion in \cref{ref:concl}.

Several proofs that are tedious, but elementary or similar to other proofs presented in the main part or other literature are deferred to \cref{sec:more_proofsa,sec:more_proofsb}.

Within the paper, we will formulate several assumptions. All formulated assumptions are assumed to hold for the remainder of the paper unless explicitly stated otherwise.

\section{The space-time levelset geometry description} \label{sec:levelset_geometry}
\subsection{The space-time geometry and some basic notation}

We assume that the space-time geometry $Q$ of interest describes a time-dependent domain on a time interval $[0,T]$, $Q = \bigcup_{t \in [0,T]} \Omega(t) \times \{ t \}$.
We assume that $\Omega(t)$ is embedded into a sufficiently large spatial \hypertarget{def:tOmega}{background domain $\tOmega \subseteq \mathbb{R}^d$}, $d=1,2,3$, for all $t \in [0,T]$ so that $Q \subset \tOmega \times [0,T]$. For the spatial boundary of the space-time domain, we introduce the symbol $\partials$, so that
\hypertarget{def:partials}{$\partials Q := \bigcup_{t \in [0,T]} \partial \Omega(t) \times \{ t \}$ is the spatial boundary of the space-time domain}. \hypertarget{def:dOmega}{We further assume that there is $\dOmega > 0$ so that $\operatorname{dist}(\partial \tOmega, \partial \Omega(t)) > \dOmega$ for all $t \in [0,T]$.}
On the space-time domains, we use differential operators, e.g. the gradient operator $\nabla$ or the general derivative operator $D$, for the \emph{spatial} operations and use the notation $\partial_t$ for the time derivative. On rare occasions we will also use the full space-time derivative which we denote by \hypertarget{def:dxt}{$\Dxt$}. 
While gradients $\nabla$ of scalars are understood as column vectors, for vector-valued functions, the gradient $\nabla$ is understood as the row-wise gradient.

For the spatial identity operator we use the symbol \hypertarget{def:idx}{$\idx: \mathbb{R}^{d+1} \to \mathbb{R}^d$} so that $\idx(x,t)=x$ and for the spatial identity matrix \hypertarget{def:Idx}{$\Idx = \nabla \idx \in \mathbb{R}^{d\times d}$}. For the corresponding space-time versions we use \hypertarget{def:idxt}{$\idxt: \mathbb{R}^{d+1} \to \mathbb{R}^{d+1}$} and \hypertarget{def:Idxt}{$\Idxt \in \mathbb{R}^{(d+1) \times (d+1)}$}

\begin{figure}[ht!]
  \begin{center}
    \includegraphics[width=0.45\textwidth]{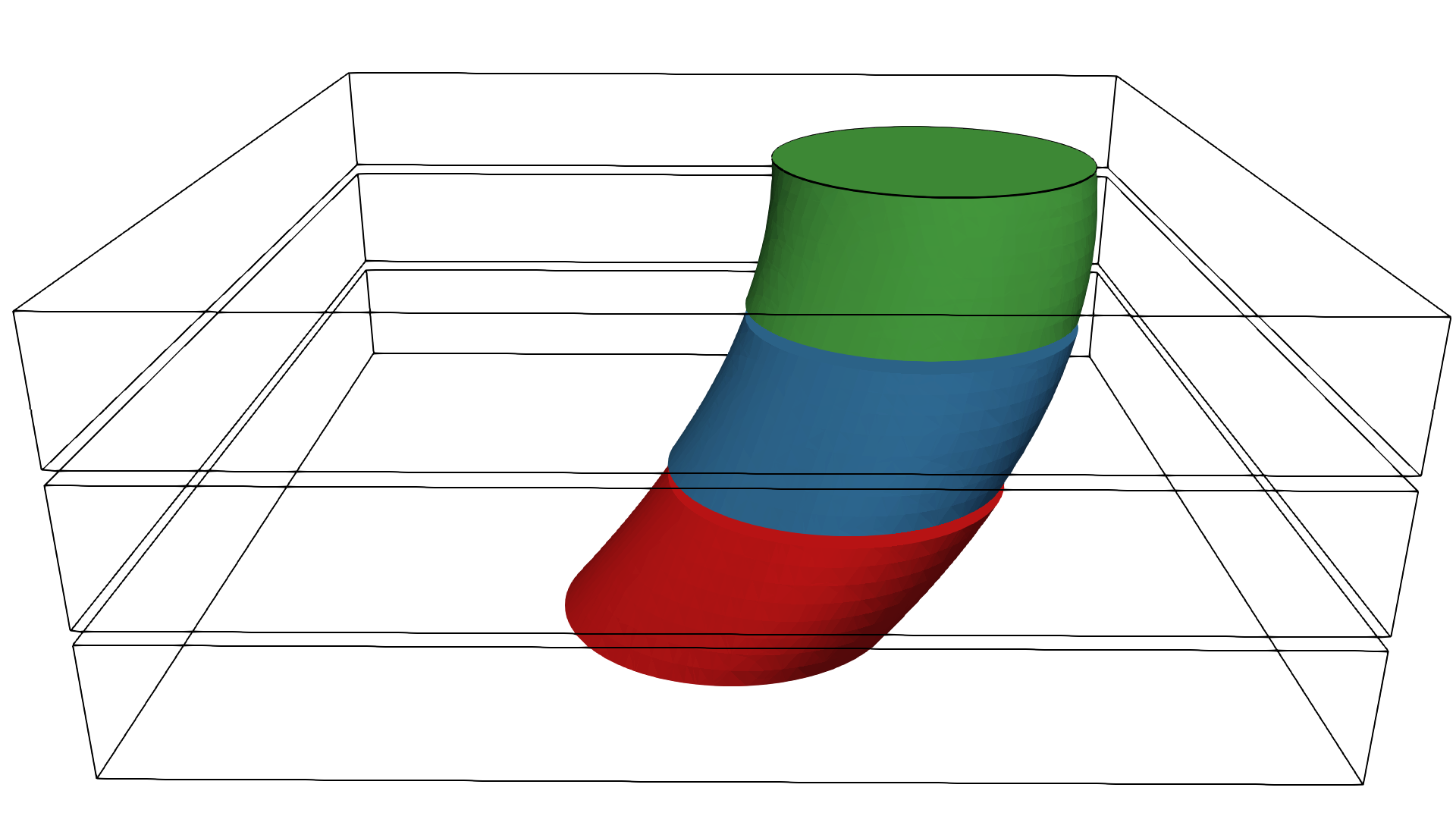}
    \includegraphics[width=0.45\textwidth]{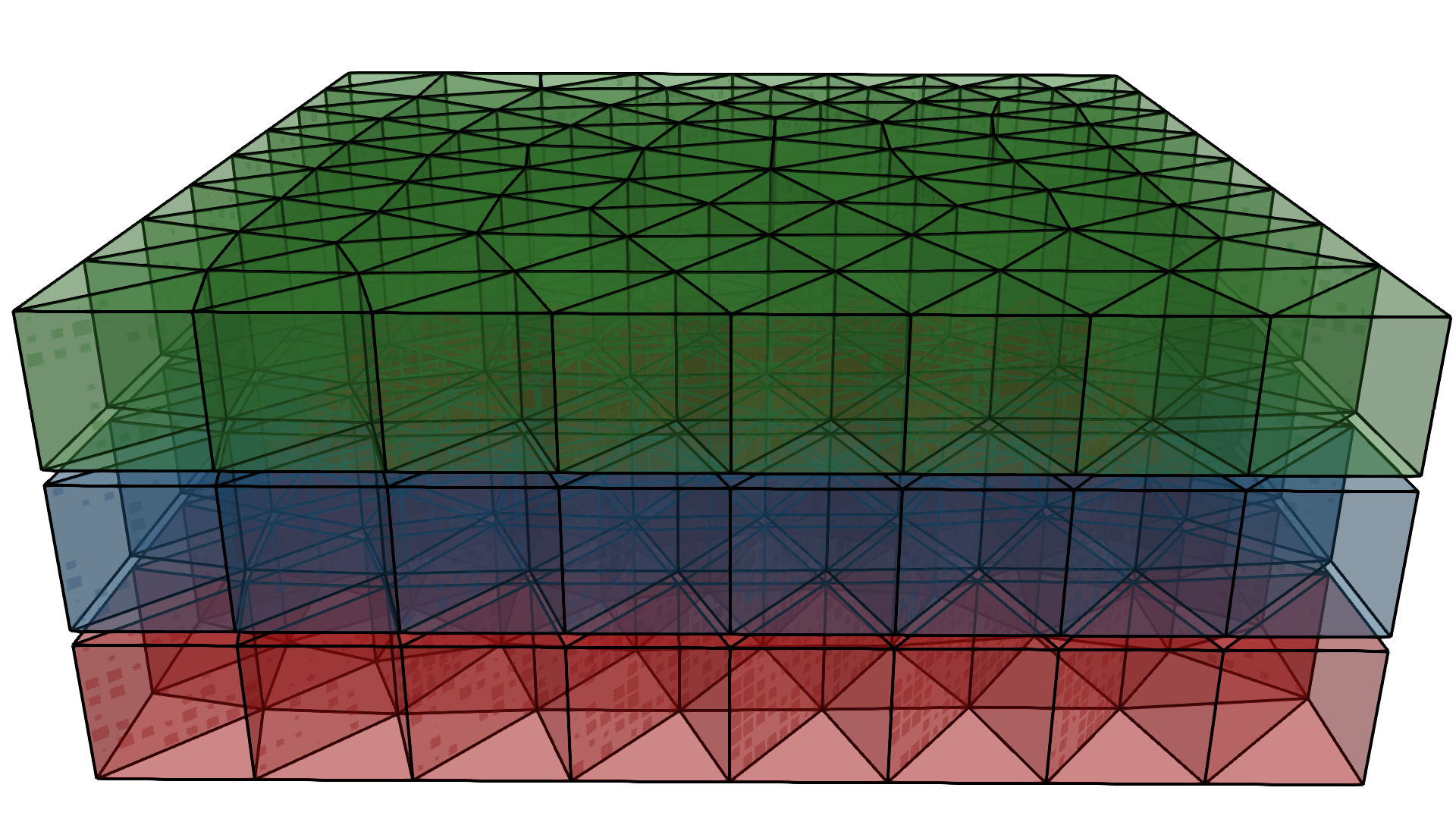}
  \end{center}\vspace*{-0.35cm}
  \caption{Illustration of the space-time geometry obtained from a circle moving in a square background domain. The left figure shows the space-time domain $Q$ embedded in a set of three time slabs, and the right figure shows the tensor-product background space-time mesh of three time slabs.} \label{fig:space_time_geometry} \vspace*{-0.25cm}
\end{figure}

Let the time interval $[0,T]$ be further subdivided into $N$ (not necessarily equally spaced) time slices \hypertarget{def:In}{$\In{n} \coloneqq (t_{n-1}, t_n]$, $n=1,\dots,N$}, $t_0 = 0, t_N = T$ and let us assume a \hypertarget{def:Th}{simplicial shape regular triangulation $\mathcal{T}_h$} of $\tOmega$ to be given that is time-independent on each time interval $\In{n}$. Then, from $\Th$ and a time interval $\In{n}$ we automatically obtain a tensor-product triangulation of the space-time domain, denoted by \hypertarget{def:ThQ}{$\ThQ$}, see \cref{fig:space_time_geometry} for an illustration.

For geometry description and approximation, time intervals can be treated independently. Hence, we will fix one time interval $\In{n}$ for our considerations in the following. However, to reflect the origin of the time interval as a time-stepping related object, we denote its length as $\Delta t = t_n - t_{n-1} = |\In{n}|$. We define the corresponding part of the space-time geometry as \hypertarget{def:Qn}{$\Qn{n} := Q \cap \tQn$} with the space-time background domain \hypertarget{def:tQn}{$\tQn \coloneqq \tOmega \times \In{n}$.}
For the spatial resolution, we define the mesh size $h = \max_{T \in \Th} h_T$ with the local mesh size $h_T = \operatorname{diam}(T),~ T \in \Th$. 

On the tensor product background space-time mesh, we introduce the following \FE~space of continuous piecewise polynomials of order $k_s$ in space and $k_t$ in time: 
\begin{equation}
  \hypertarget{def:Vh}{
 \Vh{k_s,k_t} := \Wh{k_s} \otimes \mathcal{P}^{k_t} (\In{n})} \text{ with } 
 \hypertarget{def:Wh}{\Wh{k_s} \coloneqq \{ v \in {H}^1(\tOmega) \, | \, v|_T \in \mathcal{P}^{k_s}(T) \forall T \in \Th \}}.
\end{equation}
where $H^1$ denotes the usual (spatial) Sobolev space and $\mathcal{P}^{k_t}(\In{n})=\mathcal{P}^{k_t}(\In{n};\mathbb{R})$ is the space of polynomials up to degree $k_t$ on $\In{n}$.
Analogously we define $\Vh{k_s,k_t}(\mathcal{Q}_h^{\text{sub}})$ for a tensor-product sub-mesh $\mathcal{Q}_h^{\text{sub}}$ built from the spatial sub-mesh $\mathcal{T}_h^{\text{sub}}$ of $\Omega^{\text{sub}}$ and $\In{n}$ by replacing $\tOmega$ by the corresponding spatial domain $\Omega^{\text{sub}}$.

In the following we use $\lesssim$ and $\gtrsim$ to denote inequalities up to constants that are independent of the mesh size $h$, the time step size $\Delta t$ and the cut configuration, i.e. the position of $\Qn{n}$ in the mesh. We use $\simeq$ to imply $\lesssim$ and $\gtrsim$ simultaneously.

For the forthcoming analysis we are interested in the asymptotic behavior of the error in the space-time geometry approximation. We will therefore assume at several places that the mesh size $h$ and the time step size $\Delta t$ are \emph{sufficiently small} without an explicit specification of the respective constants.

\subsection{The space-time levelset description}
Crucial for the remainder of this work is the way that we assume the domain to be described by: Through a levelset function. 
\hypertarget{def:phi}{The space-time levelset function $\phix$ is defined on $\tQn$ and its zero levelset describes the space-time geometry. We will require some smoothness of $\phix$ in the vicinity of its spatial boundary.}
For the neighborhood of the spatial boundary we assume that there is $\dU > 0$ independent of $t$ describing the \hypertarget{def:dU}{spatial neighborhood} \hypertarget{def:Ust}{$\Ust{t} := \{ x \in \tOmega \mid \operatorname{dist}(x, \partial \Omega(t)) < \dU \}$} for $t \in \In{n}$ and that $\phix$ is smooth in $\Ust{t}$ for all $t \in \In{n}$. A corresponding space-time neighborhood is obtained as \hypertarget{def:U}{$U := \bigcup_{t \in \In{n}} \Ust{t} \times \{t\}$}. 

The upcoming discretisation will be controlled by two order parameters, $q_s$ for the spatial and $q_t$ for temporal order.
We assume that $\phix$ is sufficiently smooth in $\U$ and formulate the following assumption:
\begin{assumption} \label{ass:lsetreg}
 There exists a function $\phix \in C^0(\tQn; \mathbb{R})$ s.t.
 \begin{equation}
 \Omega(t) = \{ x \in \tOmega \mid \phix(x,t) < 0 \}, \quad \Qn{n} = \{ (x,t) \in \tQn \, | \, \phix(x,t) < 0 \},
 \end{equation}
 which satisfies the following regularity assumptions on a neighbourhood $\U$ of $\partials \Qn{n}$ with a fixed $\dU > 0$:
 \begin{equation} \label{eq:cnphi} \addtocounter{equation}{1}\tag{\theequation-\texttt{$\phi$\footnotesize{bnd}}}
  \hypertarget{def:cnphi}{\cnphi{}} \leq \|  \nabla \phix \|_{\infty, \U} \! \leq \hypertarget{def:Cnphi}{\Cnphi{}}, \qquad \| \partial_t \phix \|_{\infty, \U}, \| D^2 \phix \|_{\infty, \U},  \| \partial_t \nabla \phix \|_{\infty,\U} \lesssim 1.
 \end{equation}
 Moreover, the levelset function at each time $t \in \In{n}$ fullfils a weak signed distance property: 
 \begin{equation} \addtocounter{equation}{1}\tag{\theequation-\texttt{$\phi$\footnotesize{wsd}}}
  | \phix(x+ \epsilon \nabla \phix(x),t) - \phix(x+ \tilde{\epsilon} \nabla \phix(x),t)| \simeq | \epsilon - \tilde{\epsilon} | \quad \textnormal{ for } x \in \U(t). \label{signed_dist_like_prop}
 \end{equation}
\end{assumption}

\begin{definition}
  We define the regularity index \hypertarget{def:lphi}{$\lphi$} of $\phix$ as the largest integer value (or infinity) for which $\phix \in C^0(\tOmega)\cap C^2(\U) \cap \Sinf{\lphi}{\U}$.
\end{definition}
We note that the simultaneous space-time regularity assumption could be weakened to a separate regularity assumption in space and time. However, this would complicate the analysis further and we therefore opt for the stronger variant here.
We make the following assumption on the regularity index $\lphi$ of $\phix$:
\begin{assumption} \label{ass:lphi:1}
  $\lphi \geq \max\{q_s,q_t\}+2$.
\end{assumption}

Note that the smoothness of $\phix$ in $\U$ (together with \eqref{eq:cnphi}) implies a corresponding smoothness of $\partials \Qn{n} \in C^{\lphi}$ and $\partial \Omega(t),~t \in \In{n}$, respectively. Further, we note that with \cref{ass:lsetreg} all (space-time) elements that are intersected by $\partials \Qn{n}$ and a finite number of neighbor elements can be guaranteed to lie within $\U$ for sufficiently small $h$ and $\Delta t$. Let \hypertarget{def:ThU}{$\ThU$ be the submesh of elements that are contained in $\U$ for all times $t \in \In{n}$, i.e. $T \times \In{n} \subset \U$ for all $T \in \ThU$.} \hypertarget{def:Ubar}{We denote the spatial domain corresponding to $\ThU$ as $\Ubar$} and \hypertarget{def:Qhu}{the set of space-time elements as $\QhU = \{T \times \In{n} \mid T \in \ThU \}$} with corresponding space-time domain \hypertarget{def:QU}{$\QU$}.



\subsection{Discrete levelset functions}
In practice, we will typically not have access to the exact levelset function $\phix$ but only to a discrete approximation. We assume that we are given a discrete levelset approximation \hypertarget{def:phih}{$\phih \in \Vh{q_s,q_t}$ of $\phix$} with $q_s, q_t \in \mathbb{N}$ and that the following approximation assumption holds.
\begin{assumption} \label{diff_phi_phihn}
  For $(m_s, m_t) \in \{0,\dots,q_s+1\} \times \{0,1\} \cup \{0\} \times \{0,\dots,q_t+1\}$ 
    \begin{align}
    \| D^{m_s} \partial_t^{m_t} (\phix - \phih) \|_{\infty, \QhU} & \lesssim h^{q_s +1 - m_s} + \Delta t^{q_t + 1 - m_t}. \label{diff_phi_phihn_spaceD} 
    \addtocounter{equation}{1}\tag{\theequation-\texttt{$\phi_h$\footnotesize{acc}}}    
 \end{align}
\end{assumption}
Note that with the regularity assumption in \cref{ass:lphi:1}, we have $m_s\leq q_s+1\leq \lphi$ and $m_t\leq q_t+1\leq \lphi$ in \cref{diff_phi_phihn} so that it would hold for $\phih$ being a suitable 
interpolation operator. 


In the upcoming analysis we will make use of the tensor product structure of the \FE~space $\Vh{q_s,q_t}$ and we assume a \emph{nodal} basis \hypertarget{def:elli}{ $\{ \elli{i} \}$ of $\mathcal{P}^{q_t}(\In{n})$ }
 with nodes \hypertarget{def:ti}{ 
  $\{ \ti{i} \}$ 
  } so that $\elli{j}(\ti{i}) = \delta_{ij}$.
We further define with \hypertarget{def:It}{$\It{q_t}$ the temporal (nodal) interpolation operator applied in space-time, i.e. $\It{q_t}: v(x,t) \mapsto \sum_{i=0}^{q_t} \ell_i(t) v(x,\ti{i})$.}
We can hence develop $\phih \in \Vh{q_s,q_t}$ in terms of spatial functions \hypertarget{def:phihi}{$\phihi{i}(x,t) = \phih(x,t_i) \in \Wh{q_s}$} and have
\begin{equation}
  \phih (x,t) = \sum_{i=0}^{q_t} \elli{i}(t) \phihi{i}(x), 
  \label{eq:decomp_phi_h}
\end{equation}
Note that we specifically opted for a nodal basis in time here. More general options would also be possible but would complicate the notation and analysis slightly.
For notational consistency we also introduce \hypertarget{def:phii}{$\phii{i}(x) = \phix(x,\ti{i})$}.

Next, we will assume the existence of two semi-discrete versions of $\phix$ which have tensor product form and are only discrete in time and continuous in space or vice versa.
These semi-discrete versions will be crucial for the construction of semi-discrete mappings in \cref{sec:discrete-in-time-mapping,sec:discrete-in-space-mapping} that are needed to obtain robustness in the error bounds w.r.t. anisotropy in space and time.
\begin{assumption} \label{diff_phidt_phihn}
  We assume that there is \hypertarget{def:phidt}{$\phidt \in \mathcal{P}^{q_t}(I_n;C^0(\tOmega)\cap C^{q_s+1}(\bar U))$} 
with \hypertarget{def:phidti}{$\phidti{i} \coloneqq \phidt(\cdot,\ti{i})$} s.t. for $m_s = 0, \dots, q_s+1$ there holds
\begin{subequations}
  \begin{align}
   \|D^{m_s} (\phidti{i} - \phihi{i})\|_{\infty, \ThU} & \lesssim h^{q_s +1 - m_s}, \text{ which implies } \label{diff_phid_phih:i}
   \addtocounter{equation}{1}\tag{\theequation-\texttt{$\phi_{\Delta t\mid h,i}$}}    
   \\ \|D^{m_s} (\phidt - \phih) \|_{\infty, \QhU} & \lesssim h^{q_s +1 - m_s}, ~\text{ and } \label{diff_phid_phih} 
   \addtocounter{equation}{1}\tag{\theequation-\texttt{$\phi_{\Delta t\mid h}$}}    
   \\
    \|D^{m_s} (\phidt - \phix) \|_{\infty, \QhU} & \lesssim h^{q_s +1 - m_s} + \Delta t^{q_t +1 }. \label{diff_phid_phi}
    \addtocounter{equation}{1}\tag{\theequation-\texttt{$\phi_{\Delta t}$\footnotesize{acc}}}    
   \end{align}
  \end{subequations}
  \end{assumption}

  \begin{assumption} \label{diff_phiH_phihn}
    We assume that there is \hypertarget{def:phiH}{$\phiH \in C^{q_t+1}(I_n;W_h^{q_s})$} s.t. for $m_t = 0, \dots, q_t+1$ and $m_s \in \{0,1\}$ there holds
  \begin{subequations}
    \begin{align}
     \|\partial_t^{m_t} (\phiH - \phih) \|_{\infty, \QhU} & \lesssim \Delta t^{q_t +1 - m_t}, ~\text{ and } \label{diff_phiH_phih} 
     \addtocounter{equation}{1}
     \tag{\theequation-\texttt{$\phi_{H\mid h}$}}    
     \\
      \|D^{m_s} \partial_t^{m_t} (\phiH - \phix) \|_{\infty, \QhU} & \lesssim h^{q_s +1 - m_s} + \Delta t^{q_t +1 - m_t}. \label{diff_phiH_phi}
      \addtocounter{equation}{1}\tag{\theequation-\texttt{$\phi_{H}$\footnotesize{acc}}}    
     \end{align}
    \end{subequations}
    In contrast to the previous assumptions we will not automatically assume it to hold, but will refer to it explicitly.
    \end{assumption}
  
Note that at first glance, $\phiH = \phih$ seems a straightforward choice rendering the introduction of \cref{diff_phiH_phihn} obsolete. However, the assumed accuracy in \eqref{diff_phiH_phi} is stronger than the one in \eqref{diff_phi_phihn_spaceD} as the bound for $m_s=1$ and $m_t > 1$ is not included in \cref{diff_phidt_phihn}. Alternatively, \cref{diff_phiH_phihn} could be circumvented by strengthening \cref{diff_phidt_phihn}. This would however imply to increase the assumptions of the realized approximation $\phih$ while the stronger assumption in \cref{diff_phiH_phihn} is only an assumption on a theoretical proxy $\phiH$.

  \begin{remark}
    The previous assumptions are reasonable if $\phih$ stems from a discretisation with separable temporal and spatial errors. Let us consider the simplest case where $\phih$ is obtained from a tensor-product space-time interpolation, $\phih = \It{} I^s \phix$ with $I^s$ a corresponding spatial interpolation operator. Then, we can set $\phiH = I^s \phix$ and $\phidt = \It{} \phix$ and obtain the desired properties from continuity and the respective approximation properties of $\It{}$ and $I^s$, respectively. 
    \cref{diff_phiH_phihn} will explicitly be mentioned when used and corresponding weaker statements will be given in parallel for the case that the assumptions are not available.
  \end{remark}

  We note that for sufficiently small $h$ and $\Delta t$ the weak signed distance property from \eqref{signed_dist_like_prop} of $\phi$ carries over to $\bar\phi \in \{ \phidt, \phiH, \phih \}$:
  \begin{equation} \addtocounter{equation}{1}\tag{\theequation-\texttt{$\phi_{h}$\footnotesize{wsd}}}
    | \bar \phi(x+ \epsilon \nabla \bar \phi(x),t) - \bar \phi(x+ \tilde{\epsilon} \nabla \bar \phi(x),t)| \simeq | \epsilon - \tilde{\epsilon} | \quad \textnormal{ for } x \in \U(t). \label{signed_dist_like_prop_dt}
   \end{equation}


\subsection{The reference geometry based on \texorpdfstring{$\phi^{\text{lin}}$}{linear levelset fct.}}

Dealing with piecewise higher-order polynomial levelset functions is difficult in the context of robust numerical integration, cf. \cref{ssec:lit}. This is why we introduce a reference geometry based on a linear approximation of $\phix$ in space in the spirit of the spatial isoparametric FEM, cf. \cref{ssec:spatial}. 
For the function $\phih \in \Vh{q_s,q_t}$, we define the spatial piecewise linear nodal interpolation:
\begin{equation}
  \hypertarget{def:philin}{\philin(x,t)}
   \coloneqq (I_h^1 \phih(\cdot, t))(x) \stackrel{\cref{eq:decomp_phi_h}}{=}\sum_{i=0}^{q_t} \elli{i}(t) ( I_h^1 \phihi{i}(x)) \eqqcolon \hypertarget{def:philini}{\sum_{i=0}^{q_t} \elli{i}(t) \philini{i}(x)},
\end{equation}
where $I_h^1: C^0(\Omega) \to W_h^1$ denotes the spatial nodal interpolation operator. This implies
\begin{corollary} \label{diff_phihn_philin}
It holds for $m=0,1$
\begin{equation}
 \| D^{m} (\phih - \philin) \|_{\infty, \QhU} \lesssim h^{2 - m}.
 \addtocounter{equation}{1}\tag{\theequation-\texttt{$\phi_{\text{lin}|h}$\footnotesize{acc}}} \label{eq:diff_phihn_philin}
\end{equation}
\end{corollary}
Also, the time derivative of $\philin$ is (at each time) the spatial interpolation of $\partial_t \phih$:
\begin{equation}
 (I_h^1 \partial_t \phih(\cdot, t))(x) = \sum_{i=0}^{q_t} (\partial_t \elli{i}(t)) ( I_h^1 \phihi{i}(x)) = \sum_{i=0}^{q_t} (\partial_t \elli{i}(t)) \philini{i}(x) = \partial_t \philin(x,t).
\end{equation}
The same holds for the time derivatives of the gradients of $\phih$ and $\philin$. These results imply together with \cref{diff_phi_phihn} the following result.
\begin{corollary} \label{diff_phihn_philin_dt}
It holds for $\bar \phi \in \{\phih, \phidt\}$, $m_s=0,1$ and $m_t=0,\dots, q_t+1$,
  \begin{align}
 \| D^{m_s} \partial_t^{m_t} (\bar \phi - \philin) \|_{\infty, \QhU} &\lesssim h^{2-m_s} , 
 \label{eq:phidt-philin} 
 \addtocounter{equation}{1}\tag{\theequation a-\texttt{$\phi^{\text{lin}}$\footnotesize{acc}}}    
 \\
 \| D^{m_s} \partial_t^{m_t} (\phix - \philin) \|_{\infty, \QhU} &\lesssim h^{2-m_s} + \Delta t^{q_t +1-m_t}.
 \addtocounter{equation}{0}\tag{\theequation b-\texttt{$\phi^{\text{lin}}$\footnotesize{acc}}}    
 \label{diff_grad_phi_philin_dt} 
\end{align}
\end{corollary}

The function $\philin$ induces the linear (in space) domains
\begin{equation}
\hypertarget{def:Omlin}{\Omlin{t}} \coloneqq \{ x \in \tOmega \, | \, \philin(x,t) < 0 \}, \textnormal{ and } \hypertarget{def:Qlinn}{\Qlin} \coloneqq {\textstyle \bigcup_{t \in \In{n}} } \Omlin{t} \times \{ t \}.
\end{equation}

We are now in a similar setting as for the merely spatial unfitted isoparametric FEM approach. We have a reference configuration with $\Qlin$ of lower order accuracy (in space) described by $\philin$ 
and a higher-order accurate description with $\phih$ which is however not directly suitable for numerical integration. 

The geometry approximation obtained by the parametric approach in this paper, as already in the spatial case in \cite{L_CMAME_2016,LR_IMAJNA_2018}, relies on capturing the cut topologies with $\philin$ and improving its geometrical accuracy with a carefully constructed mapping. However, capturing the cut topologies with $\philin$ requires resolution (in the mesh size) depending on the curvature of domain boundary and we hence assume in the following that $h$ is sufficiently small in that sense.

\section{Blending methods} \label{sec:blendings}
%
For the construction of a mapping from $\Qlin$ that approximates $\Qn{n}$ with higher-order accuracy we need to find a mapping $\ThehG$ that maps the spatial boundary $\partials \Qlin$ close to $\partials \Qn{n}$, and relatedly a mapping $\Psi^\Gamma$, mapping $\partials \Qlin$ exactly onto $\partials \Qn{n}$. Eventually, however, these mappings should be defined on the whole space-time slab $\tQn$. 
We will consider two different approaches for the construction and extension of the mapping for which we distinguish by the choice of a triple $(\ThG,\OG,\b)$, \hypertarget{def:ThG}{where $\ThG$ is a submesh of $\ThU$ of \emph{active elements}}, \hypertarget{def:OG}{$\OG$ is the corresponding domain} and \hypertarget{def:b}{$\b:\tQn \to [0,1]$} is a blending function.
\hypertarget{def:ThGQ}{Corresponding to $\ThG$ we further introduce the space-time submesh of active elements $\ThGQ \coloneqq \bigcup_{T \in \ThG} T \times \In{n}$ } and \hypertarget{def:QGn}{$\QGn \coloneqq \OG \times \In{n}$ as the corresponding space-time domain.}
For notational convenience we define according to $\b(x,t)$ the time restrictions to $\ti{i}$ by \hypertarget{def:bi}{$\bi{i}(x) \coloneqq \b(x, \ti{i})$ for $i=0,\dots,q_t$}.
\begin{enumerate}
  \item In a first approach -- the \textit{\FE~blending} -- we consider all space-time cut elements as the active elements $\ThG$ and construct in a first step $\Psi^\Gamma$ and $\ThehG$ on those. Afterwards, in a second step, we will use a FE~blending which is a continuous extension which vanishes within an $\mathcal{O}(h)$ layer of one element from the cut elements in order to obtain $\Psi = \mathcal{E} \Psi^\Gamma$, $\Theta_h = \mathcal{E} \ThehG$.
  \item In the second approach -- the \textit{smooth blending} --  we use a smooth function in the construction of the mapping to transition between the region where the mapping takes full effect to regions where the mapping is the identity. To accomodate for this, a wider stripe of elements -- depending on the width of a scalar blending function -- than merely the cut elements is chosen for the active elements $\ThG$ and the mappings $\Psi^\Gamma$ and $\ThehG$ are constructed in the same way as in the first step of the \FE~blending. Then, the second step of extension becomes trivial as outside of $\ThG$ the mappings are set to identity.
\end{enumerate}
We will unify the analysis of both cases by introducing two steps: The construction of a mapping on a set of \emph{active elements} and a \FE~extension. In the case of the \FE~blending the active set of elements is merely the set of cut space-time elements while in the second case it is the set of space-time elements where the scalar blending function is smaller than one.
Note that in both cases for sufficiently small $h$ the set of active elements lies within $\ThU$ and hence within $\Ubar$.

After defining a mapping on the active set of elements the \FE~extension is applied. However, it becomes trivial in the case of the smooth blending, as the mapping will be the trivial mapping on the boundary of the set of active elements already.

\begin{remark}[Problems purely posed on space-time cut elements]
  If problems are to be solved only on the space-time cut elements as it is the case for surface PDEs when treated with TraceFEM approaches, cf. e.g. \cite{sassreusken23, reusken2024analysis}, problems with blendings vanish and the \FE~blending is a very natural (and simple) choice. Furthermore, the additional more restrictive assumptions that we will introduce for the \FE~blending would not be needed for proper mappings only on the space-time cut elements.  
\end{remark}

\subsection{Finite element blending}
For the \emph{\FE~blending} only cut space-time elements, respectively their spatial counterparts, become \emph{active}.
For the time slice under consideration, $\In{n}$, the set of \emph{active} elements is 
\begin{equation}
\hypertarget{def:Thb1}{\ThbOne \coloneqq \{ T \in \Th \, | \, \exists t \in \In{n} ~ \partial \Omlin{t} \cap T \neq \varnothing \}}, \quad  \hypertarget{def:OgbOne}{\OGbOne = \bigcup \ThbOne.}
\end{equation}
We give an illustration for this active domain for two time step choices in \cref{fig:feblendregions}. We also show the domain of adjacent elements, which is the domain where the FE blending will decay. (More details about this will be given in \cref{sec:global_mappings}.)

\begin{figure}
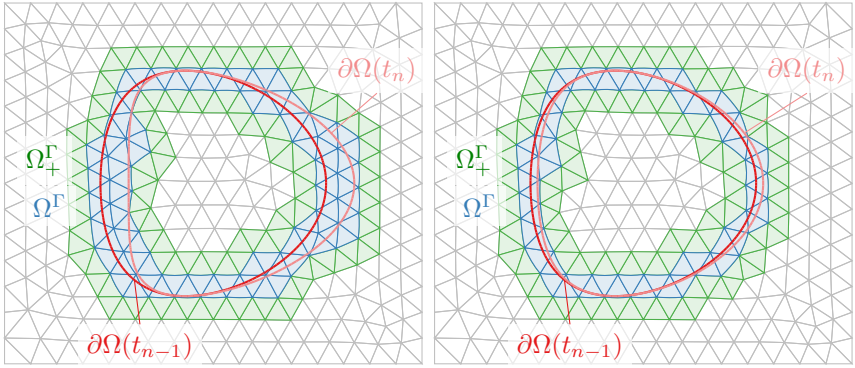

 \begin{center}
  \begin{tikzpicture}[ xscale=1.5, yscale=1.5]
   \begin{scope}[meshtrig/.style={fill=Set1-B!15!white, draw=Set1-B},
     Omplus/.style={fill=Set1-C!15!white, draw=Set1-C},
     backgrtrig/.style={draw=gray!50!white, thin}, Gamma_h/.style={draw=Set1-A, thick},
     Gamma_lin/.style={draw=none}]
     \input{mesh_examples/tikz_kite_mesh_feblend_cut_region_largedt0}
     \node[Set1-B,fill=white,opacity=0.7] at (-1.2,-0.2) {\phantom{$\Omega^\Gamma$}};
     \node[Set1-B] at (-1.2,-0.2) {$\Omega^\Gamma$};
     \node[Set1-E,fill=white,opacity=0.7] at (-1.25,0.2) {\phantom{$\OGp$}};
     \node[Set1-E] at (-1.25,0.2) {$\OGp$};

     \node[Set1-A,fill=white,opacity=0.7] (Omtnm1) at (-0.4,-1.45) {\phantom{$\partial \Omega(t_{n-1})$}};
     \node[Set1-A] (Omtnm1) at (-0.4,-1.45) {$\partial \Omega(t_{n-1})$};
     \node[Set1-A!50!white,fill=white,opacity=0.7] (Omtn) at (1.7, 1) {\phantom{$\partial \Omega(t_{n})$}};
     \node[Set1-A!50!white] (Omtn) at (1.7, 1) {$\partial \Omega(t_{n})$};

     \draw[Set1-A] (Omtnm1.north) -- (-0.45,-0.85);
     \draw[Set1-A!50!white] (Omtn.south) -- (1.3,0.45);
   \end{scope}
   \begin{scope}[meshtrig/.style={draw=none},
     Omplus/.style={draw=none},
     backgrtrig/.style={draw=none},
     Gamma_h/.style={draw=Set1-A!50!white, thick, opacity=1},
     Gamma_lin/.style={draw=none}]
     \input{mesh_examples/tikz_kite_mesh_feblend_cut_region_largedt1}
   \end{scope}
  \end{tikzpicture}
  \begin{tikzpicture}[ xscale=1.5, yscale=1.5]
   \begin{scope}[meshtrig/.style={fill=Set1-B!15!white, draw=Set1-B},
     Omplus/.style={fill=Set1-C!15!white, draw=Set1-C},
     backgrtrig/.style={draw=gray!50!white, thin}, Gamma_h/.style={draw=Set1-A, thick},
     Gamma_lin/.style={draw=none}]
     \input{mesh_examples/tikz_kite_mesh_feblend_cut_region_smalldt0}
     \node[Set1-B,fill=white,opacity=0.7] at (-1.2,-0.2) {\phantom{$\Omega^\Gamma$}};
     \node[Set1-B] at (-1.2,-0.2) {$\Omega^\Gamma$};
     \node[Set1-E,fill=white,opacity=0.7] at (-1.25,0.2) {\phantom{$\OGp$}};
     \node[Set1-E] at (-1.25,0.2) {$\OGp$};

     \node[Set1-A,fill=white,opacity=0.7] (Omtnm1) at (-0.4,-1.45) {\phantom{$\partial \Omega(t_{n-1})$}};
     \node[Set1-A] (Omtnm1) at (-0.4,-1.45) {$\partial \Omega(t_{n-1})$};
     \node[Set1-A!50!white,fill=white,opacity=0.7] (Omtn) at (1.7, 1) {\phantom{$\partial \Omega(t_{n})$}};
     \node[Set1-A!50!white] (Omtn) at (1.7, 1) {$\partial \Omega(t_{n})$};

     \draw[Set1-A] (Omtnm1.north) -- (-0.45,-0.85);
     \draw[Set1-A!50!white] (Omtn.south) -- (1.15,0.45);
   \end{scope}
   \begin{scope}[meshtrig/.style={draw=none},
     Omplus/.style={draw=none},
     backgrtrig/.style={draw=none},
     Gamma_h/.style={draw=Set1-A!50!white, thick, opacity=1},
     Gamma_lin/.style={draw=none}]
     \input{mesh_examples/tikz_kite_mesh_feblend_cut_region_smalldt1}
   \end{scope}
  \end{tikzpicture}
 \end{center} \vspace*{-0.25cm}
 \caption{Illustration of the discrete regions involved in the FE blending construction. We sketch the situation of a large (left) and small (right) time step. The solid and transparent red lines indicate the discrete interface at the beginning and end of the time step, respectively. The blue elements indicate the active mesh as it relates to the blending construction, $\protect\ThbOne$. The green elements indicate the additional adjacent elements where the FE blending will operate.}
 \label{fig:feblendregions} \vspace*{-0.25cm}
\end{figure}

For the blending function we set $\hypertarget{def:b1}{\bOne(x,t) \equiv 0}$ and obtain the triple
$(\ThG,\OG,\b) = (\ThbOne,\OGbOne,\bOne)$.
This blending approach has also been discussed in \cite{HLP2022}. It however comes at the price of additional assumptions on the regularity of the levelset function and a CFL-type time step restriction which we formulate in the next two assumptions
on the regularity of the levelset function and the time step size:
\begin{assumption}\label{ass:reg_feblend}
  We assume $\lphi \geq q_s + q_t + 2$ if the \FE~ blending is applied.
\end{assumption}
\begin{assumption}\label{ass:dtlesssimh_feblend}
  We assume $\Delta t \lesssim h$ if the \FE~ blending is applied.
\end{assumption}
An alternative that alleviates these restrictions is the \emph{smooth blending} discussed next.

\subsection{Smooth blending}
For the second option of the blending, the \emph{smooth blending}, we let the corresponding domains follow from a blending function $\bTwo\colon \tQn \to [0,1]$, which has the properties summarised in the following assumption.
\begin{assumption} \label{ass:blending}
\hypertarget{def:b2}{
 In the case of the \emph{smooth blending} we assume that there exists a blending function $\bTwo \in C^{\lphi}(\tQn;[0,1])$ with width $w_b \in \mathbb{R}, 0<w_b<\min\{\dOmega,\dU\}$ such that} 
 \begin{enumerate}
  \item On all cut elements, $\bTwo$ vanishes, i.e. $\bTwo(x,t)=0$ for all $x \in \OGbOne$, $t\in \In{n}$. 
  \item Sufficiently far away from the interface, $\bTwo$ takes the value 1: For all $t \in \In{n}$, for all $x \in \tOmega$ such that $\operatorname{dist}(x, \partial \Omlin{t})> w_b$, there is $\bTwo(x,t) = 1$.
  \item $\Vert \partial_t^{m_t} D^{m_s} \bTwo \Vert_{\infty,\tQn} \leq C_b$ with $C_b$ independent of $h$ and $\Delta t$,~$m_s + m_t \leq \lphi$.
 \end{enumerate}
\end{assumption}
This allows us to define the following counterparts of $ \ThbOne$ and $\OGbOne$:
\begin{equation}
 \hypertarget{def:Thb2}{\ThbTwo \coloneqq  \{ T \in \Th \, | \, \exists t \in \In{n} ~ \exists x \in T ~ \bTwo(x,t) < 1 \}}, \quad \hypertarget{def:OGbTwo}{\OGbTwo \coloneqq \bigcup \ThbTwo.}
\end{equation}
For the \emph{smooth blending} we obtain the triple 
$(\ThG,\OG,\b) = (\ThbTwo,\OGbTwo,\bTwo)$.

We illustrate the geometric domains for an example case in \Cref{fig:discr_dom_smooth_blend}.
\begin{figure}
 \begin{center}
   \begin{tikzpicture}[ xscale=1.5, yscale=1.5]
   \begin{scope}[meshtrig/.style={fill=Set1-B!15!white, draw=Set1-B},
     Omplus/.style={fill=Set1-C!15!white, draw=Set1-C},
     backgrtrig/.style={draw=gray!50!white, thin}, Gamma_h/.style={draw=none, thick},
     Gamma_lin/.style={draw=none}]
     \input{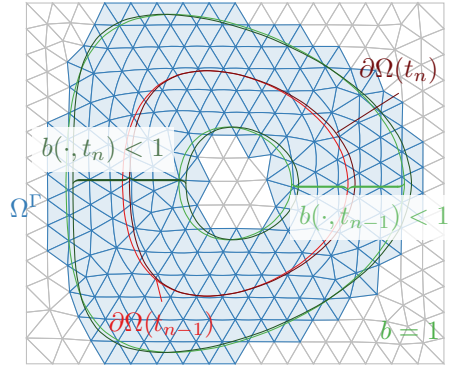}

       \draw[Set1-A, domain=-3.141:3.141,smooth,variable=\t]
plot ({ cos(\t r)+(1 - sin(\t r)^2)*0.25 },{sin(\t r)});

  \draw[Set1-C, domain=-3.141:3.141, smooth,variable=\t]
 plot ({ 1.5*cos(\t r)+(1 - (1.5*sin(\t r))^2)*0.25 },{ 1.5*sin(\t r)});
%

  \draw[Set1-C, domain=-3.141:3.141, smooth,variable=\t]
plot ({ 0.5*cos(\t r)+(1 - (0.5*sin(\t r))^2)*0.25 },{ 0.5*sin(\t r)});
%

       \draw[Set1-A!50!black, domain=-3.141:3.141,smooth,variable=\t]
plot ({ cos(\t r)+(1 - sin(\t r)^2)*0.3125 },{sin(\t r)});

  \draw[Set1-C!50!black, domain=-3.141:3.141, smooth,variable=\t]
plot ({ 1.5*cos(\t r)+(1 - (1.5*sin(\t r))^2)*0.3125 },{ 1.5*sin(\t r)});
%

  \draw[Set1-C!50!black, domain=-3.141:3.141, smooth,variable=\t]
plot ({ 0.5*cos(\t r)+(1 - (0.5*sin(\t r))^2)*0.3125 },{ 0.5*sin(\t r)});


      \node[Set1-B] at (-1.6,-0.2) {$\Omega^\Gamma$};
%
      \node[Set1-A] (Omtnm1) at (-0.4,-1.25) {$\partial \Omega(t_{n-1})$};
      \node[Set1-A!50!black] (Omtn) at (1.7, 1) {$\partial \Omega(t_{n})$};

      \draw[Set1-A] (Omtnm1.north) -- (-0.45,-0.85);
      \draw[Set1-A!50!black] (Omtn.south) -- (1.15,0.45);

      \draw[decorate, thick, pen colour={Set1-C},
    decoration = {calligraphic brace}] (1.75,0) -- (0.75,0) node [midway, xshift=0.35cm, yshift=-0.45cm, Set1-C, fill=white, fill opacity=0.8] {$b(\cdot, t_{n-1}) < 1$};

      \draw[decorate, thick, pen colour={Set1-C!50!black},
    decoration = {calligraphic brace}] (-1.1875,0) -- (-0.1875,0) node [midway, xshift=-0.35cm, yshift=0.45cm, Set1-C!50!black, fill=white, fill opacity=0.8] {$b(\cdot, t_{n}) < 1$};

    \node[Set1-C] at (1.8,-1.3) {$b=1$};
   \end{scope}
  \end{tikzpicture}
 \end{center} \vspace*{-0.25cm}
 \caption{Illustration of the discrete domain induced by a smooth blending function $b$. The surface is shown in red for the beginning and the end of a time step $I_n$. At the surface and the neighboring elements, $b=0$. Away from the surface, the value of $b$ increases until the green lines, where $b=1$ (dashed lines for $b=1$ at the end of the time step). The resulting region $\Omega^\Gamma$, constructed in line with \cref{ass:blending}, is shown in blue. Compare with \cref{fig:blending_def_examples} for plots of example functions $b$.} \vspace*{-0.25cm}
 \label{fig:discr_dom_smooth_blend}
\end{figure}

Compared to \cref{ass:dtlesssimh_feblend} we only require a much milder condition on the smallness of the time-step to ensure the well-posedness of the upcoming constructions:
\begin{assumption} \label{ass:timefiner}
In case of the smooth blending, we assume there is $\epsilon > 0$, s.t.
 \begin{equation}
 \Delta t^{q_t + 1} \leq C_{\Delta t - h} h^{1+ \epsilon}. \label{delta_t_smallness}
 \addtocounter{equation}{1}\tag{\theequation-\texttt{$\Delta t$\footnotesize{bnd}}}
\end{equation}
\end{assumption}
Note that this assumption allows for anisotropy in space and time resolution to a large extent.
We refer to \cref{fig:timefiner} for an illustration of the restrictions.
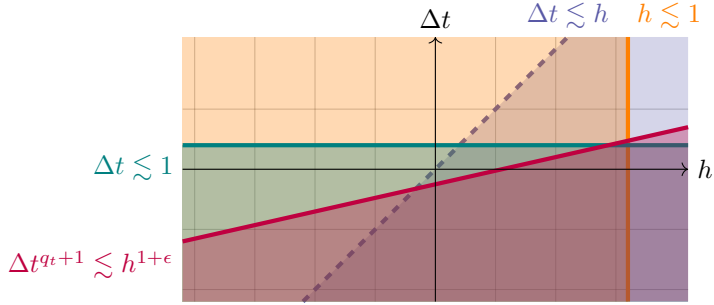
\begin{figure}[ht!]
  \begin{center}
\begin{tikzpicture}[domain=0:4,scale=0.8]
  \foreach \X in {-4,-3,-2,-1,0,1,2,3}
  {
    \foreach \x in {
      0.0
      }
    {
      \draw[black,opacity=0.25] (\X+\x,-2.2) -- (\X+\x,2.2);
    }
  }
  \foreach \Y in {-2,-1,0,1}
  {
    \foreach \y in {
      0.0
      }
    {
      \draw[black,opacity=0.25] (-4.2,\Y+\y) -- (4.2,\Y+\y);
    }
  }

  \fill[gray!70!blue, opacity=0.25]  (-2.2,-2.2) -- (2.2,2.2) node[above,opacity=1.0] {$\Delta t \lesssim h$} -- (4.2,2.2) -- (4.2,-2.2) -- cycle;
  \draw[gray!70!blue,ultra thick,dashed] (-2.2,-2.2) -- (2.2,2.2);

  \fill[orange, opacity=0.3]  (-4.2, -2.2) -| (3.2,2.2) node[above right,opacity=1.0] {$h \lesssim 1$} -| (-4.2,-2.2) -- cycle ;
  \draw[orange,ultra thick] (3.2,-2.2) -- (3.2,2.2);
  \fill[teal, opacity=0.3]  (-4.2, 0.4) node[below left,opacity=1.0] {$\Delta t \lesssim 1$} -| (4.2,-2.2)  -| (-4.2,-2.2) -- cycle;
  \draw[teal,ultra thick] (-4.2,0.4) -- (4.2,0.4);
  \fill[purple, opacity=0.3]  (-4.2,-1.2) node[below left,opacity=1.0] {$\Delta t^{q_t+1} \lesssim h^{1+\epsilon}$} -- (-4.2,-2.2) -| (4.2,0.7) -- cycle;
  \draw[purple,ultra thick] (-4.2,-1.2) -- (4.2,0.7);
  \draw[->] (-4.2,0) -- (4.2,0) node[right] {$h$};
  \draw[->] (0,-2.2) -- (0,2.2) node[above] {$\Delta t$};

\end{tikzpicture} \vspace*{-0.35cm}
\end{center}
\caption{Illustration of restrictions on time step and mesh size in a double logarithmic plot. The {\color{orange}orange area corresponds to the restriction $\Delta t$ sufficiently small}, {\color{teal}teal to $h$ sufficiently small} and {\color{purple} purple to $\Delta t^{q_t+1} \lesssim h^{1+\epsilon}$} in \cref{ass:timefiner} and {\color{gray!70!blue} grey-blue to $\Delta t \lesssim h$} in \cref{ass:dtlesssimh_feblend}.} \label{fig:timefiner}
\end{figure}
\begin{remark}[Comparison of blending options]
 The \FE~blending causes the mesh deformation to decay within a layer of one element, i.e. width $\mathcal{O}(h)$, leading to $\mathcal{O}(1/h)$ gradients. 
 This restricts the boundedness proofs of the upcoming analysis to space-time refinement strategies with $\Delta t \lesssim h$ which may be sufficient for many practical applications and hence the first choice. 
 To allow for a wider range of combinations of space and time refinements, the smooth blending is better-suited as it implies a decay of the mesh deformation on a fixed width, i.e. of order $\mathcal{O}(1)$. 
 Further, the \FE~blending leads to mesh deformations which are discontinuous along time slice boundaries, which necessitates a mesh transfer operation when the geometry is used in a \FE~simulation, cf. \cite{HLP2022}. 
 On the other hand, a continuous-in-time mesh deformation function can be obtained for the smooth blending variant.
 However, the major disadvantage of the smooth blending function is the larger bandwidth of active elements which scales like $h^{-d}$ while it scales only with $h^{1-d}$ for the \FE~blending.
\end{remark}

\section{Construction and analysis of mappings on active elements} \label{sec:mappings}
The section gathers aims at constructing and analysing a theoretical ideal mapping and a realisable mapping on the set of active elements. 
Several results that will provide the accuracy and regularity of the constructed discrete mapping are technical. To help the reader digesting the section despite many technical details, we start the section with an outline in the subsequent subsection. 
\subsection{Outline of the section} \label{sec:outline}
The goal of this section is to construct a realisable and continuous mapping $\Theta_h^\Gamma$ on the active mesh that maps $\partials \Qlin$ close to $\partials \Qn{n}$. 
First, we will construct an ideal mapping $\Psi^\Gamma$ that is geometrically exact so that 
$\partials \Qlin$ is exactly mapped to $\partials \Qn{n}$ and show important regularity and smallness estimates for this mapping. The difference between the discretely constructed mapping and the ideal mapping then allows us to characterise the accuracy of the discrete mapping. 
In particular, to obtain optimal estimates on the difference between the two mappings with respect to (first order) spatial and temporal derivatives, we introduce two further intermediate mappings that are semi-discrete, one discrete only in time, the other discrete only in space. 

Next, we will explain the procedure for the subsections. To provide an overview and orientation, we have outlined key results with allocation to the subsections in \cref{fig:analysis_diagram}.
The remainder of this section develops the results mentioned in the diagram, starting at the top left and proceeding clockwise.

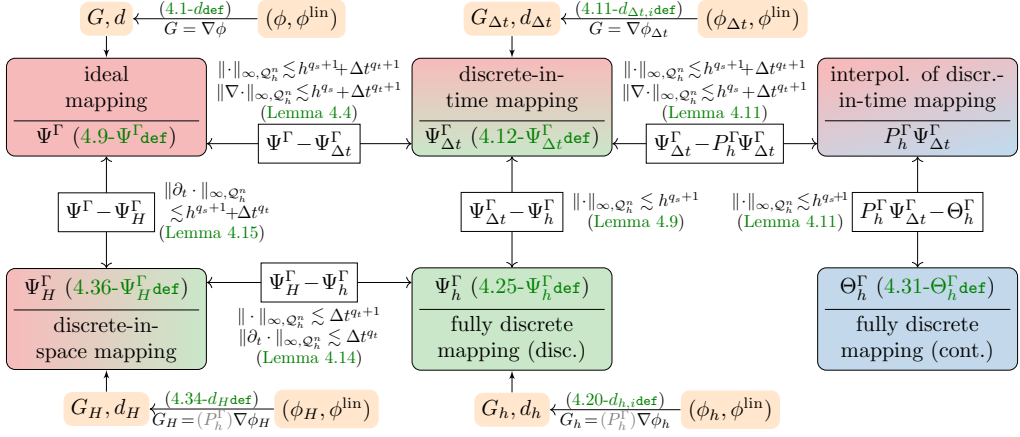
\begin{figure}
  \tikzset{external/export next=false}
  \hspace*{-0.8cm}
  \scalebox{0.825}{
  \begin{tikzpicture}[
    line/.style={draw, -latex'},
    block/.style={rectangle, rounded corners, draw, fill=gray!20, text width=8.5em, text centered, minimum height=3em}
]

\node[block, 
fill=Set1-A!30,
] (PsiG) {
ideal \\ mapping \\[1ex] \hrule ~\\[-0.5ex]
$\Psi^\Gamma_{\vphantom{\Delta t}}$ \eqref{eq:def:PsiG}
};

\node[block, 
shading = axis, left color=Set1-A!30, right color=Set1-C!30,shading angle=0,
right=3.3cm of PsiG] (Psidt) {
discrete-in-time mapping \\[1ex] \hrule ~\\[-0.5ex]
$\Psi_{\Delta t}^\Gamma$ \eqref{def:PsidtG}
};

\node[block, 
shading = axis, left color=Set1-A!30, right color=Set1-C!30,shading angle=90,
fill=teal!20, below=1.8cm of PsiG] (PsiH) {
$\Psi_{H}^\Gamma$ \eqref{eq:def:PsiHG}
    \\[1ex] \hrule ~\\[-0.5ex]
discrete-in-space mapping 
};

\node[block, fill=Set1-C!30, right=3.3cm of PsiH] (Psih) {
$\Psi_{h}^\Gamma$ \eqref{eq:def:PsihG}
\\[1ex] \hrule ~\\[-0.5ex]
fully discrete mapping (disc.)
};

\node[block, fill=Set1-B!30, right=3.3cm of Psih, scale=1] (Thetah) {
$\Theta_{h}^\Gamma$ 
\eqref{eq:def:ThehG}
\\[1ex] \hrule ~\\[-0.5ex]
fully discrete mapping (cont.)
};

\node[block, 
shading = axis, left color=Set1-A!30, right color=Set1-B!30,shading angle=15,
right=3.3cm of Psidt, scale=1] (PhGPsidt) {
interpol. of discr.-in-time mapping
\\[1ex] \hrule ~\\[-0.5ex]
$P_h^\Gamma \Psi_{\Delta t}^\Gamma$
};

\node[fill=orange!20, rounded corners, above=0.35cm of PsiG] (toPsiG) {
$G,d$
};

\node[fill=orange!20, rounded corners, right=2cm of toPsiG] (totoPsiG) {
$(\phi, \phi^{\text{lin}})$
};

\node[fill=orange!20, rounded corners, below=0.35cm of PsiH] (toPsiH) {
$G_H, d_H$
};

\node[fill=orange!20, rounded corners, right=2cm of toPsiH] (totoPsiH) {
$(\phi_H, \phi^{\text{lin}})$
};

\node[fill=orange!20, rounded corners, above=0.35cm of Psidt] (toPsidt) {
$G_{\Delta t}, d_{\Delta t}$
};

\node[fill=orange!20, rounded corners, right=2.2cm of toPsidt] (totoPsidt) {
$(\phi_{\Delta t}, \phi^{\text{lin}})$
};

\node[fill=orange!20, rounded corners, below=0.35cm of Psih] (toPsih) {
$G_h, d_h$
};

\node[fill=orange!20, rounded corners, right=2.1cm of toPsih] (totoPsih) {
$(\phi_{h}, \phi^{\text{lin}})$
};

\path[line,->] (totoPsiG) -- 
node[midway, inner sep=0cm, below=-0.3cm, scale=0.8](totPsiG){
    \begin{minipage}{2cm}
      \centering 
      \eqref{ddef}\\
      $G = \nabla \phi$
    \end{minipage}
} (toPsiG);

\path[line,->] (totoPsidt) -- 
node[midway, inner sep=0cm, below=-0.3cm, scale=0.8](totPsidt){
    \begin{minipage}{2.2cm}
      \centering 
      \eqref{ddtdef}\\
      $G = \nabla \phi_{\Delta t}$
    \end{minipage}
} (toPsidt);

\path[line,->] (totoPsiH) -- 
node[midway, inner sep=0cm, below=-0.3cm, scale=0.8](totPsiH){
    \begin{minipage}{2.5cm}
      \centering 
      \eqref{def_dH}\\
      $G_H \!=\! {\color{gray}(P_h^\Gamma)} \nabla \phi_H$
    \end{minipage}
} (toPsiH);

\path[line,->] (totoPsih) -- 
node[midway, inner sep=0cm, below=-0.3cm, scale=0.8](totPsih){
    \begin{minipage}{2.5cm}
      \centering 
      \eqref{def_dhi}\\
      $G_h\! =\! {\color{gray}(P_h^\Gamma)} \nabla \phi_h$ 
    \end{minipage}
} (toPsih);

\path[line,<->] (PsiG.340) -- 
node[midway, above=0.25cm,scale=0.8](PsiGdt){
    \begin{minipage}{4cm}
        \centering 
        $\Vert \!\cdot\! \Vert_{\infty,\mathcal{Q}_h^n} \!\lesssim\! h^{q_s+1} \!\!\!+\! \Delta t^{q_t + 1}$ \\
        $\Vert \nabla \! \cdot\! \Vert_{\infty,\mathcal{Q}_h^n} \!\lesssim\! h^{q_s} \!+\! \Delta t^{q_t + 1}$ \\
        (\cref{lem:PsidtG-PsiG})
    \end{minipage}
} 
node[midway,draw=black,fill=white]{
    $\Psi^\Gamma\!-\!\Psi^\Gamma_{\Delta t}$
} 
(Psidt.200);

\path[line,<->] (PsiH.20) -- 
node[midway, below=0.3cm,scale=0.8](PsihH){
    \begin{minipage}{4cm}
        \centering 
        $\Vert \cdot \Vert_{\infty,\mathcal{Q}_h^n} \lesssim \Delta t^{q_t + 1}$ \\
        $\Vert \partial_t \cdot \Vert_{\infty,\mathcal{Q}_h^n} \lesssim \Delta t^{q_t}$ \\
        (\cref{lem:PsiHG-PsihG})
    \end{minipage}
} 
node[midway,draw=black,fill=white]{
    $\Psi_H^\Gamma\!-\!\Psi_h^\Gamma$
} 
(Psih.160);

\path[line,<->] (PsiG.270) -- 
node[midway, right=0.cm, scale=0.8](PsiGH){
    \begin{minipage}{4cm}
        \centering 
        $\Vert \partial_t \cdot \Vert_{\infty,\mathcal{Q}_h^n}$~~~  \\
        ~~~$\lesssim\! h^{q_s+1}\!\! +\!\! \Delta t^{q_t}$ \\
        (\cref{Psi_H_Psi_diff_lemma}) 
    \end{minipage}
} 
node[pos=0.5,draw=black,fill=white,rotate=0](L){
    $\Psi^\Gamma\!-\!\Psi_H^\Gamma$
} 
(PsiH.90);

\path[line,<->] (Psidt.270) -- 
node[midway, right=0.3cm, scale=0.8](Psihdt){
    \begin{minipage}{4cm}
        \centering 
        $\Vert \!\cdot\! \Vert_{\infty,\mathcal{Q}_h^n} \lesssim h^{q_s+1}$ \\
        (\cref{lem:PsidtG-PsihG}) 
    \end{minipage}
} 
node[midway,draw=black,fill=white,rotate=0](R){
    $\Psi_{\Delta t}^\Gamma\!-\!\Psi_h^\Gamma$
} 
(Psih.90);

\path[line,<->] (PhGPsidt.270) -- 
node[midway, left=0.35cm, scale=0.8](PhGPsidtQ){
    \begin{minipage}{4cm}
        \centering 
        $\Vert \!\cdot\! \Vert_{\infty,\mathcal{Q}_h^n} \!\lesssim\! h^{q_s\!+\!1}$ \\
        (\cref{lem:PhGPsidtG}) 
    \end{minipage}
} 
node[midway,draw=black,fill=white,rotate=0](R){
    $P_h^\Gamma \Psi_{\Delta t}^\Gamma\!-\!\Theta_h^\Gamma$
} 
(Thetah.90);


\path[line] (toPsiG) -- (PsiG);
\path[line] (toPsiH) -- (PsiH);
\path[line] (toPsidt) -- (Psidt);
\path[line] (toPsih) -- (Psih);


\path[line,<->] (Psidt.340) -- 
node[midway, above=0.25cm,scale=0.8](PsidtPhG){
    \begin{minipage}{4cm}
        \centering 
        $\Vert \!\cdot\! \Vert_{\infty,\mathcal{Q}_h^n} \!\lesssim\! h^{q_s+1} \!\!\!+\! \Delta t^{q_t + 1}$ \\
        $\Vert \nabla \! \cdot\! \Vert_{\infty,\mathcal{Q}_h^n} \!\lesssim\! h^{q_s} \!+\! \Delta t^{q_t + 1}$ \\
        (\cref{lem:PhGPsidtG})                
    \end{minipage}
} 
node[midway,draw=black,fill=white]{
    $\Psi^\Gamma_{\Delta t}\!-\!P_h^\Gamma \!\Psi^\Gamma_{\Delta t}$
} 
(PhGPsidt.200);

\end{tikzpicture}
  }
  \hspace*{-1.2cm}
  \vspace*{-0.5cm}
  \caption{Outline of the construction and analysis of the mappings on active elements. The ideal mapping is on the top left corner of the diagram while the realizable mapping used in practice is on the bottom right corner. The intermediate mappings allow for a step-by-step analysis of the error (in different norms) between the ideal and the realizable mapping.} 
  \label{fig:analysis_diagram}
  \vspace*{-0.5cm}
\end{figure}

In \cref{sec:ideal_mapping} will first construct an ideal space-time mapping on active elements. The mapping is based on the following procedure: We first fix a \emph{spatial search direction} $G: \QGn \to \mathbb{R}^d$ based on $\phix$, so that for every point $(x,t)$ in $\QGn$ we can determine a \emph{distance function} $d: \QGn \to \mathbb{R}$ as the solution of a scalar one-dimensional problem that involves $\philin$, $\phix$ and the blending $\b$. 
After showing well-posedness of the local problems and regularity of $d$
in \cref{lem:d}
we can define the ideal mapping $\Psi^\Gamma$ and deduce regularity properties, cf. \eqref{eq:def:PsiG} and \cref{cor:PsiG}.

In a next step, in \cref{sec:discrete-in-time-mapping} we consider an analogue mapping that is discrete-in-time based on spatial problems on a set of time points $\{\ti{i}\}_{i=0}^{q_t}$ and the level set function $\phix$ replaced by the time-discrete approximation $\phidt$, yielding a discrete-in-time search direction $G_{\Delta t}$, a distance function $d_{\Delta t}$ and bounds on $d_{\Delta t}$ that are independent of $\Delta t$, cf. \cref{lem:ddt}. The resulting mapping $\Psi_{\Delta t}^\Gamma$ is then defined in \eqref{def:PsidtG} and proximity to $\PsiG$
in the $\infty$-norm as well as in the $\infty$-norm for the gradients
is shown in \cref{lem:PsidtG-PsiG}.

Both previous constructions are theoretical mappings needed only in the analysis. The fully discrete mapping $\Theta_h^\Gamma$, respectively a proxy $\Psi_h^\Gamma$ is then constructed in \cref{sec:discrete_mapping}. 
For the construction we then consider the available approximation $\phih$ of $\phix$ and construct a proper discrete search direction $G_h$ approximating $G$, cf. \cref{G_Gh_diff_lemma}. As in \cref{sec:discrete-in-time-mapping} the construction is reduced to a set of time points $\{\ti{i}\}_{i=0}^{q_t}$.
In contrast to the previous level set functions $\phix$ and $\phidt$ and the corresponding search directions $G$ and $G_{\Delta t}$, both $\phih$ and $G_h$ are no longer globally smooth, but only piecewise smooth and at most globally continuous. As in the spatial problem, cf. \cite{L_CMAME_2016}, to obtain a well-posed and computational feasible one-dimensional problem for each point a polynomial extension of $\mathcal{E}_T\phih$ instead of $\phih$ is considered and shown to be sufficiently accurate, cf. \cref{diff_phi_phihn_spaceD_special}, which allows the problems to determine the distance function can be solved element by element.
In \cref{d_h_lemma} we show well-posedness and smallness of the discrete distance function that allows us to define a discrete mapping $\Psi_h^\Gamma$ in \eqref{eq:def:PsihG}.
In \cref{lem:PsidtG-PsihG} we show (optimal) proximity of $\Psi_h^\Gamma$ to $\Psi_{\Delta t}^\Gamma$ in the $\infty$-norm. Together with the same proximity results of $\Psi_{\Delta t}^\Gamma$ to $\PsiG$ in \cref{lem:PsidtG-PsiG} this yields proximity of $\Psi_h^\Gamma$ to $\PsiG$, cf. \cref{Psi_hi_Psi_diff_lemma}.

As a consequence of the polynomial extension applied in the construction of the distance function $\Psi_h^\Gamma$ will in general be discontinuous across element boundaries. To obtain a continuous mapping, we apply an Oswald-type projection $P_h^\Gamma$ to define $\Theta_h^\Gamma$ in \eqref{eq:def:ThehG} which only introduces a minor spatial perturbation, cf. \cref{lem:PhGPsidtG}. Proximity results for $\Theta_h$ to $\PsiG$, including optimal bounds for the gradient, are shown in \cref{lem:Theta_h_Psi_Gamm_diff_bound}, a core results of this section.
These bounds provide good bounds for the mapping and its gradient, but bounds for the time derivative derived from it would not be robust in an anisotropy between spatial and temporal resolution.

To fix this, an additional route in estimating the proximity of $\Theta_h^\Gamma$ to $\PsiG$ w.r.t. the time derivative is taken in \cref{sec:discrete-in-space-mapping}. A discrete-in-space, but continuous-in-time, mapping is constructed based on the same construction principles as for the fully discrete mapping in \cref{sec:discrete_mapping}, but with $\phih$ replaced by $\phiH$, yielding the search direction $G_H$ and a distance function $d_H$ with similar bounds as for $d_h$, cf. \cref{d_H_lemma}. The correspondingly defined mapping $\Psi_H^\Gamma$, cf. \eqref{eq:def:PsiHG} is close to $\Psi_h^\Gamma$ with bounds that are independent of the spatial resolution, cf. \cref{lem:PsiHG-PsihG}. Furthermore, $\Psi_H^\Gamma$ offers sufficient proximity to $\PsiG$ also w.r.t. the time derivative, cf. \cref{Psi_H_Psi_diff_lemma}. Essentially, putting these results together then yields a similar proximity bound for $\Theta_h$ to $\PsiG$ which is shown in the final main result of this section, \cref{lem:Theta_h_Psi_Gamm_diff_bound_first_dt}.




\subsection{Construction of ideal mapping on active elements} \label{sec:ideal_mapping}
As a first step, we introduce a mapping $\PsiG$ which maps $\Qlin$ close to $\Qn{n}$. Fix some space-time point $(x,t) \in \QGn$, we want to match the levelset value of $(1- \b(x,t)) \philin(x,t) + \b(x,t) \phix(x,t)$ with the levelset value of $\phix$ at a point $y \in \tOmega$. For this, we define the (purely spatial) \emph{search direction} \hypertarget{def:G}{ $\G(x,t) \coloneqq \nabla \phix(x,t)$}, 
and the \emph{distance function} 
\hypertarget{def:d}{$\d(x, t) \colon \OG \to \mathbb{R}$} 
(this implies $\d \colon \QGn \to \mathbb{R}$) 
to be the (in absolute value) smallest number s.t. 
\begin{equation}
 \!\! \phix(x + \d(x,t) \G(x,t), t) = (1- \b(x,t)) \philin(x,t) + \b(x,t) \phix(x,t) ~~~ \forall (x,t) \in \QGn. \label{ddef}
 \addtocounter{equation}{1}\tag{\theequation-\texttt{$d$\footnotesize{def}}} 
\end{equation}
We correspondingly define \hypertarget{def:di}{$\di{i}(x) \coloneqq \d(x,\ti{i})$} and \hypertarget{def:Gi}{$\Gi{i}(x) \coloneqq \G(x,\ti{i})$.}

Note that for $(x,t)$ such that $b(x,t)=1$, the trivial solution is $\d(x,t) = 0$. On the other side, for $b(x,t) = 0$, the condition reduces to
\begin{equation*}
 \phix(x + \d(x,t) \G(x,t), t) = \philin(x,t)\quad \textnormal{for all } (x,t) \in \QGn,
\end{equation*}
which is very similar to the construction in \cite[Eq. (3.1)]{LR_IMAJNA_2018} for the merely spatial case. We illustrate this construction for a space-time point where $\philin(x,t) =0$ and $b(x,t)=0$ (this case could be called ``full blending'') in \cref{ddef_illu}: Starting from a point $(x,t)$ with some value of $\philin(x,t)$, in this case 0, we calculate the search direction $G(x,t)$ and move so far along this direction until we obtain the same value in $\phix(\dots,t)$.

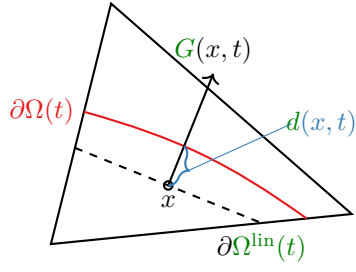
\begin{figure} \centering
 \begin{tikzpicture}[scale=0.8]
  \draw[thick] (0,0) -- (1,4) -- (5,0.5) -- cycle;
  \draw[thick, dashed] (0.4,4*0.4) -- (0.7*5, 0.7*0.5) node[below] {$\partial \Omlin{t}$};
  \draw[thick, Set1-A, name path=exact] (0.55,4*0.55) node[left] {$\partial \Omega(t)$} to [bend left=10] (0.85*5, 0.85*0.5);
  \draw[thick] (0.5*0.4 +0.5*0.7*5, 0.5*4*0.4+0.5*0.7*0.5) circle (2pt) node[below] {$x$};
  \draw[thick, ->, name path=searchdir] (0.5*0.4 +0.5*0.7*5, 0.5*4*0.4+0.5*0.7*0.5) -- + ($0.6*(4*0.4 - 0.7*0.5, -0.4 + 0.7*5)$) node[above] {$\G(x,t)$};
  \pgfdecorationsegmentamplitude=5pt
  \draw[thick, Set1-B, decorate, decoration=brace, name intersections={of = exact and searchdir}]  (intersection-1) -- node[midway, name=LP]{} (0.5*0.4 +0.5*0.7*5, 0.5*4*0.4+0.5*0.7*0.5);
  \node[Set1-B] (dL) at (4.5,2) {$\d(x,t)$}; \draw[Set1-B] ($(dL) + (-0.5,0)$) -- ($(LP)+(0.2,-0.08)$);
 \end{tikzpicture} \vspace*{-0.3cm}
 \caption{Illustration of definition of the function $\protect\d$ at a space-time point with $\protect\philin(x,t)=0$.}
 \label{ddef_illu} \vspace*{-0.25cm}
\end{figure}

The function $d$ has the following fundamental properties:
\begin{lemma} \label{lem:d}
For $h$ and $\Delta t$ sufficiently small, the relation \cref{ddef} defines a unique $\d(x,t)$, and $d \in C^0(\QGn) \cap C^{\lphi}(\ThGQ) $. Furthermore, there holds
\begin{subequations} 
  \label{eq:d}
  \begin{align}
    \| \partial_t^{m_t} D^{m_s} \d\|_{\infty, \QGn} & \!\lesssim h^{2-m_s} \!+\! \Delta t^{q_t + 1 - m_t}~ \text{for } m_s, m_t \in \{0,1\},
    \tag{\ref{eq:d}a-\texttt{$d$\footnotesize{bnd}}}
    \label{eq:da}
    \\
  \| \partial_t^{m_t} \d \|_{\infty,\QGn} & \! \lesssim h^2 \!+\! \Delta t^{q_t+1-m_t}~ \text{for } m_t \in \{0,..,q_t+1\}, 
   \tag{\ref{eq:d}b-\texttt{$d$\footnotesize{bnd}}}
   \label{eq:db}
    \\
   \| \partial_t^{m_t} D^{m_s} \d\|_{\infty,\ThGQ} & \!\lesssim 1 ~ \textnormal{for } m_s\!=\!0,..,q_s\!+\!1, m_t\!=\!0,..,q_t+1, m_t\!+\!m_s \!\leq\! \lphi,
  \tag{\ref{eq:d}c-\texttt{$d$\footnotesize{bnd}}}
  \label{eq:dc}
\end{align}
\end{subequations}
\end{lemma}
\begin{proof}
  The proof follows similar ideas as the proof of \cite[Lemma 3.1]{LR_IMAJNA_2018}. Differences originate from the fact that $\philin$ deviates from $\phix$ not only by spatial but also temporal errors and that we have to involve the blending function $\b(x,t)$ which was not considered in \cite{LR_IMAJNA_2018}. Furthermore, with the time-dependence more derivate terms become involved. We treat the proof of some basic properties in the main part and refer to appendices for the more technical proofs.

  \underline{Well-posedness of $d$}:\\ 
 Define for fixed $\alpha_0 > 0$ on $\alpha \in [-\alpha_0 h, \alpha_0 h]$ at the point $(x,t) \in \QGn$ the function
 \begin{align}
  g(\alpha) &\coloneqq \phix(x + \alpha \G(x,t),t) - (1- \b(x,t)) \philin(x,t) - \b(x,t) \phix(x,t). \label{g_def} \\
  & = \phix(x + \alpha \G(x,t),t) - \philin(x,t) + \b(x,t) (\philin(x,t) - \phix(x,t)). \nonumber
 \end{align}
 As $\phix \in C^{\lphi}(\QU)$, $g \in C^{\lphi}([-\alpha_0 h, \alpha_0 h])$ for $h\leq h_0$ sufficiently small.
 This allows us to apply the Taylor series theorem, leading to
 \begin{equation}
  g(\alpha) = g(0) + g'(0) \alpha + \nicefrac{1}{2}~ g''(\xi) \alpha^2 \quad \textnormal{ for some } \xi \in [\min\{0, \alpha\},\max\{0, \alpha\}]. \label{g_Taylor}
 \end{equation}
 Combining \eqref{diff_phi_phihn_spaceD} from \cref{diff_phi_phihn} and \eqref{diff_grad_phi_philin_dt} from \cref{diff_phihn_philin}, we obtain 
 at point $(x,t)$ that $|g(0)| = | (1-b) (\philin - \phix)| \lesssim  h^2 + \Delta t^{q_t + 1} \lesssim h^{1+\epsilon}$ (with \eqref{delta_t_smallness} from \cref{ass:timefiner}). From the chain rule, we obtain furthermore $g'(0)= \| \nabla \phix \|_2^2$. We now claim that there is a unique $\alpha \in [-\alpha_0 h, \alpha_0 h]$ such that $g(\alpha) = 0$. First assume that there are two such $\alpha$ and denote them as $\alpha_1$ and $\alpha_2$. Then they are equal by \cref{g_def} and \cref{signed_dist_like_prop} as at 
 $$0 = |g(\alpha_1) - g(\alpha_2)| \stackrel{\eqref{g_def}}{=} |\phix(x+\alpha_1 \G(x,t),t) - \phix(x+\alpha_2 \G(x,t),t)| \stackrel{\eqref{signed_dist_like_prop}}{\simeq} | \alpha_1 - \alpha_2 |.$$ 
 For existence, we see that with \eqref{eq:cnphi} and $|g(0)| \leq (\alpha_0 h)^{1+\epsilon}$
 \begin{align} \nonumber
  g(\alpha_0 h) & \!\stackrel{\eqref{g_Taylor}}{=}\!  g(0) +  \| \nabla \phix \|_2^2 \alpha_0 h + \nicefrac{1}{2}~ g''(\xi) \alpha_0^2 h^2 
  \geq \cnphi{2} (\alpha_0 h) - c (\alpha_0 h)^{1+\epsilon} \!\!-\! \nicefrac{1}{2}~c  (\alpha_0 h)^2  
  \end{align}
and hence $g(\alpha_0 h)>0$ for $h \leq h_0$ sufficiently small and by the same argument, $g(-\alpha_0 h)<0$. Then with the continuity of $g$ and the intermediate value theorem, we obtain the existence of a root of $g$. Hence, $\d(x,t)$ is well defined as this root, has $|\d(x,t)| < \alpha_0 h$ and regularity $d \in C^{\lphi}(T)$.

The proof of the estimates in \eqref{eq:da} will be subdivided into the cases $(m_t,m_s) \in \{(0,0), (0,1), (1,1)\}$ while $(m_t,m_s) = (1,0)$ will be treated within the proof of \eqref{eq:db}.

 \underline{Proof of \eqref{eq:da} for $(m_t,m_s)=(0,0)$}:\\ 
 Writing out the above Taylor expansion for $g(\d(x,t))$, we obtain
 \begin{align}
  & g(\d(x,t))  = 0 = g(0) + \| \nabla \phix \|_2^2 \d(x,t) + \nicefrac{1}{2}~ g''(\xi) \d^2(x,t) \label{g_in_first_lemma_proof_Taylor_expanded} \\[-1ex]
  \Rightarrow \quad &|\d(x,t)| \leq \| \nabla \phix \|_2^{-2} \left| - g(0) + \nicefrac{1}{2} ~ |g''(\xi)| (\alpha_0 h)^2 \right| \stackrel{\eqref{eq:cnphi}}{\lesssim} h^2 + \Delta t^{q_t + 1}, \nonumber
 \end{align}
 With the above estimate on $|g(0)|$ and the regularity $\phix \in C^{q_s+q_t+2}(\U)$, the estimate is complete.

 \underline{Proof of \eqref{eq:da} for $(m_t,m_s)=(0,1)$}:\\ 
 Next, we apply the chain rule 
 on \cref{ddef}. 
\hypertarget{def:pi}{
We introduce $\ppi\colon \tQn \to \tQn, (x,t)^T \mapsto (x + \d(x,t) \G(x,t),t)^T$, 
}
such that 
\begin{equation} \label{char:pi} 
  \addtocounter{equation}{1}\tag{\theequation-\texttt{$d\!\mid\!\!\pi$}}    
  \phix \circ \ppi = (1-\b) \philin + \b \phix = \philin + \b (\phix - \philin). 
\end{equation}
For the spatial derivative, we obtain \vspace*{-0.2cm}
 \begin{align}
  \overbrace{\nabla \philin (x,t) - \nabla \phix(\ppi(x,t))}^{=:A} = &\nabla \d ~~ \overbrace{\G(x,t)^T \nabla \phix(\ppi(x,t))}^{=:B} + \overbrace{\d \nabla \G(x,t)^T \nabla \phix(\ppi(x,t))}^{=:C}\nonumber \\
  &+ \overbrace{\nabla \b (\philin - \phix)(x,t)}^{=:D} + \overbrace{\b (\nabla \philin - \nabla \phix)(x,t)}^{=:E} \label{eq_step_234}
 \end{align}
 For the terms $A, B, C, D, E$ we obtain the following estimates:
 \begin{subequations}
 \begin{align}
  \| A \| &\leq \underbrace{\| \nabla \philin (x,t) - \nabla \phix(x,t) \|}_{\lesssim h + \Delta t^{q_t + 1}} + \underbrace{\| \nabla \phix(x,t) - \nabla \phix(\ppi(x,t))\|}_{\lesssim |\phix|_{2,\infty, \U} \d \lesssim h^2 + \Delta t^{q_t+1}} \lesssim h + \Delta t^{q_t + 1}. \\
  B &= \| \nabla \phix (x,t)\|_2^2 + \G(x,t)^T(\nabla \phix(\ppi(x,t)) - \nabla \phix(x,t)) \label{eq_step_238} \gtrsim 1 - h^2 - \Delta t^{q_t+1}, \\
 \| C\| &\lesssim |d| |\phix|_{2,\infty } (1+h^2 + \Delta t^{q_t +1})\lesssim h^2 + \Delta t^{q_t+1}, \label{eq_step_239} \\
 \| D\| & \lesssim |b|_{H^1} ( h^2 + \Delta t^{q_t + 1}) \lesssim h^2 + \Delta t^{q_t + 1},  \quad 
 \| E\| \lesssim h + \Delta t^{q_t+1}. 
\end{align}
\end{subequations}
Overall, this yields $
 |\nabla d | \lesssim (h + \Delta t^{q_t+1}) / ( 1 - h^2 - \Delta t^{q_t+1}),
$
which proves \eqref{eq:da} for $(m_t,m_s)=(0,1)$.

\underline{Proofs of \eqref{eq:da} for $(m_t,m_s)=(1,1)$, \eqref{eq:db} and \eqref{eq:dc}}:\\
Similarly, the proofs of \eqref{eq:da} for $(m_t,m_s)=(1,1)$, \eqref{eq:db} for $m_t \geq 1$ and \eqref{eq:dc} rely on (repeated) applications of the chain rule. We give the proofs in \cref{app:damtms11,app:eqdbmtgt1,app:eqdc}.
\end{proof}
 
Having introduced the function $\d(x,t)$, we now define
\begin{equation}
\hypertarget{def:PsiG}{\PsiG (x,t) \coloneqq x + \d(x,t) \G(x,t)} \text{ and } \hypertarget{def:PsiGi}{\PsiGi{i} (x) \coloneqq \PsiG(x,\ti{i}),~i=0,..,q_t}
\label{eq:def:PsiG}
\addtocounter{equation}{1}
\tag{\theequation-\texttt{$\Psi^\Gamma$\!\footnotesize{def}}} 
\end{equation}

Then, the previous lemma immediately implies with $\G(x,t) = \nabla \phix \in C^1(U)\cap C^{\lphi}(\QU)$ and $\nabla \phix = \mathcal{O}(1)$ the following results for $\PsiG$:
\begin{corollary} \label{cor:PsiG}
 There holds
 \begin{align}
  \| \partial_t^{m_t} D^{m_s} (\PsiG\! - \idx) \|_{\infty, \QGn} &\!\lesssim\! h^{2-m_s} + \Delta t^{q_t+1-m_t}, m_t, m_s \in \{0,1\}, 
  \label{Psi_Gamma_lo_deriv_ho_boundA} 
  \addtocounter{equation}{1}\tag{\theequation a-\texttt{$\Psi^\Gamma$\!\footnotesize{bnd}}}    
  \\
  \| \partial_t^{m_t} \PsiG\|_{\infty, \QGn} &\!\lesssim\! h^2 + \Delta t^{q_t+1-m_t}, \quad m_t\in\{0,\dots,q_t+1\} \label{Psi_Gamma_first_dt_boundB} 
  \addtocounter{equation}{0}\tag{\theequation b-\texttt{$\Psi^\Gamma$\!\footnotesize{bnd}}}    
  \\
  \|D^{m_s} \partial_t^{m_t} \PsiG \|_{\infty,\ThGQ} &\!\lesssim\! 1, m_s\!\!\leq q_s\!+\!1, m_t\!\!\leq q_t\!+\!1, m_t\!+\!m_s \!\leq\! \lphi.\! \label{Psi_Gamma_dt_ho_derivs_bounded}
  \addtocounter{equation}{0}\tag{\theequation c-\texttt{$\Psi^\Gamma$\!\footnotesize{bnd}}}    
 \end{align}
\end{corollary}

\subsection{Construction of discrete-in-time mapping on active elements} 
\label{sec:discrete-in-time-mapping}

Before we turn to the realisable mapping in \cref{sec:discrete_mapping} we introduce a semi-discrete mapping $\PsidtG$
on the active domain elements based on the semi-discrete levelset approximations $\phidt$.

We start with a mapping that is discrete in time based on the ansatz
$\PsidtG = \sum_{i=0}^{q_t} \ell_i(t) \PsidtGi{i}(x)$.
In analogy to \eqref{ddef}, we use the 
\hypertarget{def:Gdti}{\emph{search direction} $\Gdti{i}(x) := \nabla \phidt(x,\ti{i})$}
and the 
\hypertarget{def:ddti}{\emph{distance function} $\ddti{i} \colon \OG \to \mathbb{R}$} which is defined to be the (in absolute value) smallest number s.t. 
\begin{equation}
  \phidti{i}(x + \ddti{i} \Gdti{i}) = (1- \bi{i}(x)) \philini{i}(x) + \bi{i}(x) \phidti{i}(x) \quad \forall~ x \in \OG. \label{ddtdef}
  \addtocounter{equation}{1}\tag{\theequation-\texttt{$d_{\Delta t,i}$\footnotesize{def}}}   
\end{equation}
We then define 
\begin{equation} \label{def:PsidtG}
\hypertarget{def:PsidtGi}{\PsidtGi{i}(x) \coloneqq x + (\ddti{i} \Gdti{i})(x)} \text{ and } \hypertarget{def:PsidtG}{\PsidtG \coloneqq \sum_{i=0}^{q_t} \ell_i(t) \PsidtGi{i}(x)}. 
\addtocounter{equation}{1}\tag{\theequation-\texttt{$\Psi_{\Delta t}^\Gamma$\footnotesize{def}}}    
\end{equation}
The following bounds on $\ddti{i}$ and $\PsidtGi{i}$, for $i=0,\dots,q_t$ are the analogs of \cref{lem:d} and \cref{cor:PsiG} restricted to $\ti{i}$ and right-hand side terms that are independent of $\Delta t$.
\begin{lemma} \label{lem:ddt}
  For $h$ sufficiently small, the relation \eqref{ddtdef} defines unique functions $\ddti{i}(x)$ and $\ddti{i} \in C^0(\OG) \cap C^{q_s+1}(\ThG)$. Furthermore, there holds
 \begin{equation}
  \| D^{m_s} \ddti{i}\|_{\infty,\ThG } \lesssim \min\{1, h^{2-m_s}\} \quad \textnormal{for } m_s \leq q_s + 1. \label{eq:ddt}
  \addtocounter{equation}{1}\tag{\theequation-\texttt{$d_{\Delta t}$\footnotesize{bnd}}}    
\end{equation}
which implies $\PsidtG \in C(\QGn)$ and
\begin{equation}
  \Vert D^{m_s} (\PsidtG - \idx) \Vert_{\infty,\ThGQ} \lesssim \min\{1, h^{2-m_s}\} \quad \textnormal{for } m_s \leq q_s + 1. \label{eq:PsidtG-id}
  \addtocounter{equation}{1}\tag{\theequation-\texttt{$\Psi_{\Delta t}^\Gamma$\footnotesize{bnd}}}    
\end{equation}
\end{lemma}
\begin{proof}
  The proof follows along the same lines as the proof of \cref{lem:d}. There, the r.h.s. bounds at the end are all reduced to the error terms $\philin - \phix$ and derivatives thereof. In the present case, we have to replace $\philin - \phix$ by $\philini{i} - \phidti{i}$ (and derivatives thereof) so that with 
  \eqref{eq:phidt-philin} we can drop the dependency on the time step size and obtain the desired result.
\end{proof}

\begin{lemma} \label{lem:PsidtG-PsiG}
  There holds
  \begin{align} 
    \Vert \PsidtG - \PsiG \Vert_{\infty,\QGn} &\lesssim h^{q_s+1} + \Delta t^{q_t+1}, \label{eq:PsidtG-PsiG} 
    \addtocounter{equation}{1}\tag{\theequation-\texttt{$\Psi_{\Delta t}^\Gamma$\footnotesize{acc}}-a}    
    \\
    \Vert \nabla (\PsidtG - \PsiG) \Vert_{\infty,\ThGQ} &\lesssim h^{q_s} + \Delta t^{q_t+1}, \label{eq:grad-PsidtG-PsiG}
    \addtocounter{equation}{0}\tag{\theequation-\texttt{$\Psi_{\Delta t}^\Gamma$\footnotesize{acc}}-b}    
  \end{align}
\end{lemma}
\begin{proof}
  First, we note that with \eqref{Psi_Gamma_dt_ho_derivs_bounded} we have 
  $\Vert \PsiG - \It{q_t} \PsiG \Vert_{\infty,\QGn} \lesssim \Delta t^{q_t+1}$.
  Hence, it suffices to prove \eqref{eq:PsidtG-PsiG} and \eqref{eq:grad-PsidtG-PsiG} at the nodes $\ti{i}$.

  \underline{Proof of \eqref{eq:PsidtG-PsiG} at $\ti{i}$, $\Vert \PsidtGi{i} - \PsiGi{i} \Vert_{\infty,\OG} \lesssim h^{q_s+1} + \Delta t^{q_t+1}$ :}\\
  We have for $x \in \OG$ (we will drop the argument for $\di{i}$, $\ddti{i}$, $\Gi{i}$, $\Gdti{i}$ and $\bi{i}$ in the following chain of equations, so that e.g. $\Gi{i} = \Gi{i}(x)$)
  \begin{align}
  & \hspace*{-.7cm}   | \di{i}  - \ddti{i} | \stackrel{\eqref{signed_dist_like_prop}}{\simeq} | \phii{i}(x + \di{i} \Gi{i}) - \phii{i}(x + \ddti{i} \Gi{i}) | \label{eq:di-ddti} \hspace*{-2cm} \\
&   \hspace*{-2.3cm} \stackrel{\hphantom{\eqref{ddef}\&\eqref{ddtdef}}}{\lesssim} \!\!\!\!\!\!\!\!\!\!\!\!\!\!\!\!\!\!\!\!\!\!\!\! | \phii{i}(x + \di{i} \Gi{i}) - \phii{i}(x + \ddti{i} \Gdti{i}) | + | \phii{i}(x + \ddti{i} \Gdti{i}) - \phii{i}(x + \ddti{i} \Gi{i}) | \nonumber \hspace*{-2cm}\\ 
&   \hspace*{-2.3cm} \stackrel{\eqref{ddef}\&\eqref{ddtdef}}{\lesssim} \!\!\!\!\!\!\!\!\!\!\!\!\!\!\!\!\!\!\!\!\!\!\!\!| \bi{i} (\phii{i}(x) - \phidti{i}(x))| + |\nabla \phii{i}|_\infty |\ddti{i}|_\infty |\Gi{i} - \Gdti{i}|_\infty \lesssim h^{q_s+1} + \Delta t^{q_t+1}. \nonumber \hspace*{-2cm}
  \end{align}
  Hence 
  $ \displaystyle
  |\PsidtGi{i} - \PsiGi{i}| \lesssim |\ddti{i} - \di{i}| |\Gi{i}| + |\di{i}| |\Gdti{i} - \Gi{i}|
  $
   and the result follows with $\|\G\|_{\infty,\QhU} \lesssim 1$, $\|\d\|_{\infty,\ThGQ} \lesssim h^2 + \Delta t^{q_t+1}$ and $\|\nabla \phidt - \G\|_{\infty,\QhU} = \| \nabla (\phidt - \phix)\|_{\infty,\QhU} \lesssim h^{q_s} + \Delta t^{q_t+1}$, cf. \eqref{diff_phid_phi}.

   \underline{Proof of \eqref{eq:grad-PsidtG-PsiG}:}\\
   The proof follows similar strategies as the proof of \eqref{eq:da} for $(m_t,m_s)=(0,1)$ and we give the details for the sake of completeness in \cref{app:gradPsidtacc}.
\end{proof}

\subsection{Construction of fully discrete mapping on active elements} \label{sec:discrete_mapping}
We now turn to the realisable discrete mapping $\ThehG$ on the active elements. Here, we try to mimic the construction of the continuous mapping $\PsiG{}$ in \eqref{eq:def:PsiG} with adaptations due to the fact that $\phii{}$ may not be available in general and $\PsiG{}$ is not a \FE~function.

\subsubsection{Discrete search directions}
To transfer the construction principle of the spatially continuous mappings in the previous sections to a computable setting, we need to first define a \emph{discrete search direction}. 
We follow \cite{LR_IMAJNA_2018} in approximating $\G = \nabla \phix$ where two options are presented: Let $x \in \OG$, then
\begin{equation}
 \hypertarget{def:Ghi}{\Ghi{i}(x) \coloneqq  \nabla \phihi{i}(x) \quad \text{ or } \quad \Ghi{i}(x) \coloneqq (\PhG \nabla \phihi{i})(x)}, \quad ~i=0,\dots,q_t.
\end{equation}
and \hypertarget{def:Gh}{$\Gh(x,t) \!\coloneqq\! \sum_{i=0}^{q_t} \elli{i}(t) \Ghi{i}(x)$}, s.t.
$\Gh(x,t) \!=\! \nabla \phih(x,t)$ or $\Gh(x,t) \!=\! \PhG \nabla \phih(x,t)$, respectively, for $(x,t) \in \QGn$. 
\hypertarget{def:PhG}{Here, $\PhG: C(\ThG) \to \Wh{k}$ is an Oswald-type spatially averaged interpolation operator as in \cite[Sect. 2.2]{LR_IMAJNA_2018}. It satisfies
\begin{equation} \label{PhG_prop}
 \PhG w = \Is{q_s} w, \quad \| \PhG w \|_{\infty, \Omega} \lesssim \|w\|_{\infty,\Th} \quad \forall w \in C(\OG)^d.
\end{equation}
}
where \hypertarget{def:Is}{$\Is{q_s}$ denotes a generic spatial interpolation operator.}
Then it holds
\begin{lemma} \label{G_Gh_diff_lemma}
 For $\Gh = \nabla \phih$ or $\Gh = \PhG \nabla \phih$ there holds for $m=0,1$:
 \begin{align}
  \| D^m(\Gh - \G) \|_{\infty, \ThGQ} &\lesssim h^{q_s-m} + \Delta t^{q_t + 1} \label{G_Gh_diff_lemma_eq2}.
  \addtocounter{equation}{1}\tag{\theequation-\texttt{$G_h$\footnotesize{acc}}} 
 \end{align}
\end{lemma}
\begin{proof}
 Let $\Gh = \nabla \phih$. Then the result follows from \cref{diff_phi_phihn_spaceD}. 
 For $\Gh = \PhG \nabla \phih$ a proof is given in  \cref{app:proofG_Gh_diff_lemma}.
\end{proof}
\begin{remark}[Choice of the search direction]
 The overall purpose of the discrete search direction is to allow for a computationally feasible coordination between linear reference and higher order geometry, as indicated by its application below. We show that both choices for $\Gh$ are reasonable candidates for this. It might be computationally convenient to consider a continuous in space search direction, which is why both options are presented. This overall structure is a straightforward generalisation of the merely spatial construction, c.f. \cite[Remark 3.4]{LR_IMAJNA_2018}.
\end{remark}
\subsubsection{Construction of element-wise approximation of \texorpdfstring{$\Psi^\Gamma$}{ideal mapping}}

Next, our aim is to introduce a computable approximation of $\Psi^\Gamma$. 

To this end, we recall that the higher-order approximation of $\phix$ was given as $\phih = \sum_{i=0}^{q_t} \elli{i}(t) \phihi{i}(x)$. Each of these $\phihi{i}(x)$ can be used to construct a spatial approximation of $\PsiGi{i}$ as follows:

Let $i=0,\dots,q_t$ and \hypertarget{def:ET}{$\ET \phihi{i}$ be the polynomial extension} of $\phihi{i}|_T$ for $T\in \ThG$. 
Then, we define the function $\dhi{i}\colon \OG \to [-\delta,\delta]$ for $\delta > 0$ sufficiently small as follows: For $x \in \OG$, let \hypertarget{def:dhi}{$\dhi{i}(x)$ be the (in absolute value) smallest number such that}
\begin{equation}
 \ET \phihi{i} (x + (\dhi{i} \Ghi{i})(x)) = (1-\bi{i}(x)) \philini{i}(x) + \bi{i}(x) \phihi{i}(x). \label{def_dhi}
 \addtocounter{equation}{1}\tag{\theequation-\texttt{$d_{h,i}$\footnotesize{def}}}
\end{equation}
We define the space-time function \hypertarget{def:dh}{$\ddh : \QGn \to \mathbb{R}$} correspondingly.

\subsubsection{Accuracy of element-wise approximation of \texorpdfstring{$\Psi^\Gamma$}{ideal mapping}}

We aim to prove a variant of \cite[Lemma 3.5]{LR_IMAJNA_2018}.
To this end, we start by validating the proximity of $\ET \phihi{i}$ to $\phihi{i}$ in a neighborhood of each element $T \in \ThG$.
\begin{lemma} \label{diff_phi_phihn_spaceD_special}
 Fix $T\in\ThG$ and $i\in\{0,\dots,q_t\}$ and let $y := x + \alpha \Ghi{i}(x)$ be given such that $| x - y | \lesssim h$ (which is the case if and only if $| \alpha| \lesssim h$). Then, we have for $m=0,1$
 \begin{subequations}
 \begin{align} \label{eq:diff_phi_phihn_spaceD_special:1}
  | D^m (\ET \phihi{i} - \phii{i}) (y) | & \lesssim h^{q_s + 1 - m} + \Delta t^{q_t +1},
  \addtocounter{equation}{1}\tag{\theequation-\texttt{$\mathcal{E}_T$\footnotesize{acc}}}
  \\
  | D^m \partial_t (\ET \phih - \phix) (y,t) | & \lesssim h^{q_s + 1 - m} + \Delta t^{q_t}, \quad t \in \In{n}
  \addtocounter{equation}{1}\tag{\theequation-\texttt{$\mathcal{E}_T$\footnotesize{acc}}}
  \label{eq:diff_phi_phihn_spaceD_special:1b}
  \\
    | D^m (\ET \phihi{i} - \phidti{i})(y) | & \lesssim h^{q_s + 1 - m}.  \label{eq:diff_phi_phihn_spaceD_special:2}
\addtocounter{equation}{1}\tag{\theequation-\texttt{$\mathcal{E}_T$\footnotesize{acc}}}
\end{align}
\end{subequations}
\end{lemma}
\begin{proof}
  The proof is similar to a result in \cite[A. 1]{LR_IMAJNA_2018}. We provide it in \cref{app:diff_phi_phihn_spaceD_special}. 
  \end{proof}
Now, we come to the counterpart of \cite[Lemma 3.5]{LR_IMAJNA_2018}:
\begin{lemma} \label{d_h_lemma}
 For all $i=0,\dots,q_t$, and $h$ sufficiently small, \cref{def_dhi} defines a unique $\dhi{i}: \OG \to \mathbb{R}$ with $\dhi{i} \in C(\OG)\cap C^\infty(\ThG)$. Furthermore: 
 \begin{subequations}
 \begin{align}
  &\dhi{i}(x_V) = 0 \quad \textnormal{ for all vertices } x_V \textnormal{ of } T \in \ThG, 
  \addtocounter{equation}{1}\tag{\theequation-\texttt{$d_{h}$\footnotesize{vrt}}}
  \\
  & \| \dhi{i} \|_{\infty, \OG} \lesssim h^2 + \Delta t^{q_t+1}, \quad \| \nabla \dhi{i} \|_{\infty, \ThG} \lesssim h + \Delta t^{q_t+1} . 
  \label{d_h_lemma_eq2}
  \addtocounter{equation}{1}\tag{\theequation-\texttt{$d_{h,i}$\footnotesize{bnd}}}
 \end{align}
\end{subequations}
\end{lemma}
\begin{proof}
 The first result of $\dhi{i}$ vanishing on vertices follows from the construction of $\philini{i}$ as the vertex interpolant of $\phihi{i}$. Next, we show the existence of a unique solution $\dhi{i}$ of \cref{def_dhi} for $i\in\{0,\dots,q_t\}$. We define, for a fixed $\alpha_0 > 0$ at the point $x \in \OG$, the function $g\colon [-\alpha_0 h, \alpha_0 h] \to \mathbb{R}$ as follows:
 \begin{equation}
  g(\alpha) := \ET \phihi{i} (x + \alpha \Ghi{i}(x)) - (1-\bi{i}(x)) \philini{i}(x) - \bi{i}(x) \phihi{i}(x)
 \end{equation}
 We observe 
 at $x$ as argument for $\phii{i}$, $\bi{i}$, $\philini{i}$ and $\phihi{i}$, respectively,
 \begin{align*}
   | \phii{i} - (1-\bi{i}) \philini{i} - \bi{i} \phihi{i} |
  =| (1-\bi{i}) \left(\phii{i} -\philini{i}\right) - \bi{i} \left(\phihi{i} - \phii{i}\right) |
  \lesssim h^2 + \Delta t^{q_t+1},
 \end{align*}
 due to \eqref{diff_phi_phihn_spaceD}, \eqref{diff_grad_phi_philin_dt} and $\bi{i} \in [0,1]$.
 In combination with 
 \eqref{eq:diff_phi_phihn_spaceD_special:1}, we obtain 
 $$
 g(\alpha) = \phii{i}(x + \alpha \Ghi{i}(x)) - \phii{i}(x) + \mathcal{O}(h^2 + \Delta t^{q_t+1}).
 $$
 Next, we exploit proximity of $\phii{i}$ and $\phihi{i}$ w.r.t. both arguments, $x$ and $x + \alpha \Ghi{i}(x)$ using \eqref{diff_phi_phihn_spaceD}, so that 
 $$
 g(\alpha) = \phihi{i}(x + \alpha \Ghi{i}(x)) - \phihi{i}(x) + \mathcal{O}(h^2 + \Delta t^{q_t+1}).
 $$
 We can now apply a Taylor development as in \cref{g_in_first_lemma_proof_Taylor_expanded} to obtain:
 \begin{align}
  g(\alpha) &= \alpha | \nabla \phihi{i}(x) |^2 + \mathcal{O}(h^2 + \Delta t^{q_t + 1}) 
 \end{align}
 Note that $\dhi{i}(x)$ is $\alpha$ so that $g(\alpha)=0$.
 We now apply a similar argument to the one in the proof of \cref{g_in_first_lemma_proof_Taylor_expanded}.
 Involving \cref{delta_t_smallness} to bound the remainder term by $\mathcal{O}(h)$, noting that $| \nabla \phihi{i}(x) |$ is bounded from below and above for sufficiently small $h$, we can easily see that $g(\pm \alpha_0 h) \lessgtr 0$ for sufficiently small $h$. As $g$ is continuous (more specifically a linear function with a higher-order perturbation), we can apply the intermediate value theorem to conclude that there exists a unique $\alpha \in [-\alpha_0 h, \alpha_0 h]$ such that $g(\alpha) = 0$. Further, it satisfies the first estimate in \cref{d_h_lemma_eq2}.
 The second estimate in \cref{d_h_lemma_eq2} is proven in \cref{app:d_h_lemma_eq2}.
\end{proof}

%
In terms of $\dhi{i}(x)$, we then define
\begin{equation} \label{eq:def:PsihG}
 \hypertarget{def:PsihGi}{\PsihGi{i}(x) \coloneqq x + (\dhi{i} \Ghi{i})(x)}, \quad \quad \hypertarget{def:PsihG}{\PsihG(x,t) = \sum_{i=0}^{q_t} \elli{i}(t) \PsihGi{i} (x)}.
 \addtocounter{equation}{1}\tag{\theequation-\texttt{$\Psi_{h}^\Gamma$\footnotesize{def}}}  
\end{equation}

\begin{lemma} There holds \label{lem:PsidtG-PsihG}
  \begin{equation}
    \Vert \PsidtG - \PsihG \Vert_{\infty,\QGn} \lesssim h^{q_s+1}  \label{eq:PsidtG-PsihG}
    \addtocounter{equation}{1}\tag{\theequation-\texttt{$\Psi_{\Delta t\mid h}^\Gamma$}}  
  \end{equation}
\end{lemma}
\begin{proof}
  We note that it suffices to prove
  $\Vert \PsidtGi{i} - \PsihGi{i} \Vert_{\infty,\OG} \lesssim h^{q_s+1}$ for $i=0,\dots,q_t$. Then the following is very similar to the proof of \cite[Lemma 3.6]{LR_IMAJNA_2018} given in \cite[A.2]{LR_IMAJNA_2018}.
  By the definitions of $\PsihGi{i}$ and $\PsidtGi{i}$, we have 
  for $x \in \OG$
  \begin{align}
    | (\PsidtGi{i} - \PsihGi{i})(x) | &= | (\ddti{i} \Gdti{i} - \dhi{i} \Ghi{i})(x) | \nonumber\\
   &~~\leq |(\ddti{i} - \dhi{i})(x)| ~ |\Gdti{i}(x)| + |\dhi{i}(x)| ~ |(\Gdti{i} - \Ghi{i})(x) | \label{PsidtG-PsihG:1}
  \end{align}
  Here, the second summand can be bounded by $h^{q_s+1}$ with \eqref{diff_phid_phih:i} and $m_s=1$ from \cref{diff_phidt_phihn} and \eqref{d_h_lemma_eq2} from \cref{d_h_lemma} together with \eqref{delta_t_smallness} so that $|\dhi{i}(x)|\lesssim h$. 
  In the first summand, $|\Gdti{i}(x)|$ is bounded $\lesssim 1$ due to \eqref{diff_phid_phi} and \eqref{eq:cnphi}, so that it remains to estimate $|(\ddti{i} - \dhi{i})(x)|$.
 
  From the definitions of $\ddti{i}$ and $\dhi{i}$, we obtain
  \begin{align}
   \phidti{i}(x + (\ddti{i} \Gdti{i})(x)) &= (1-\bi{i}(x)) \philini{i}(x) + \bi{i}(x) \phidti{i}(x) \nonumber \\ 
   \ET \phihi{i} (x + (\dhi{i}\Ghi{i})(x)) &= (1-\bi{i}(x)) \philini{i}(x) + \bi{i}(x) \phihi{i}(x), \nonumber \\
   \Rightarrow \phidti{i}(x + (\ddti{i}\Gdti{i})(x)) & = 
   \ET \phihi{i} (x + (\dhi{i} \Ghi{i})(x)) + \mathcal{O}(h^{q_s+1}) \label{phidt-phiET}
  \end{align}
  where we made use of $m_s=0$ in \eqref{diff_phid_phih:i} from \cref{diff_phidt_phihn}.  
  We set 
  $y_h\coloneqq\PsihGi{i}(x)$
  and $\tilde y_h\coloneqq x + (\dhi{i} \Gdti{i})(x)$. Then, with \cref{signed_dist_like_prop_dt}, it follows:
  \begin{align*}
   |\ddti{i}&(x) - \dhi{i}(x)| \simeq | \phidti{i}(x + (\ddti{i} \Gdti{i})(x)) - \phidti{i}(x + (\dhi{i}\Gdti{i})(x))| \\
   & \stackrel{\eqref{phidt-phiET}}{=} | \ET \phihi{i} (\overbrace{x + (\dhi{i} \Ghi{i})(x)}^{=y_h}) - \phidti{i}(\overbrace{x + (\dhi{i}\Gdti{i})(x)}^{=\tilde y_h})| + \mathcal{O}(h^{q_s+1}) \\
   & \stackrel{\hphantom{\eqref{phidt-phiET}}}{\leq} | (\ET \phihi{i} - \phidti{i}) (y_h)| +  | \phidti{i}(y_h) - \phidti{i}(\tilde y_h) | + 
   \mathcal{O}(h^{q_s+1})
   \end{align*}
   Exploiting \eqref{eq:diff_phi_phihn_spaceD_special:2} in \cref{diff_phi_phihn_spaceD_special} the first summand in this expression can be bounded by $\lesssim h^{q_s + 1}$.
   For the second summand, we have $\| \dhi{i}\|_{\infty,\OG} \lesssim h^2 + \Delta t^{q_t+1}$ with \eqref{d_h_lemma_eq2} from \cref{d_h_lemma} and $\| \Ghi{i} - \Gdti{i}\|_{\infty,\OG} \lesssim h^{q_s}$ from $m_s=1$ in \eqref{diff_phid_phih:i} from \cref{diff_phidt_phihn} so that $|y_h - \tilde y_h| \lesssim h^{q_s+2} + \Delta t^{q_t+1} h^{q_s} \lesssim h^{q_s+1}$ where we used \eqref{delta_t_smallness} in the last step.
   With the boundedness $\| \nabla \phidt \|_{\infty,U} \lesssim \| \nabla \phix\|_{\infty,U} \lesssim \Cnphi{}$, cf. \eqref{eq:cnphi}, we hence have 
    \begin{align*}
    | \phidti{i}(y_h) - \phidti{i}(\tilde y_h)|\lesssim \|\nabla \phidti{i}\|_{\infty, \Ubar} |y_h - \tilde y_h| \lesssim h^{q_s + 1}.
  \end{align*}
  Taking these results together, we have
  \begin{equation}
   \| \ddti{i} - \dhi{i}\|_{\infty,\OG} \lesssim h^{q_s + 1} . \label{ddt_d_hi_diff_bound}
  \end{equation}
  With \eqref{PsidtG-PsihG:1} and the previous considerations, this implies \eqref{eq:PsidtG-PsihG}.
\end{proof}

Each of these mappings $\PsihGi{i}$ ($i=0,\dots,q_t$) is close to the restriction of the mapping from the previous subsection, $\PsiGi{i}$, as is stated in the following:
\begin{lemma} \label{Psi_hi_Psi_diff_lemma} 
 For all $i=0,\dots,q_t$ and $h$ sufficiently small,
 \begin{align}
  \| \PsihG - \PsiG \|_{\infty,\QGn} & \lesssim h^{q_s+1} + \Delta t^{q_t + 1}, \label{eq:PsihG-PsiG}
  \addtocounter{equation}{1}\tag{\theequation-\texttt{$\Psi_h^\Gamma$\footnotesize{acc}}}  
 \end{align} 
\end{lemma}
\begin{proof}
  The statement follows from a triangle inequality
  $\| \PsihG \!-\! \PsiG \|_{\infty,\QGn} \leq \| \PsihG \!-\! \PsidtG \|_{\infty,\QGn} + \| \PsidtG \!- \PsiG \|_{\infty,\QGn}$
  and \eqref{eq:PsidtG-PsihG} and \eqref{eq:PsidtG-PsiG}.
\end{proof}

\subsubsection{Continuous approximation of \texorpdfstring{$\Psi_h^\Gamma$}{ideal mapping} on active elements}

Note that $\PsiG$ is not necessarily continuous in space across element boundaries. Next, we define a spatially continuous version of $\PsihG$, $\ThehG$, by applying the spatial interpolation operator $\PhG$ \hyperlink{def:PhG}{as before} 
\begin{equation}\label{eq:def:ThehG}
  \hypertarget{def:ThehG}{
 \ThehG(x,t) \coloneqq \PhG \PsihG = \sum_{i=0}^{q_t} \elli{i}(t) (\PhG \PsihGi{i}(x)). }
 \addtocounter{equation}{1}\tag{\theequation-\texttt{$\Theta_h^\Gamma$\footnotesize{def}}}    
\end{equation}
With respect to the application of the spatial interpolation operator on $\PsihG$ as well as on $\PsidtG$, we easily find the following results:
\begin{lemma}\label{lem:PhGPsidtG}
  For $m_s \in \{0,1\}$ there holds
  \begin{align}
\Vert D^{m_s} (\PhG \PsidtG - \PsidtG) \Vert_{\infty,\ThGQ} &\lesssim h^{q_s+1-m_s}, \label{eq:PhGPsidtG}
\addtocounter{equation}{1}\tag{\theequation a-\texttt{$P_h^\Gamma\Psi_{\Delta t}^\Gamma|\Psi_{\Delta t}^\Gamma$}} \\
\| D^{m_s} (\ThehG - \PhG \PsidtG)\|_{\infty, \QGn} & \lesssim h^{q_s+1-m_s},
\addtocounter{equation}{0}\tag{\theequation b-\texttt{$\Theta_h^\Gamma|P_h^\Gamma\Psi_{\Delta t}$}}
\label{eq:ThetahPsih}\\
\text{and hence \quad }\| D^{m_s} (\ThehG - \PsidtG)\|_{\infty, \QGn} & \lesssim h^{q_s+1-m_s}.
\addtocounter{equation}{0}\tag{\theequation c-\texttt{$\Theta_h^\Gamma|\Psi_{\Delta t}$}}
\label{eq:ThetahPsidt}
\end{align}
\end{lemma}
\begin{proof}
\underline{Proof of \eqref{eq:PhGPsidtG}:}\\
We have that $\PsidtG$ is continuous in space so that $\PhG$ becomes a classical interpolation operator $\PhG \PsidtG = \Is{q_s} \PsidtG$ and we can apply standard interpolation results 
$$
  = \| D^{m_s} (\Is{q_s} \PsidtG - \PsidtG)\|_{\infty, \QGn} \lesssim h^{q_s+1-m_s} \| D^{q_s+1} \PsidtG \|_{\infty,\ThGQ} \lesssim h^{q_s+1-m_s}
$$
where we bounded the latter norm by $1$ due to \cref{lem:ddt}. \\
\underline{Proof of \eqref{eq:ThetahPsih}:}\\
For $m_s=1$ we note that the gradient is applied to a spatially discrete function and we can apply an inverse inequality 
\begin{align*}
  \| D^{m_s} (\ThehG \!-\! \PhG \PsidtG)\|_{\infty, \QGn}  
   = \| D^{m_s} \PhG (\PsihG \!-\! \PsidtG)\|_{\infty, \QGn}  
  \lesssim h^{-m_s} \| \PhG (\PsihG \!-\! \PsidtG)\|_{\infty, \QGn}  
\end{align*}
which is also trivially true for $m_s=0$. Now applying continuity of the Oswald-interpolator we arrive at
\begin{align*}
  \| D^{m_s} (\ThehG - \PhG \PsidtG)\|_{\infty, \QGn}  
\lesssim h^{-m_s} \| \PsihG - \PsidtG \|_{\infty, \ThGQ}
\end{align*}
Together with \eqref{eq:PsidtG-PsihG} in \cref{lem:PsidtG-PsihG} we hence obtain the claim. \\
\eqref{eq:ThetahPsidt} then follows by a triangle inequality.
\end{proof}

We can now put the results for the time instances $\ti{i}$, $i=0,\dots,q_t$ together to obtain the following result in space-time:
\begin{theorem}\label{lem:Theta_h_Psi_Gamm_diff_bound}
  For $m_s=0,\dots,q_s+1$ and $m_t=0,\dots,q_t + 1$ the following holds: 
\begin{align}
  \Delta t^{m_t} h^{m_s} \| D^{m_s} \partial_t^{m_t} (\ThehG - \PsiG) \|_{\infty, \ThGQ} &\lesssim h^{q_s + 1} + \Delta t^{q_t+1}, 
  \label{Theta_h_Psi_Gamm_diff_bound_Dr_dt} 
  \addtocounter{equation}{1}\tag{\theequation a-\texttt{$\Theta_h^\Gamma$\footnotesize{acc}}}    
  \\
   \| \nabla (\ThehG - \PsiG) \|_{\infty, \ThGQ} &\lesssim h^{q_s} + \Delta t^{q_t + 1} \label{Theta_h_Psi_Gamm_diff_bound_first_grad}.
   \addtocounter{equation}{0}\tag{\theequation b-\texttt{$\Theta_h^\Gamma$\footnotesize{acc}}}    
  \end{align}
\end{theorem}
Let us stress that the second claim is stronger than the case $m_s=1$ and $m_t=0$ in the first estimate. Only restrictions as strong as $\Delta t \lesssim h$ would let the first estimate imply the second. 
Below, in \cref{lem:Theta_h_Psi_Gamm_diff_bound_first_dt} -- after having introduced another semi-discrete mapping -- we will deduce an analog to \eqref{Theta_h_Psi_Gamm_diff_bound_first_grad} for the time derivative (using \cref{diff_phiH_phihn}) which could only be implied from the first claim for $h \lesssim \Delta t$.
\begin{proof}
  We prove both results one after another.

\underline{Proof of \eqref{Theta_h_Psi_Gamm_diff_bound_Dr_dt}:}\\
With $\PhG \PsiG = \Is{q_s} \PsiG$
and hence $\It{q_t} \PhG \PsiG = \It{q_t} \Is{q_s}  \PsiG$ 
we can apply a triangle inequality to obtain one discrete (in   time and space) difference $\PhG (\PsihG - \It{q_t}\PsiG)$ and an interpolation error (in time and space): 
\begin{align}
  &\Delta t^{m_t} h^{m_s}  \| \partial_t^{m_t} D^{m_s}  (\PhG \PsihG  - \PsiG) \|_{\infty, \ThGQ}  \nonumber\\
  \leq& 
  \Delta t^{m_t} h^{m_s} (\| \partial_t^{m_t} D^{m_s} \PhG (\PsihG - \It{q_t} \PsiG) \|_{\infty, \ThGQ} + \| \partial_t^{m_t} D^{m_s} (\It{q_t} \Is{q_s} \PsiG - \PsiG) \|_{\infty, \ThGQ}) \!\lesssim\! \dots \nonumber 
\end{align}
Applying an inverse inequality on the discrete (in time and space) term, exploiting continuity of $\PhG$ and standard  tensor product interpolation gives
 \begin{align}
  \dots & \lesssim \| \PsihG - \It{q_t} \PsiG \|_{\infty, \ThGQ} + (h^{q_s + 1} + \Delta t^{q_t + 1}) \|\PsiG\|_{H^{q_s + 1,q_t+1,\infty}(\ThGQ)}\nonumber \\ 
  & \lesssim \| \PsihG - \PsiG \|_{\infty, \ThGQ} + \| \It{q_t} \PsiG - \PsiG \|_{\infty, \ThGQ} + (h^{q_s + 1} + \Delta t^{q_t + 1}) \|\PsiG\|_{H^{q_s + 1,q_t+1,\infty}(\ThGQ)} \nonumber 
\end{align}
Now, making use of \eqref{eq:PsihG-PsiG}, standard temporal interpolation with 
the bounds from \eqref{Psi_Gamma_dt_ho_derivs_bounded} and \eqref{Psi_Gamma_dt_ho_derivs_bounded} for the last term we obtain the claim.

\underline{Proof of \eqref{Theta_h_Psi_Gamm_diff_bound_first_grad}:}\\
To circumvent using the inverse inequality on terms that involve temporal resolution quantities (and hence avoiding quotients of the form ${\Delta t^{l}}/{h}$ for some $l\in\mathbb{N}$) we start with a triangle inequality that involves $\PhG \PsidtG$ and $\PsidtG$:
\begin{align*}
  \| & \nabla (\ThehG - \PsiG)\|_{\infty, \QGn}  
 =  \| \nabla (\PhG \PsihG - \PsiG)\|_{\infty, \QGn}  \\
& \leq  \underbrace{\| \nabla \PhG (\PsihG - \PsidtG)\|_{\infty, \QGn}}_{A}  
+  \underbrace{\| \nabla (\PhG \PsidtG - \PsidtG)\|_{\infty, \QGn}}_{B}  
+  \underbrace{\| \nabla (\PsidtG - \PsiG)\|_{\infty, \QGn}}_{C} \\[-6ex]
\end{align*} 
For the first summand $A$,
we apply \eqref{eq:ThetahPsih}, for the second summand $B$ obtains a proper bound from \eqref{eq:PhGPsidtG} with $m_s=1$ and the last summand $C$ we bound by \eqref{eq:grad-PsidtG-PsiG},
$C = \| \nabla (\PsidtG - \PsiG)\|_{\infty, \QGn} \lesssim h^{q_s} + \Delta t^{q_t+1}$.
Putting the bounds for $A$, $B$ and $C$ together yields the claim. 
%
\end{proof}

\subsection{Construction of discrete-in-space mapping on active elements} 
\label{sec:discrete-in-space-mapping}
In this section, we define a semi-discrete mapping that is continuous in time and discrete in space. Throughout this section, we assume that \cref{diff_phiH_phihn} is valid.

For the discrete search direction, we make the same choice as in the fully discrete setting, i.e. we define either $\GH = \nabla \phiH$ or $\GH = (\PhG \nabla \phiH)$.

We define the function $\dH \colon \QGn \to [-\delta,\delta]$ for $\delta > 0$ sufficiently small as follows: For $x \in \OG$, $t\in \In{n}$, let \hypertarget{def:dH}{$\dH(x,t)$ be the (in absolute value) smallest number such that}
\begin{equation}
 \ET \phiH (x + (\dH \GH)(x,t),t) = (1-\b(x,t)) \philin(x,t) + \b(x,t) \phiH(x,t). \label{def_dH}
 \addtocounter{equation}{1}\tag{\theequation-\texttt{$d_{H}$\footnotesize{def}}}
\end{equation}

In analogy to \eqref{eq:diff_phi_phihn_spaceD_special:1} in \cref{diff_phi_phihn_spaceD_special} and \eqref{d_h_lemma_eq2} in \cref{d_h_lemma}, we have the following result for $\dH$:
\begin{lemma} \label{d_H_lemma}
  \begin{subequations}
    For $T \in \ThG$, $t\in\In{n}$ and $y \coloneq x + \alpha (\GH)(x,t)$, $|\alpha| \lesssim h$ there holds
    \begin{align}
      | \partial_t^{m_t} (\ET \phiH - \phix) (y,t) | & \lesssim h^{q_s + 1} + \Delta t^{q_t +1 - m_t}, \quad m_t=0,1.
      \addtocounter{equation}{1}\tag{\theequation-\texttt{$\mathcal{E}_T$\footnotesize{acc}}}
      \label{diff_ETphiH_phi}
\intertext{  
    For $h$ sufficiently small, the relation \cref{def_dH} defines a unique $\dH \in C(\QGn)$ with
}    
     \| \dH \|_{\infty, \QGn} & \lesssim h^2 + \Delta t^{q_t+1}. 
   \label{d_H_lemma_eq2}
   \addtocounter{equation}{1}\tag{\theequation-\texttt{$d_{H}$\footnotesize{bnd}}} \\
   \| \partial_t^{m_t} \dH \|_{\infty, \QGn} & \lesssim 1, \quad m_t=0,..,q_t+1. 
   \label{d_H_lemma_eq3}
   \addtocounter{equation}{1}\tag{\theequation-\texttt{$d_{H}$\footnotesize{bnd}}}
  \end{align}
 \end{subequations}
 \end{lemma}
 \begin{proof}
  The proof of \eqref{diff_ETphiH_phi} follows the same lines of the proof of \eqref{eq:diff_phi_phihn_spaceD_special:1} and \eqref{eq:diff_phi_phihn_spaceD_special:1b} with $\phihi{i}(\cdot)$ replaced by $\phiH(\cdot,t)$ and exploiting \eqref{diff_phiH_phi} instead of \eqref{diff_phi_phihn_spaceD}. 
  The proof of  \eqref{d_H_lemma_eq2} follows the same lines of the proof of \cref{d_h_lemma} when freezing $t \in \In{n}$ and exchanging $\phihi{i}(\cdot)$ with $\phiH(\cdot,t)$, $\Ghi{i}(\cdot)$ with $\GH(\cdot,t)$
  and $\philini{i}(\cdot)$ with $\philin(\cdot,t)$. These requirements \eqref{diff_phi_phihn_spaceD} and \eqref{eq:diff_phi_phihn_spaceD_special:1} are replaced by 
  \eqref{diff_phiH_phi} and \eqref{d_H_lemma_eq2}. Note that in contrast to \cref{d_h_lemma} and \cref{diff_phi_phihn_spaceD_special} we do not consider spatial derivatives of $\dH$ or $\ET \phiH$.
  \eqref{d_H_lemma_eq3} follows with similar arguments as in the proof of \eqref{eq:db} (for $m_t>1$) in \cref{app:eqdbmtgt1}. For completeness, we have included the proof in \cref{sec:proof:d_H_lemma_eq3}.
 \end{proof}

 In terms of $\dH$, we then define
\begin{equation}\label{eq:def:PsiHG}
\hypertarget{def:PsiHG}{\PsiHG(x,t) = x + (\dH \GH)(x,t)}.
\addtocounter{equation}{1}\tag{\theequation-\texttt{$\Psi_{H}^\Gamma$\footnotesize{def}}}
 \end{equation}
 In the next two lemmas, we will show proximity to $\PsihG$ and $\PsiG$, including bounds for the time derivative.

 \begin{lemma} There holds for $m_t \in \{0,1\}$ \label{lem:PsiHG-PsihG}
  \begin{align}
    \Vert \partial_t^{m_t} (\PsiHG - \PsihG) \Vert_{\infty,\QGn} & \lesssim \Delta t^{q_t+1-m_t}.  \label{eq:PsiHG-PsihG}
    \addtocounter{equation}{1}\tag{\theequation-\texttt{$\Psi_{H\mid h}^\Gamma$}}  
  \end{align}
\end{lemma}
\begin{proof}
  \underline{Proof of for $m_t=0$:}\\
  By the definitions of $\PsihG$ and $\PsiHG$, we have 
  for $x \in \OG,~t\in\In{n}$
  \begin{align}
    | (\PsiHG - \PsihG)(x,t) | &= | (\dH \GH - \ddh \Gh)(x,t) | \nonumber\\
   &~~\leq |(\dH - \ddh)(x,t)| ~ |\Gh(x,t)| + |\dH(x,t)| ~ |(\GH - \Gh)(x,t) | \label{PsiHG-PsihG:1}
  \end{align}
  The individual terms can now be bounded similarly to the procedure of the proof of \cref{lem:PsidtG-PsihG}.
  $|(\dH - \ddh)(x,t)| \lesssim \Delta t^{q_t+1}$ follows effectively from \eqref{diff_phiH_phih}, $|\Gh| \lesssim 1$ follows from $|\G| \lesssim 1$ and proximity, $|\dH| \lesssim h^2 + \Delta t^{q_t+1} \lesssim h$ follows from \eqref{d_H_lemma_eq2} and 
  \eqref{delta_t_smallness}  
  and $|(\GH - \Gh)| \lesssim \frac{1}{h} \Delta t^{q_t+1}$ follows from \eqref{diff_phiH_phih} and an inverse inequality. Inserting these bounds into \eqref{PsiHG-PsihG:1} yields the claim.

  \underline{Proof of for $m_t=1$:}\\
  We shift in the time interpolant of $\PsiHG$, $\It{q_t} \PsiHG$ \footnote{which is not necessarily the same as $\PsihG$}, and apply a triangle inequality to obtain \vspace*{-0.5cm}
  \begin{align*}
    \| \partial_t (\PsihG - \PsiHG)\|_{\infty, \ThGQ} 
    \leq 
    \overbrace{\| \partial_t (\PsihG - \It{q_t} \PsiHG)\|_{\infty, \ThGQ}}^{A}
    +
    \overbrace{\| \partial_t (\It{q_t} \PsiHG - \PsiHG)\|_{\infty, \ThGQ}}^{B}.
  \end{align*}
  Next, for $A$ we can apply an inverse inequality 
  \begin{align*}
   A & = \| \partial_t (\PsihG - \It{q_t} \PsiHG)\|_{\infty, \ThGQ} 
   \lesssim 
   \Delta t^{-1} \| (\PsihG - \It{q_t} \PsiHG)\|_{\infty, \ThGQ} 
   \\ 
   & 
   \lesssim 
   \Delta t^{-1} \max_{i} \| \PsihGi{i} - \PsiHG(\cdot,t_i)\|_{\infty, \OG}
   \lesssim 
   \Delta t^{-1}  \| \PsihG - \PsiHG\|_{\infty, \ThGQ} \lesssim \Delta t^{q_t}
  \end{align*}
  where we exploited the claim for $m_t=0$. For $B$ we have 
    \begin{align*}
    \| \partial_t (\It{q_t} \PsiHG - \PsiHG)\|_{\infty, \ThGQ} 
    \lesssim
    \Delta t^{q_t} \| \partial_t^{q_t+1} \PsiHG \|_{\infty, \ThGQ} \lesssim \Delta t^{q_t}
  \end{align*}
  where the last inequality follows from the Leibniz rule
  $$
  \| \partial_t^{q_t+1} \PsiHG \|_{\infty, \ThGQ} \lesssim \textstyle \sum_{k=0}^{q_t+1} \| \partial_t^{q_t+1-k} \dH \|_{\infty, \ThGQ} \| \partial_t^{k} \GH \|_{\infty, \ThGQ}
  $$
  and the bounds \eqref{d_H_lemma_eq3} and \eqref{diff_phiH_phi}.
\end{proof}

\begin{lemma} \label{Psi_H_Psi_diff_lemma} 
  For all $i=0,\dots,q_t$ and $h$ sufficiently small,
  \begin{align}
    \| \partial_t (\PsiHG - \PsiG) \|_{\infty,\QGn} & \lesssim h^{q_s+1} + \Delta t^{q_t}, \label{eq:PsihG-PsiG-dt}
   \addtocounter{equation}{1}\tag{\theequation-\texttt{$\Psi_H^\Gamma$\footnotesize{acc}}}  
  \end{align} 
 \end{lemma}
 \begin{proof}
  The proof is given in \cref{app:Psi_H_Psi_diff_lemma}.
 \end{proof}

Finally, we can use the mapping $\PsiHG$ as a proxy to obtain the missing time derivative bound for $\ThehG$ compared to $\PsiG$:
 \begin{theorem} \label{lem:Theta_h_Psi_Gamm_diff_bound_first_dt}
  For $m_s=0,\dots,q_s+1$ and $m_t=0,\dots,q_t + 1$ there holds: 
\begin{align}
 \| \partial_t (\ThehG - \PsiG) \|_{\infty, \ThGQ} & \lesssim h^{q_s+1} + \Delta t^{q_t} \label{Theta_h_Psi_Gamm_diff_bound_first_dt}.
   \addtocounter{equation}{0}\tag{\theequation-\texttt{$\Theta_h^\Gamma$\footnotesize{acc}}}    
  \end{align}
\end{theorem}
\begin{remark}
We note that this bound is an improved bound for \eqref{Theta_h_Psi_Gamm_diff_bound_Dr_dt} with $m_s=0$ and $m_t=1$ as only for $h \lesssim \Delta t$ \eqref{Theta_h_Psi_Gamm_diff_bound_first_dt} follows from \eqref{Theta_h_Psi_Gamm_diff_bound_Dr_dt} otherwise.
\end{remark}
\begin{proof}
  First, we want to work around the Oswald projection using a triangle inequality and the continuity of the Oswald projection:
  \begin{align*}
    \| \partial_t (\ThehG - \PsiG)\|_{\infty, \QGn}  
   & =  \| \partial_t (\PhG \PsihG - \PsiG)\|_{\infty, \QGn}  \\[-3ex]
  & \leq \| \partial_t \PhG (\PsihG - \PsiG)\|_{\infty, \QGn}
   + \underbrace{\| \overbrace{\PhG \partial_t \PsiG}^{= \Is{q_s} \partial_t \PsiG} - \partial_t \PsiG \|_{\infty, \QGn}}_{\lesssim h^{q_s+1} \| D^{q_s+1} \partial_t \PsiG \|_{\infty, \QGn}}  \\
   & \lesssim \| \partial_t (\PsihG - \PsiG)\|_{\infty, \ThGQ} + h^{q_s+1}  \quad \text{(with } \eqref{PhG_prop} \text{ and } \eqref{Psi_Gamma_dt_ho_derivs_bounded} \text{)}
  \end{align*}
  The bound for $\| \partial_t (\PsihG - \PsiG)\|_{\infty, \ThGQ}$
  follows from a triangle inequality and the bounds in \cref{lem:PsiHG-PsihG,Psi_H_Psi_diff_lemma}.
\end{proof}

\section{Global space-time mappings} \label{sec:global_mappings}
In this section, we want to discuss the extension of the previously defined functions $\PsiG$ and $\ThehG$ onto the whole computational domain $\tOmega \times \In{n}$. We recap that so far, both functions are defined on smaller sets $\QGn$, whose specific structure depends on the choice of the \FE~ or smooth blending.
\subsection{Extension procedure for the smooth function blending}
The task of the extension is particularly straightforward for the smooth blending case because the blending function already ensures that the functions $\PsiG$ and $\ThehG$ smoothly transition towards identity on the (spatial) boundary of $\QGn$. Hence, we can give the following definition of the extension $\mathcal{E} = \mathcal{E}_2$ for the smooth blending case for $t \in I_n$:
\begin{equation} \label{E2_intro_eq}
 (\mathcal{E}_2 \PsiG) (x,t) := \begin{cases}
                                              \PsiG(x,t) &\textnormal{ if } x \in \OG \\
                                              \idx &\textnormal{ else}.
                                             \end{cases}, 
                                             \quad \hypertarget{def:Psi}{\Psii := \mathcal{E}\PsiG }, 
                                             \quad \hypertarget{def:Theh}{\Theh := \mathcal{E} \ThehG.}
\end{equation}
We continue with the construction of the according $\mathcal{E}_1$ for the \FE~blending.
\subsection{Extension procedure for FE blending} \label{sec:extension_procedure}
We now introduce the extension procedure for the \FE~blending procedure. Hence, for the following assume that $\OG$ is the domain of the active elements and the input functions $\PsiG$ and $\ThehG$ do not transition into identity at $\partial \QGn$. Otherwise the next steps and results in this section become trivial\footnote{This is the case for the smooth blending.}.

As a first technical step, we introduce notation for all elements neighbouring $\OG$: \vspace*{-0.4cm}
\begin{subequations}
\begin{align}
  \hypertarget{def:ThGp}{
    \ThGp }&:= \{ T \in \Th \, | \, \overline{T} \cap \overline{\OG} \neq \varnothing \}, &
   \hypertarget{def:OGp}{
    \quad \OGp} & := \bigcup \ThGp,
    \\
  \hypertarget{def:ThGQp}{
     \ThGQp }&:= \{ T \times \In{n} \mid T \in \ThGp \}, &
 \hypertarget{def:QGnp}{
  \quad \QGnp} &:= \OGp \times \In{n}. 
\end{align}
\end{subequations}
These are the elements that share at least one vertex with elements in the cut domain $\ThG$. On all the elements $T \in \ThGp \backslash \ThG$, we will apply a blending step which realises a continuous transition from the given values on $\partial \OG$ to the identity on $\partial \OGp$ by a construction involving barycentric coordinates on the reference element. It is described in detail in \cite[Section 3.3]{LR_IMAJNA_2018} as $\mathcal{E}^{\partial \OG}$ and we assume from here on that this linear extension operation is given w.r.t. our set of relevant elements. Central for the applicability of this mapping is the following lemma, which gives bounds on $\mathcal{E}^{\partial \OG} w$ and spatial derivatives thereof in terms of higher derivatives of $w$: \cite[Theorem 3.11]{LR_IMAJNA_2018}
\begin{lemma} \label{ext_bnd_lemma}
 Let \hypertarget{def:verts}{$\Verts(\partial \OG)$ denote the set of vertices in $\partial \OG$} and \hypertarget{def:facets}{$\Facets(\partial \OG)$} the set of all edges ($d=2$) or faces ($d=3$) in $\partial \OG$. The following estimate holds for $m_s =0,1$ and sufficiently regular $w$ s.t. the r.h.s. is well-defined and bounded:
 \begin{subequations}
 \begin{align}
  \!\!\!\! \| D^{m_s} \!\mathcal{E}^{\partial \OG}\!\! w \|_{\infty, \OGp \backslash \OG} \!\!\lesssim \! \sum_{r={m_s}}^{q_s+1} \! h^{r-{m_s}} \| D^r  w \|_{\infty,\Facets(\partial \OG)} \!+\! h^{-{m_s}} \|w\|_{\infty,\Verts(\partial \OG)}. \!\!\!\!\label{ext_bnd_lemma_lo}
  \addtocounter{equation}{1}\tag{\theequation-\texttt{$\mathcal{E}\!$\footnotesize{bnd}}}    
 \end{align}
 Moreover, for ${m_s}=2,\dots,q_s+1$, for $T \in \ThGp \backslash \ThG$ and $\Facets(T)$ the facets of $T$ on $\partial \OG$  \vspace*{-0.25cm}
 \begin{align}
  \| D^{m_s} \mathcal{E}^{\partial \OG} w \|_{\infty, T} \lesssim \sum_{r={m_s}}^{q_s+1} h^{r-{m_s}} \|D^r w \|_{\infty, \Facets(T)}. \label{ext_bnd_lemma_ho}
  \addtocounter{equation}{1}\tag{\theequation-\texttt{$\mathcal{E}\!$\footnotesize{bnd}}}    
 \end{align}
\end{subequations}
\end{lemma}
 The extension can be applied to any continuous, piecewise (sufficiently) smooth function $w$. 
 If the argument $w$ is piecewise (spatial) polynomial, $\mathcal{E}^{\partial \OG} w$ will be as well, cf. \cite[Lemma 3.9]{LR_IMAJNA_2018}.
Applying the $\mathcal{E}^{\partial \OG}$ to the mapping $\PsiG$, we define for $t \in \In{n}$
\begin{equation} \label{eq:intro_Psi_FE_blend}
 (\mathcal{E}_1 \PsiG) (x,t) = \begin{cases}
                                              \PsiG(x,t) &\textnormal{ if } x \in \OG \\
                                              \idx + \mathcal{E}^{\partial \OG}(\PsiG(\cdot,t) - \idx) &\textnormal{ if } x \in  \OGp \backslash \OG \\
                                              \idx &\textnormal{ else}.
                                             \end{cases}
\end{equation}
%
%
As a first step in the proving of
geometry approximation and interpolation results for $\Psii, \Theh$, we state a commutation property for the \emph{spatial} extension $\mathcal{E}^{\partial \OG}$.
\begin{lemma} \label{dt_E_commutation_lemma}
 For $u \in C^{m_t}(\In{n};C(\partial \OG)),~m_t\in\mathbb{N}$ time derivates and spatial extension commute, i.e.  $\partial_t^{m_t} \mathcal{E}^{\partial \OG} u(\cdot,t) = \mathcal{E}^{\partial \OG} \partial_t^{m_t} u(\cdot,t)$. 
\end{lemma}
\begin{proof}
 For the proof of this lemma, we refer to the construction of $\mathcal{E}^{\partial \OG}$ in \cite{LR_IMAJNA_2018}. 
\end{proof}
\begin{remark}[Continuity of space-time mesh deformation across time slabs]
  If the \FE~blending approach is used, the mesh deformation on subsequent time slabs may not be continuous in time as the set of active elements will change in general. This applies both to the mappings $\Psi$ and $\Theta_h$. In relation to $\Theta_h$, the change in the set of active elements will induce discontinuities both before and after the \FE~blending extension is applied. First, the Oswald operator involved in \Cref{eq:def:ThehG} might refer to different domains when seen from above and below in time. Hence, from the local perspective of one active element, and averaging on the boundary degrees of freedom might take place or not depending on whether the according neighbor element is active or not. Second, a similar situation occurs one layer further outside for both $\Theta_h$ and $\Psi$ with the extension mechanism in \Cref{eq:intro_Psi_FE_blend}, as one particular element might be, for instance, in the set of active elements when seen from one side and in the transition set of elements $\ThGp$ when seen from the other side. This will in general cause the resulting deformation function to be discontinuous.
  
  For the discretization of PDEs on the deformed meshes mesh transfer operations will become necessary. One approach for a mesh transfer operation has been proposed and analyzed for an unfitted time stepping method in \cite{LL_ARXIV_2021}.
  
  For the smooth blending approach the space-time mesh deformation will be continuous in time. The difference to the \FE~blending case is that the Oswald operator in \Cref{eq:def:ThehG} will only be applied to identity mappings on the boundary degrees of freedom of active element because of the structure of the blending function defining the set of active elements. Moreover, the trivial extension mechanism of \Cref{E2_intro_eq} will not induce discontinuities.
\end{remark}

\subsection{Interpolation results: Boundedness of $\protect\Psii$}
We bound the difference between $\Psii$ and $\idx$ in different norms:
\begin{theorem} \label{lemma_diff_Psi_id_small} Let $q_t^\ast = q_t + 1$ for the smooth blending and $q_t^\ast = q_t$ for the \FE~blending. Then, it holds
  \begin{subequations}
 \begin{align}
  \| \Psii - \idx \|_{\infty, \tQn } &\lesssim h^2 + \Delta t^{q_t+1} \label{Psi_to_id_lo_boundedA},
  \addtocounter{equation}{1}\tag{\theequation-\texttt{$\Psi$\footnotesize{bnd}}}    
  \\
   \| \partial_t \Psii \|_{\infty, \tQn} &\lesssim h^2 + \Delta t^{q_t}. \label{Psi_to_id_lo_boundedB}
  \addtocounter{equation}{1}\tag{\theequation-\texttt{$\Psi$\footnotesize{bnd}}}    
  \\
  \| \nabla \Psii - \Idx \|_{\infty, \ThQ} &\lesssim h + 
      \Delta t^{q_t^\ast}, \label{Psi_to_id_lo_boundedC}
  \addtocounter{equation}{1}\tag{\theequation-\texttt{$\Psi$\footnotesize{bnd}}}    
  \\
  \| D^{m_s} \partial_t^{m_t} \Psii \|_{\infty,\ThQ} &\lesssim 1 \quad \textnormal{for } m_s  \leq q_s +1, m_t \leq q_t^\ast, m_t + q_t \leq \lphi.  
\label{Psi_ho_deriv_boundedA}
  \addtocounter{equation}{1}\tag{\theequation-\texttt{$\Psi$\footnotesize{bnd}}}    
  \end{align}
  \end{subequations}
\end{theorem}
\begin{remark}
\eqref{Psi_to_id_lo_boundedA}--\eqref{Psi_to_id_lo_boundedC} can be summarized as in 
\eqref{Psi_Gamma_lo_deriv_ho_boundA} of \cref{cor:PsiG} only for the smooth blending case. 
Resulting from the $1/h$-factor in \cref{ext_bnd_lemma_lo} for $m_s=1$ for the \FE~blending, the bound for the first spatial derivative is weaker by one order in $\Delta t$ than for the smooth blending for which no extension step is required.
Similarly, the bound \eqref{Psi_ho_deriv_boundedA} requires a stronger limitation on the order of the time derivative for the \FE~blending than for the smooth blending.
\end{remark}
\begin{proof}

 We start with the smooth function blending. By construction of the extension operator, only the space-time mesh $\ThGQ$ needs to be discussed, where $\Psii = \PsiG$. Hence, all results follow from \cref{cor:PsiG}. Next, we note that the same argument applies to $\ThGQ$ for the \FE~blending.
For $\ThGQS$ with the \FE~blending, we apply \Cref{dt_E_commutation_lemma} and \cref{ext_bnd_lemma} (at every $t\in\In{n}$) to gain control over $\| \partial_t^{m_t} D^{m_s} (\Psii - \idx) \|_{\infty, \ThGQS} =  \| D^{m_s} \mathcal{E}^{\partial \OG} \partial_t^{m_t}(\PsiG - \idx) \|_{\infty, \ThGQS}$. The structure of \cref{ext_bnd_lemma} suggests a distinction of cases $m_s = 2, \dots,q_s+1$ (case 1), 
$m_s = 0$ (case 2) and  $m_s=1$ (case 3).

\underline{Case 1, proof of \eqref{Psi_ho_deriv_boundedA} for $m_s = 2, \dots,q_s+1$:}\\
We start with applying \cref{ext_bnd_lemma_ho} in order to yield for $m_s = 2, \dots,q_s+1$
\begin{align}
 & \| D^{m_s} \partial_t^{m_t} \Psii \|_{\infty, \ThGQS} =  \| D^{m_s} \mathcal{E}^{\partial \OG} \partial_t^{m_t}\PsiG \|_{\infty, \ThGQS} \nonumber \\
 \overset{\cref{ext_bnd_lemma_ho}}{\lesssim} & \max_{F \in \Facets(T)} \sum_{r=m_s}^{q_s+1} \underbrace{h^{r-m_s}}_{\lesssim 1} \| D^r \partial_t^{m_t}\PsiG \|_{\infty,F \times \In{n}} \lesssim 1 \quad \mathrm{by} \quad \cref{Psi_Gamma_dt_ho_derivs_bounded}. 
\end{align}
In the last step we exploited \cref{ass:reg_feblend} so that to bound
$\| D^r \partial_t^{m_t}\PsiG \|_{\infty,F \times \In{n}}$ with \cref{Psi_Gamma_dt_ho_derivs_bounded} we have for the indices $r$ and $m_t$ the bounds $r \leq q_s+1$, $m_t \leq q_t+1$ and $r+m_t \leq q_s+q_t+2 \leq \lphi$.

\underline{Case 2, proof of \eqref{Psi_to_id_lo_boundedA}, \eqref{Psi_to_id_lo_boundedB} and \eqref{Psi_ho_deriv_boundedA} for $m_s=0$:}\\[-3.5ex]
 \begin{align}
 & \| \partial_t^{m_t} (\Psii - \idx) \|_{\infty, \ThGQS} = \| D^0 \mathcal{E}^{\partial \OG} \partial_t^{m_t}(\PsiG - \idx) \|_{\infty, \ThGQS}
 \nonumber
 \\[-1ex]
 \!\!\! \overset{\cref{ext_bnd_lemma_lo}}{\lesssim} \!\!\!\!\!\!\! & \max_{F \in \Facets(\partial \OG)} \sum_{r=0}^{q_s+1} h^r \| D^r \partial_t^{m_t} (\PsiG - \idx)\|_{\infty,F \times \In{n}} + \overbrace{ \|\partial_t^{m_t} (\PsiG - \idx)\|_{\infty,\Verts(\partial \OG) \times \In{n}}}^{\lesssim h^2 + \Delta t^{q_t+1-m_t} \textnormal{ by } \cref{Psi_Gamma_first_dt_boundB}} \nonumber
\end{align}
Different values of $m_t$, yield the different results:
For $m_t = 0,1$, we bound all the summands $r=2,\dots, q_s+1$ by $h^2$, again by \cref{Psi_Gamma_dt_ho_derivs_bounded}. The summand for $r=0$ can be bounded with $\lesssim h^2 + \Delta t^{q_t+1-m_t}$ by \cref{Psi_Gamma_lo_deriv_ho_boundA}. The summand $r=1$ similarly with $h (h + \Delta t^{q_t+1-m_t}) \lesssim h^2 + \Delta t^{q_t+1-m_t}$. This concludes the proofs of \cref{Psi_to_id_lo_boundedA} and \cref{Psi_to_id_lo_boundedB}.
For $m_t=2,\dots,q_t+1$, we combine $h^r \lesssim 1$ with \cref{Psi_Gamma_dt_ho_derivs_bounded} (using \cref{ass:reg_feblend}) to yield  \cref{Psi_ho_deriv_boundedA}. 

\underline{Case 3, proof of \eqref{Psi_to_id_lo_boundedC} and \eqref{Psi_ho_deriv_boundedA} for $m_s=1$:} \\ 
We start by observing
\begin{align}
 & 
 \| D^1 \mathcal{E}^{\partial \OG} \partial_t^{m_t}(\PsiG - \idx) \|_{\infty, \ThGQS} \label{eq_613} \\[-2ex]
 \!\!\overset{\cref{ext_bnd_lemma_lo}}{\lesssim} \!\!\!\!\!\!\!\!\!\! \max_{F \in \Facets(\partial \OG)} & \sum_{r=1}^{q_s+1} h^{r-1} \| D^r \partial_t^{m_t} (\PsiG - \idx)\|_{\infty,F \times \In{n}} \!+\! \frac{1}{h} \overbrace{ \|\partial_t^{m_t} (\PsiG - \idx)\|_{\infty,\Verts(\partial \OG)\times\In{n}}}^{\lesssim h^2 + \Delta t^{q_t+1-m_t} \textnormal{ by } \cref{Psi_Gamma_first_dt_boundB}} \nonumber
\end{align}
We first note that -- for the purposes of proving \cref{Psi_ho_deriv_boundedA} -- by \cref{ass:dtlesssimh_feblend}, we obtain in regards to the last summand $\frac{1}{h}(h^2 + \Delta t^{q_t + 1 - m_t}) \lesssim h + \Delta t^{q_t - m_t}$. Together with $m_t \leq q_t$ required for the \FE~blending, the boundedness of this summand follows. In relation to the other summands, we confirm boundedness as before from \cref{Psi_Gamma_dt_ho_derivs_bounded}. 
\end{proof}

We introduce the following Sobolev spaces of space-time functions with weak derivatives of possibly non-uniform degree in space and time:
Let $Q$ be a space-time domain, then we define for $r,s \in \mathbb{N}$
\begin{equation} \label{eq:Hrs}
 \hypertarget{def:Hrs}{\Hrs{r}{s}}(Q) := \{ u \, | \, \partial_t^q D^\alpha u \in L^2(Q), p,q \in \mathbb{N}, p = |\alpha| \leq r, q \leq s \}.
\end{equation}
We will only need the spaces or corresponding norms for the cases $(r,s) \in \mathbb{N} \times \{0,1\} \cup \{0,1\} \times \mathbb{N}$ in the following.
\begin{remark}
The concept of mixed regularities in space and time can further be generalized and refined within the concept of $t$-anisotropic Sobolev spaces. In that context, Sobolev spaces $H^{r,s}$ are typically defined through interpolation between spaces with bounded derivatives for $(r,0)$, i.e. only higher-order spatial derivatives and $(0,s)$, i.e. only higher-order temporal derivatives.
In our simpler definition \eqref{eq:Hrs} $(r,s)$-regularity implies that also the weak $(r,s)$-derivative, i.e. the combined $r$-th order spatial and $s$-th order temporal derivate, exists. We indicate this stronger notion by the rectangle in the index%
\footnote{as opposed to a triangle that would correspond to the usual $t$-anisotropic Sobolev spaces 
$H^{r,s}(Q) = H_{\scalebox{0.45}[0.45]{$\vert$}\scalebox{0.5}[0.5]{$\!\underline{~~~\!}$}\scalebox{0.85}{\!}\!\!\scalebox{1.2}[0.4]{$\backslash$}}^{r,s}(Q) := \{ u \, | \, \partial_t^p D^\alpha u \in L^2(Q),p,q \in \mathbb{N}, q = |\alpha|, \frac{q}{r} + \frac{p}{s} \leq 1 \}$.
}
and note that w.r.t. the standard (isotropic) Sobolev space $H^m(Q)$ there holds $H^{r+s}(Q) \subseteq \Hrs{r}{s}(Q) \subseteq H^{\min(r,s)}(Q)$.
\end{remark}

\begin{corollary} \label{boundedness_result_final}
  Let $q_t^\ast = q_t + 1$ for the smooth and $q_t^\ast = q_t$ for the \FE~blending. 
  There holds
\begin{equation}
 \| \Psii\|_{\Hrs{q_s+1}{1}(\ThQ)} + \| \Psii\|_{\Hrs{1}{q_t^\ast}(\ThQ)} \lesssim 1.
\end{equation}
\end{corollary}

\subsection{Space-Time mapping functions}
In the following, 
we introduce canonical variants of the mappings with a space-time range, 
$\Psist \colon \tQn \to \tQn$ and $\Thest \colon \tQn \to \tQn$ are the space-time variants of $\Psii$ and $\Theh$ with
\begin{equation}
  \hypertarget{def:Psist}{\Psist (x,t) \coloneqq ( \Psii(x,t), t) }, \quad 
  \hypertarget{def:Thest}{\Thest (x,t) \coloneqq (\Theh(x,t), t). }
\end{equation}
In these terms, we also define the discrete space-time domain of integration:
\begin{equation}
 \hypertarget{def:Qhn}{\Qhn \coloneqq \Thest ( \Qlin )}, \quad 
 \hypertarget{def:Oh}{\Ohi{t} \coloneqq \Theh (\Omlin{t}, t).} 
\end{equation}
%

\subsection{Geometry approximation results}
Taking together results from the previous parts, we obtain a higher-order bound on the difference between $\Psi$ and $\Theta_h$:
\begin{theorem} \label{lemma_diff_Theta_h_Psi_small}
  Let $q_t^\ast = q_t + 1$ for the smooth and $q_t^\ast = q_t$ for the \FE~blending. Then, for $m_s = 0,\dots,q_s+1$, $m_t=0, \dots, q_t+1$ it holds
 \begin{subequations} \label{Theta_h_Psi_diff_ho_bounds}
 \begin{align}
  \Delta t^{m_t} h^{m_s} \| D^{m_s} \partial_t^{m_t} (\Thest - \Psist) \|_{\infty, \ThQ} 
  &\lesssim h^{q_s +1} + \Delta t^{q_t + 1},  \label{Theta_h_Psi_diff_ho_boundA}
  \addtocounter{equation}{1}\tag{\theequation-\texttt{$\Theta_h$\footnotesize{acc}}}    \\
  \|  \nabla (\Thest - \Psist)\|_{\infty, \ThQ} &\lesssim h^{q_s\hphantom{+ 1}} + \Delta t^{q_t^\ast}. \label{Theta_h_Psi_diff_ho_boundB}
  \addtocounter{equation}{1}\tag{\theequation-\texttt{$\Theta_h$\footnotesize{acc}}}    \\
  \intertext{Under assumption \cref{diff_phiH_phihn} there further holds}
   \|  \partial_t (\Theh - \Psii)\|_{\infty, \ThQ} & \lesssim h^{q_s+1} + \Delta t^{q_t}. \label{Theta_h_Psi_diff_ho_boundC-new}
  \addtocounter{equation}{1}\tag{\theequation-\texttt{$\Theta_h$\footnotesize{acc}}}   
\end{align}
\end{subequations}
\end{theorem}
\begin{proof}
First, we observe that for the smooth blending the claims \eqref{Theta_h_Psi_diff_ho_boundA} and \eqref{Theta_h_Psi_diff_ho_boundB} follow directly from \cref{Theta_h_Psi_Gamm_diff_bound_Dr_dt} 
and \cref{Theta_h_Psi_Gamm_diff_bound_first_grad}. 
In the case of the \FE~blending the same results are obtained by making use of \cref{ext_bnd_lemma} and \cref{ass:dtlesssimh_feblend} (for \eqref{Theta_h_Psi_diff_ho_boundB}). For completeness, we give the proof in \cref{app:Theta_h_Psi_diff_ho_bounds:FEblend}.\\
\eqref{Theta_h_Psi_diff_ho_boundC-new} follows directly from 
\cref{lem:Theta_h_Psi_Gamm_diff_bound_first_dt} and -- in the case of the \FE~blending -- the commuting property of the 
\emph{spatial} extension $\mathcal{E}^{\partial \OG}$, cf. \cref{dt_E_commutation_lemma}.
\end{proof}
Taking the results of this lemma together with \cref{lemma_diff_Psi_id_small}, we conclude that for $h$ and $\Delta t$ sufficiently small, the mapping $\Phihst$ introduced as follows is a bijection:
\begin{equation}
 \hypertarget{def:Psihst}{\Phihst := \Psist \circ (\Thest)^{-1}}
\end{equation}
By this definition, $\Phihst$ translates between the higher-order geometry approximation and the exact domain, $\Phihst( \Qhn )=\Qn{n}$. Following the convention so far, let $\Phi_h$ be the spatial components of $\Phihst$, i.e. let $\Phi_h$ be the function such that $\Phihst(x,t) = ( \Phi_h(x,t),t )$.
The previous results imply bounds on $\Phihst$:
\begin{corollary} \label{lemma_Phi_h_Id_diff_small}
  Let $q_t^\ast = q_t + 1$ for the smooth and $q_t^\ast = q_t$ for the \FE~blending. Then, it holds for
$m_s = 0,\dots,q_s+1$, $m_t=0, \dots, q_t+1$ 
\begin{subequations}
\begin{align}
  \Delta t^{m_t} h^{m_s} 
  \| D^{m_s} \partial_t^{m_t} (\Phihst - \idxt) \|_{\infty, \ThQ}& \lesssim h^{q_s +1} + \Delta t^{q_t + 1}. 
  \addtocounter{equation}{1}\tag{\theequation-\texttt{$\Phi_h^{\!\text{st}}$\footnotesize{bnd}}}    
  \\
  \|  \nabla \Phihst - \Idxt\|_{\infty, \tQn} &\lesssim h^{q_s} + \Delta t^{q_t^\ast},
  \addtocounter{equation}{1}\tag{\theequation-\texttt{$\Phi_h^{\!\text{st}}$\footnotesize{bnd}}}    
   \intertext{ If further \cref{diff_phiH_phihn} holds, there is
  }
  \| \partial_t \Phi_h \|_{\infty, \tQn} &\lesssim h^{q_s +1} + \Delta t^{q_t}.
  \addtocounter{equation}{1}\tag{\theequation-\texttt{$\Phi_h^{\!\text{st}}$\footnotesize{bnd}}}    
 \end{align}
\end{subequations} 
\end{corollary}
\begin{proof}
 We observe that $\Phihst - \idxt = (\Psist - \Thest )\circ (\Thest)^{-1}$. Then, with \cref{lemma_diff_Psi_id_small} and \cref{lemma_diff_Theta_h_Psi_small}, we obtain boundedness of $\Thest$ as well as of $(\Thest)^{-1}$ and its derivatives, and the result follows from 
 an application of the Faà di Bruno formula and
 \cref{Theta_h_Psi_diff_ho_boundA,Theta_h_Psi_diff_ho_boundB,Theta_h_Psi_diff_ho_boundC-new}.
\end{proof}

\subsection{Properties of \texorpdfstring{$\Theta_h$}{discrete mapping}} \label{ss_prop_of_Thetah}
Note that due to a triangle inequality, taking into account 
\cref{Theta_h_Psi_diff_ho_boundA,Theta_h_Psi_diff_ho_boundB,Theta_h_Psi_diff_ho_boundC-new} and \cref{Psi_to_id_lo_boundedC} we can conclude the following bound on the difference between $\Thest$ and $\idxt$:
\begin{corollary} \label{cor_Theta_h_Id_diff}
Let $q_t^\ast = q_t + 1$ for the smooth and $q_t^\ast = q_t$ for the \FE~blending. Then, there holds
 \begin{align}
  \| \Thest - \idxt \|_{\infty, \tQn}\lesssim h^2 + \Delta t^{q_t + 1}, ~~ \| \nabla \Thest - \Idxt\|_{\infty, \tQn} &\lesssim h + \Delta t^{q_t^\ast}.
 \end{align}
\end{corollary}
This implies that the norm of a function defined on $\Omlin{t}$ or $\Qlin$ is equivalent to the function lifted to $\Ohi{t}$ or $\Qhn$, respectively, in accordance with $\Theh$:
\begin{lemma}
 Let $w \in L^2(\Qlin)$ and $v_t \in L^2(\Omlin{t})$ for $t \in \In{n}$. Then it holds
 \begin{equation}
  \| w \|_{\Qlin} \simeq \| w \circ (\Thehst)^{-1} \|_{\Qhn}, \quad 
  \| v_t \|_{\Omlin{t}} \simeq \| v_t \circ \Theh(\cdot,t)^{-1} \|_{\Ohi{t}} 
 \end{equation}
\end{lemma}
\begin{proof}
 From \cref{cor_Theta_h_Id_diff}, we have $\det(\nabla \Thehst) = 1 + \mathcal{O}(h + \Delta t^{q_t})$.
\end{proof}

\subsection{An estimate for the discrete space-time normal} \label{ss_discrete_normal}
We define the outward pointing \emph{space-time normals} to the space-time domains $\Qn{n}$ and $\Qhn$ on the \emph{spatial} boundary as $n $ and $n_h$ respectively. 
We can give bounds between the difference of these two normals (after lifting them on the same domain):
\begin{lemma}
With \cref{diff_phiH_phihn} there holds
 \begin{equation}
  \| n \circ \Phihst - n_h \|_{\infty, \partials \Qhn} \lesssim h^{q_s}  + \Delta t^{q_t}. 
 \end{equation}
\end{lemma}
\begin{proof}
 We start with the same argument as given above of \cite[Eq. (A.20)]{LR_IMAJNA_2018}, which builds on the observation that the tangent space is transported by the mapping $\Phihst$: Fix a point $(x,t) \in \partials Q^h$ and let $t_1,\dots, t_d$ be an orthonormal basis of $n_h^\perp$. The tangent space to $\partials Q = \Phihst (\partials Q^h)$ is given by $\operatorname{span}\{ \Dxt \Phihst(x,t) t_j \, | \, 1 \leq j \leq d\} = n^\perp$. Hence, we have $n ( \Phihst(x,t)) \in \operatorname{span}\{(\Dxt \Phihst(x,t))^{-T} n_h \}$, so that we can write
  $
  n \circ \Phihst = \nicefrac{(\Dxt \Phihst)^{-T} n_h }{\| (\Dxt \Phihst)^{-T} n_h\|}
  $.
 With the help of \cref{lemma_Phi_h_Id_diff_small}, we have $(\Dxt \Phihst)^{-T} = \Idxt + \mathcal{O}(h^{q_s} + \Delta t^{q_t})$ from which we obtain the claim after a triangle inequality.
 
%

\end{proof}
\begin{remark}
  Without \cref{diff_phiH_phihn} an additional term $\nicefrac{h^{q_s+1}}{\Delta t}$ would enter from the estimates in \cref{lemma_Phi_h_Id_diff_small} yielding a worse estimate of the form 
  $
    \| n \circ \Phihst - n_h \|_{\infty, \partials \Qhn} \lesssim h^{q_s}  + \Delta t^{q_t} + \nicefrac{h^{q_s+1}}{\Delta t}
$.
\end{remark}

\section{Interpolation on \texorpdfstring{$\Psi$-}{}deformed space-time meshes} \label{sec:interpolation}

In the analysis of numerical discretizations in the isoparametric setting, we have a mismatch in the domains of definition, as the solution $u$ is defined on the physical domain, whereas discrete functions are fundamentally defined on $\Qlin$/ $\Qhn$. As the function $\Psii$ maps from the piecewise linear reference geometry to the exact geometry, it is instructive to search for a discrete variant of the function $\hat u = u^e \circ \Psist$, for an extension $u^e$ of $u$. To obtain this approximation, we will apply a result from \cite{preuss18} and discuss the relevant implications in relation to the geometry approximation.

Note that the domain of definition of $\hat u$ depends on the physical problem of interest. In a surface problem, it would be reasonable to use the set $Q^\Gamma_{h1} := \bigcup_{T \in \ThbOne}  T \times \In{n}$ as the smallest set of complete elements containing all relevant information. For a bulk problem, the following sets serve the purpose of the domain of definition
\begin{align}
 \hypertarget{def:EQlin}{\EQlin} := \bigcup \{ T \times I_n \, | \, (T \times I_n) \cap \Qlin \!\! \neq \! \varnothing \}, ~~
 \hypertarget{def:EOmlin}{\EOmlin} := \bigcup \{ T \, | \, (T \times I_n) \cap \Qlin \! \neq \! \varnothing \}.
\end{align}
From now on, we adopt the point of view of the bulk problem and interpolate on $\EQlin$, but the respective results on the smaller set $Q^\Gamma_{h1}$ follow.

We now show interpolation results on the space-time domain $\EQlin$.
\begin{lemma} \label{interpollemma1}
  There exists an interpolation operator $\Pi^n_W\colon \! L^2\!(\EQlin) \to \Vh{k_s,k_t}\!(\EQlin)$, s.t. for all functions $\hat u \in L^2(\EQlin)$ with sufficient regularity, s.t. the r.h.s. expressions in the following estimates are well-defined and bounded, there holds for the interpolation error $\hat e_u = \hat u - \Pi_W^n \hat u $ and $\ell_s \in \{ 1,\dots,k_s+1\}$, $\ell_t \in \{ 1,\dots,k_t+1\}$:
  \begin{subequations}
   \begin{align}
   \| \hat e_u \|_{\EQlin} 
   &
   \lesssim ~ \Delta t^{\ell_t\hphantom{-1}} \| \hat u \|_{\Hrs{0}{\ell_t}(\EQlin)} \!+\! \hphantom{\Delta t^{-\frac12}} h^{\ell_s\hphantom{-1}} \| \hat u\|_{\Hrs{\ell_s}{0}(\EQlin)}, 
   \\
   \!\! \| \partial_t \hat e_u \|_{\EQlin} 
   &
   \lesssim ~ \Delta t^{\ell_t-1} \| \hat u \|_{\Hrs{0}{\ell_t}(\EQlin)} \!+\! \hphantom{\Delta t^{-\frac12}} h^{\ell_s\hphantom{-1}} \| \hat u\|_{\Hrs{\ell_s}{1}(\EQlin)}\hphantom{,}~\text{if } \ell_t \!\ge\! 2, 
   \\
   \| \nabla \hat e_u \|_{\EQlin} 
   &
   \lesssim ~ \Delta t^{\ell_t\hphantom{-1}} \|\hat u \|_{\Hrs{1}{\ell_t}(\EQlin)} \!+\! \hphantom{\Delta t^{-\frac12}}h^{\ell_s-1} \| \hat u\|_{\Hrs{\ell_s}{0}(\EQlin)}\hphantom{,}~\text{if } \ell_s \!>\! 2, 
   \\
   \!\!\!\!\!\!\!\!\!\!  \| \hat e_u(\cdot,t_n) \|_{\EOmlin} 
   &
   \lesssim ~ \Delta t^{\ell_t-\frac12\!} \| \hat u \|_{\Hrs{0}{\ell_t}(\EQlin)} 
   \!+\! \Delta t^{-\frac12} h^{\ell_s\hphantom{-1}} \| \hat u\|_{\Hrs{\ell_s}{0}(\EQlin)}\hphantom{,}~\text{if } \ell_t \!\ge\! 2.
 \end{align}
\end{subequations}
\end{lemma}
 \begin{proof}
  We apply \cite[Theorem 3.18 and Lemma 3.21]{preuss18}. Note that we choose the polygonal spatial domain $\EOmlin$, which corresponds to $\Omega$ internally in \cite[Sec. 3.3]{preuss18}.
 \end{proof}
This result asks for upper bounds for $\| \hat u \|_{\Hrs{0}{\ell_t}(\EQlin)}$, $\| \hat u\|_{\Hrs{\ell_s}{0}(\EQlin)}$, as well as for $\| \nabla \hat u \|_{\Hrs{0}{\ell_t}(\EQlin)}$ and $\| \partial_t \hat u\|_{\Hrs{\ell_s}{0}(\EQlin)}$. Those will be influenced both by the regularity of $\hat u$. For $\hat u = u^e \circ \Psist$ this regularity crucially also depends on $\Psist$ where we need the bounds derived before.
\begin{lemma} \label{interpollemma2}
Bounds on higher-order spatial or temporal derivatives of $\Psist$ imply bounds on higher-order derivatives of $u^e \circ \Psist$ by norms of $u^e$ as follows for $\ell_t, \ell_s \in \mathbb{N}$:  
\vspace*{-0.5cm}
\begin{subequations}
\begin{align} 
  \| \Psist\|_{\Hrs{0}{\ell_t}(\EQlin)} \lesssim 1 
  &~\quad\Longrightarrow~&
  \| u^e \circ \Psist\|_{\Hrs{0}{\ell_t}(\EQlin)} &\lesssim \| u^e \|_{H^{\ell_t}(\Psist(\EQlin))},  \\ 
  \| \Psist\|_{\Hrs{\ell_s}{0}(\EQlin)} \lesssim 1
  &~\quad\Longrightarrow~&
  \| u^e \circ \Psist\|_{\Hrs{\ell_s}{0}(\EQlin)} &\lesssim \| u^e \|_{H^{\ell_s}(\Psist(\EQlin))}, \\
  \| \Psist\|_{\Hrs{\ell_s}{1}(\EQlin)} \lesssim 1
  &~\quad\Longrightarrow~&
  \| \partial_t (u^e \circ \Psist)\|_{\Hrs{\ell_s}{0}(\EQlin)} &\lesssim \| u^e \|_{H^{\ell_s+ 1}(\Psist(\EQlin))}, \\
  \| \Psist\|_{\Hrs{1}{\ell_t}(\EQlin)} \lesssim 1
  &~\quad\Longrightarrow~&
  \| \nabla (u^e \circ \Psist)\|_{\Hrs{0}{\ell_t}(\EQlin)} &\lesssim \| u^e \|_{H^{\ell_t+ 1}(\Psist(\EQlin))}. 
\end{align} 
\end{subequations}
\end{lemma}
\begin{proof}
  The proof relies on a Faà di Bruno formula in combination with the l.h.s. bounds. The proof is given in \cref{app:interpollemma2} for completeness.
\end{proof}  
The previous two results do not depend on the results of the previous sections but display the need and purpose of the higher-order bounds on $\Psist$. Combining the bounds for $\Psist$ with the interpolation results, allows to derive interpolation results for isoparametrically mapped space-time \FE~spaces for space-time discretizations.

\section{Numerical examples} \label{ref:numexp}
We conclude the paper with a section on numerical experiments. Note that in \cite{HLP2022}, we already studied numerically convection-diffusion discretisations involving our geometry approximation with \FE~blending. Hence, we complement these results with some further numerical investigations.

All experiments are performed with \texttt{ngsxfem} \cite{xfem_joss}, an unfitted \FE~extension of \texttt{ngsolve}.
Reproduction data are available at \url{https://gitlab.gwdg.de/fabian.heimann/repro-ho-unf-space-time-fem-geom}.
The linear systems are solved with the direct solvers \texttt{umfpack} and \texttt{pardiso} of IntelMKL.

To fix a specific setting, we choose a polynomial order parameter $k=1,\dots,6$ and set all geometry approximation and discrete \FE~space orders to this, $k=k_s = k_t = q_s = q_t$. Similarly, we focus in the first subsections on simultaneous space-time refinements, i.e. we introduce parameters $i= i_s = i_t$, where the time step $\Delta t$ and the mesh size $h$ satisfy
$\Delta t = {T} \cdot {2^{-i_t-2}}$, $h = 0.9 \cdot {2^{-i_s}}.$
For the evolving geometry, we use the first example of \cite{HLP2022}, where a circular geometry in 2D deforms into a kite. It is described in terms of the following levelset function $\phi$
\begin{equation*}
 \rho(t, y) = (1 - y^2) \cdot t, \quad r = \sqrt{(x- \rho)^2 + y^2}, \quad \phi = r - r_0, \quad r_0 = 1,
\end{equation*}
where $\tilde \Omega = [-1.6, 2.1] \times [-1.6,1.6]$, $T=0.5$. We use unstructured simplicial meshes of the maximal mesh size $h$ given above.

For the smooth blending function, we use the following construction: Fixing a smoothness order parameter $s \in \mathbb{N}$ 
we define $b$ for a blending width parameter $w_b$:
\begin{equation}
  b = \begin{cases}
    0 &\textnormal{if }  |\phi| \leq w_b, \\
    1  &\textnormal{if }  2 w_b < |\phi|, \\
    \pi_s\left(\frac{|\phi|-w_b}{w_b} \right)  &\textnormal{else,} 
   \end{cases}             
 \text{ with }\pi_s (x) = \begin{cases}
              2^{s-1} x^s & \textnormal{if} \quad x \leq \frac{1}{2} \\
              1 - 2^{s-1} (1-x)^s  &\textnormal{else}.
             \end{cases} \label{eq:def_b_example}
\end{equation}
We show examples of this blending function with the resulting deformation magnitudes in \cref{fig:blending_def_examples}, as well as a plot of $\pi_s$.
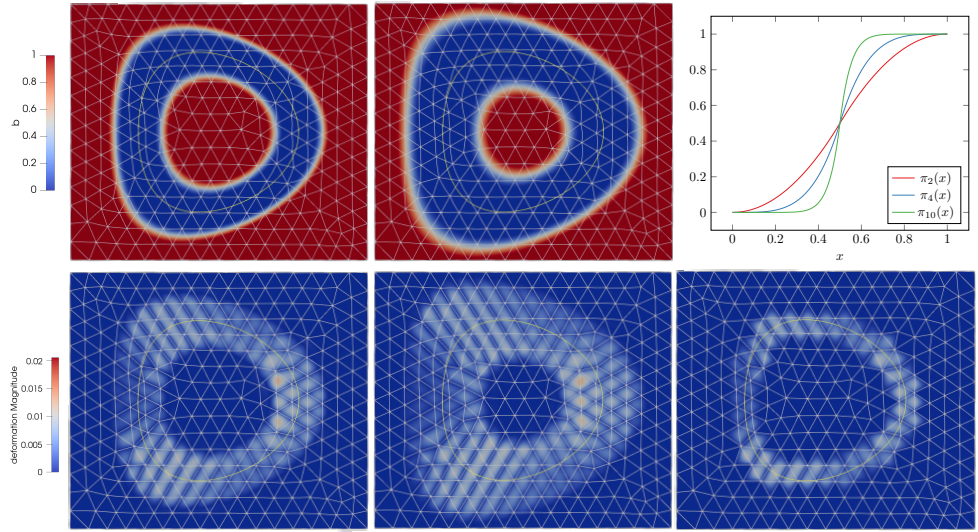
\begin{figure}
  \includegraphics[width=0.05\textwidth]{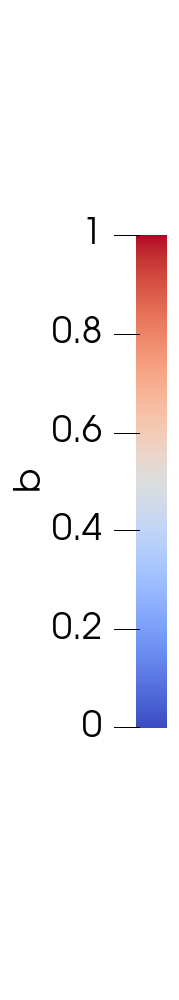}
  \includegraphics[width=0.31\textwidth, angle=0.5]{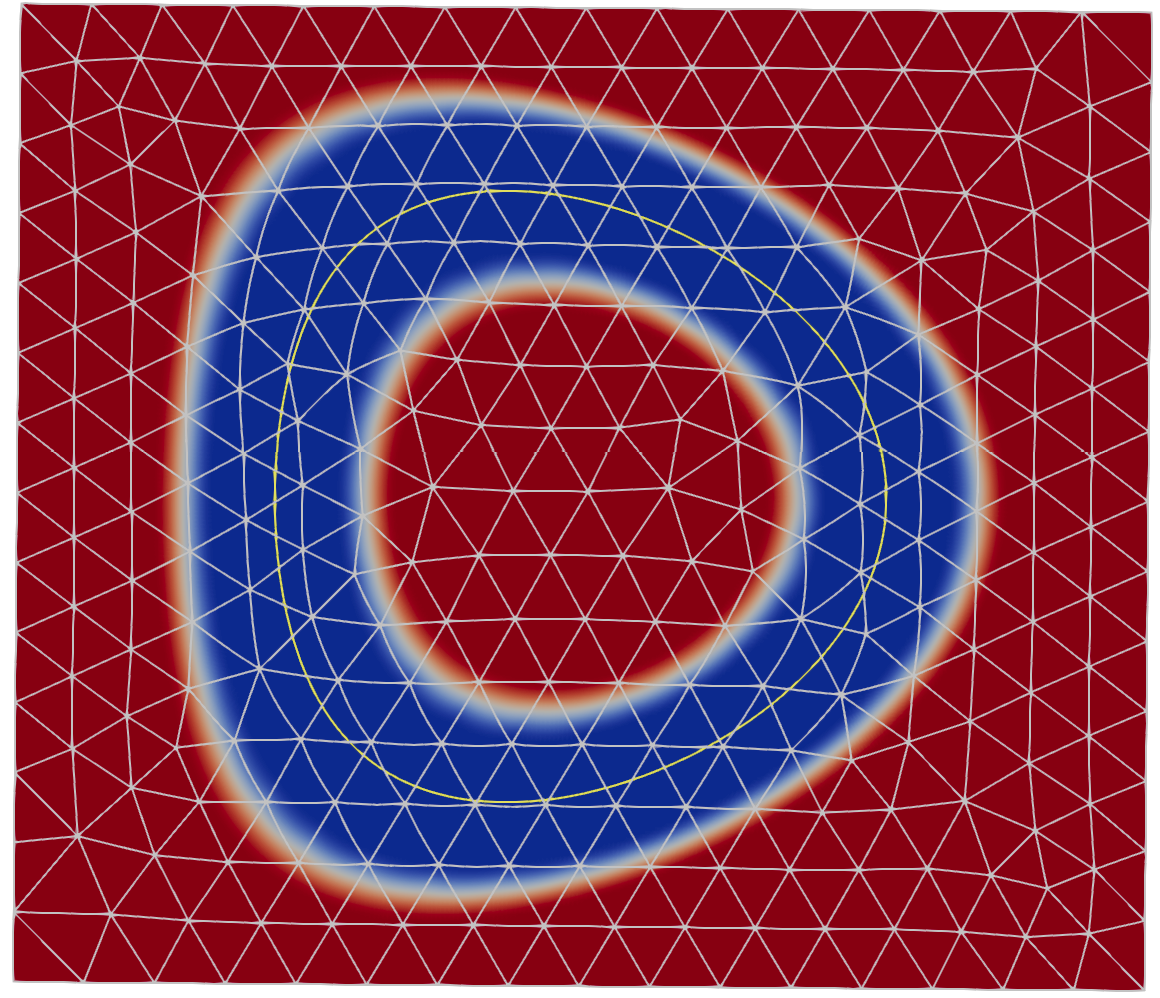} \hspace{-0.25cm}
  \includegraphics[width=0.31\textwidth, angle=0.5]{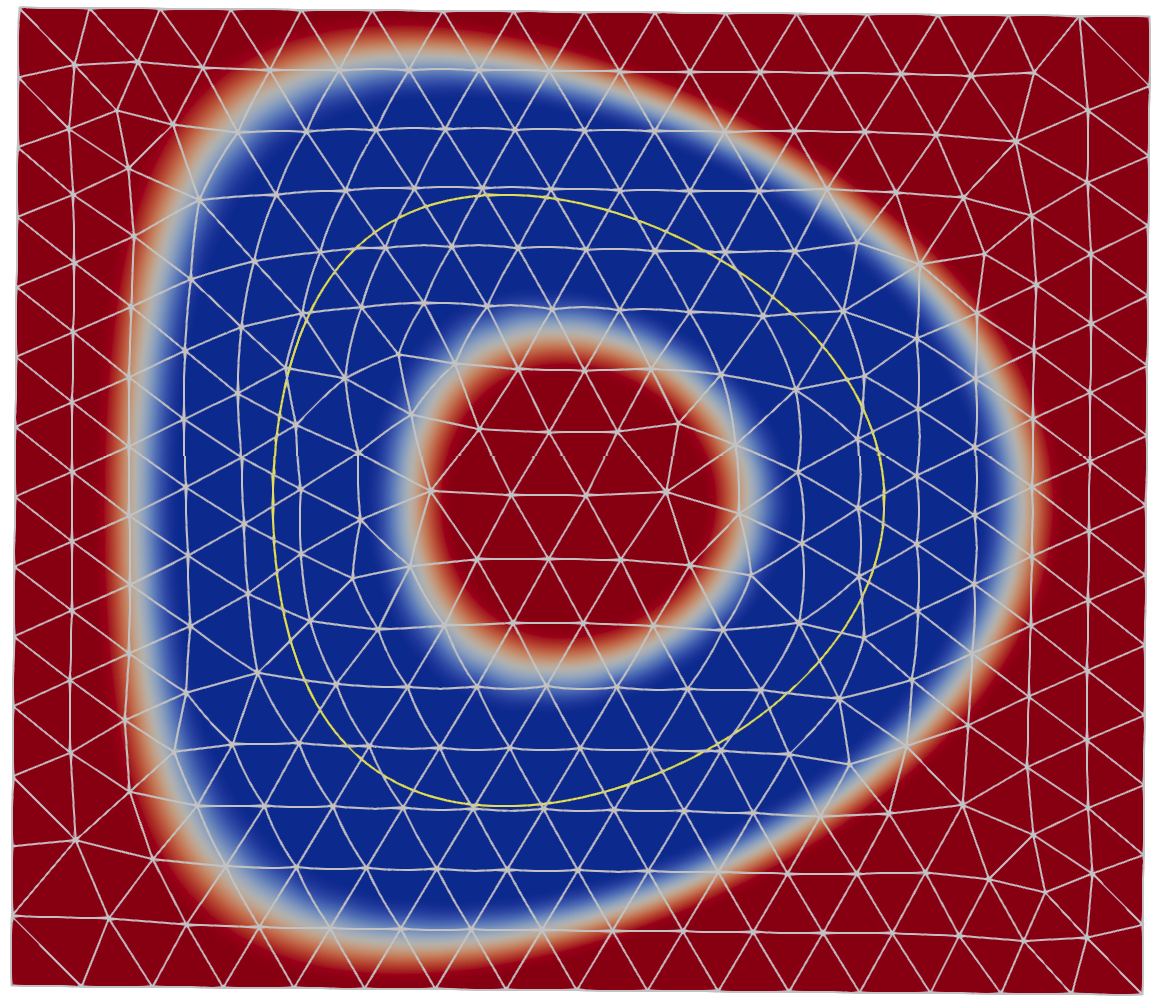} 
  \begin{tikzpicture}[scale=0.5]
\begin{axis}[
xlabel={$x$}, legend entries={$\pi_2(x)$, $\pi_4(x)$, $\pi_{10}(x)$}, legend pos=south east,
]
\addplot [Set1-A, domain=0:1, samples=201] {(x<0.5)*(2)*x^2 + (x>=0.5)*(1 - (2)*(1-x)^2) };
\addplot [Set1-B, domain=0:1, samples=201] {(x<0.5)*(2^3)*x^4 + (x>=0.5)*(1 - (2^3)*(1-x)^4) };
\addplot [Set1-C, domain=0:1, samples=201] {(x<0.5)*(2^9)*x^10 + (x>=0.5)*(1 - (2^9)*(1-x)^10) };
\end{axis}
  \end{tikzpicture} \\
  \includegraphics[width=0.0525\textwidth]{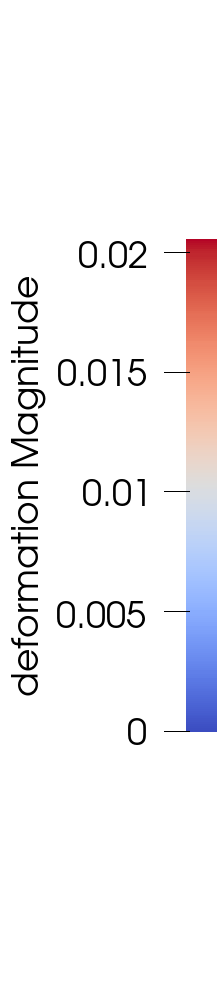}
  \includegraphics[width=0.31\textwidth, angle=0.5]{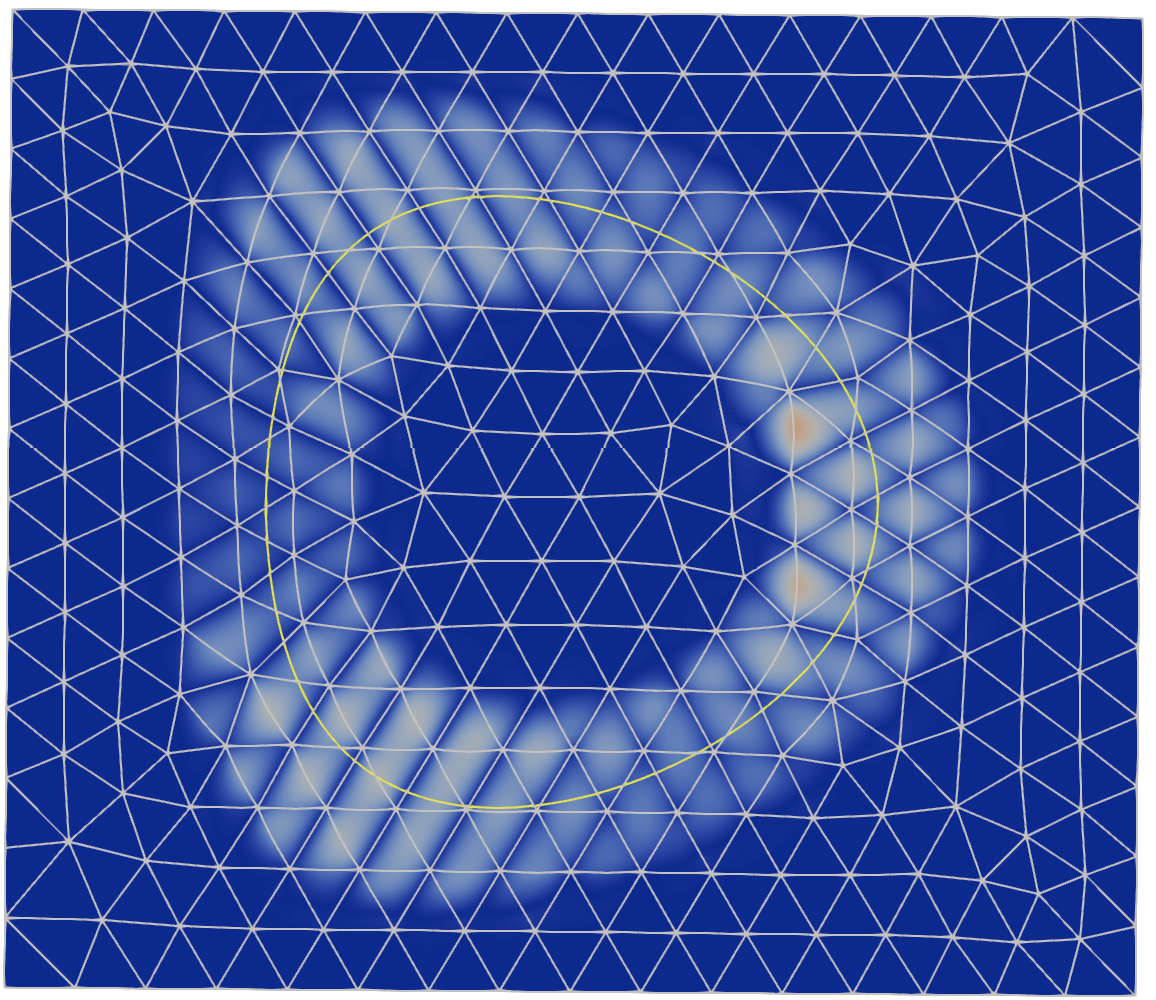} \hspace{-0.25cm}
  \includegraphics[width=0.31\textwidth, angle=0.5]{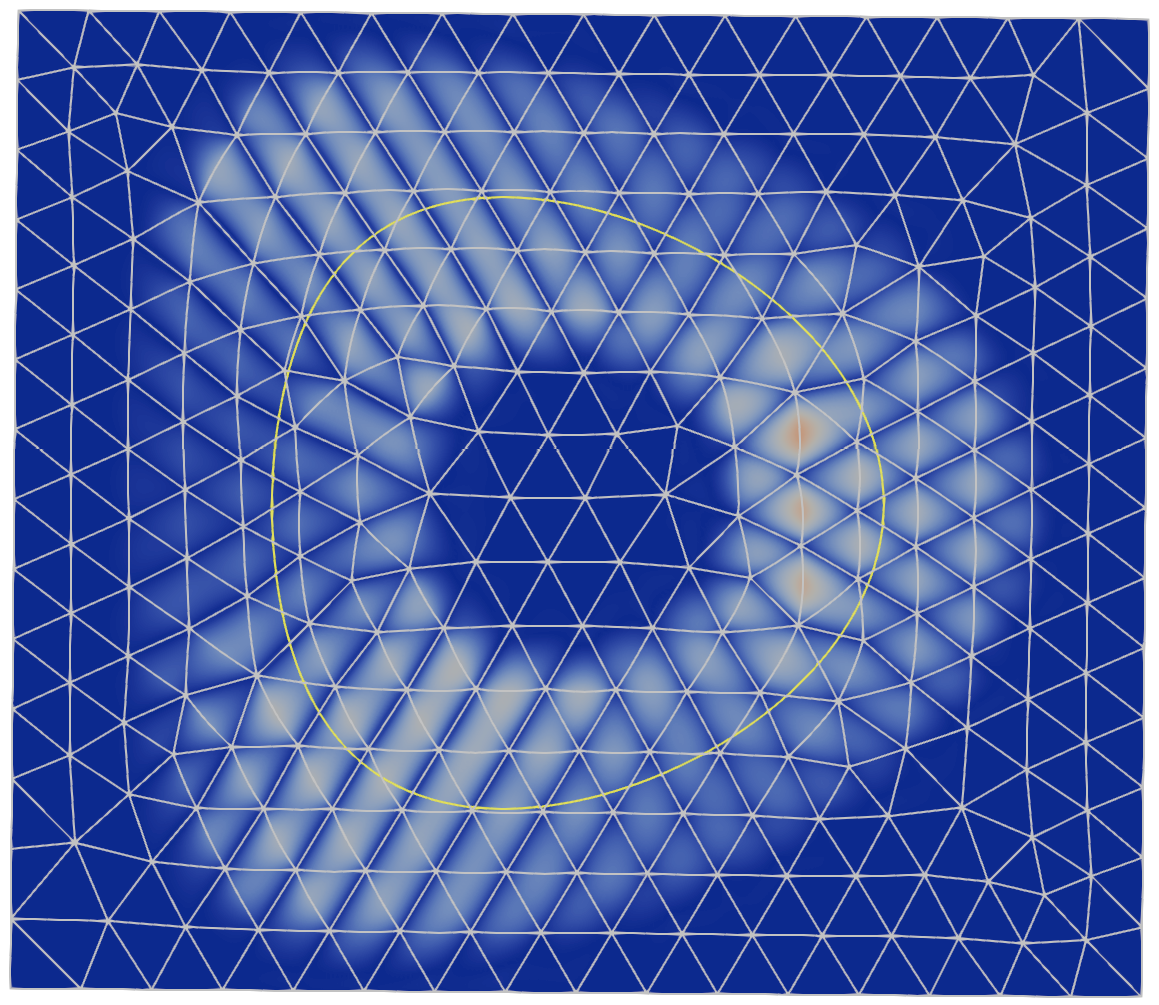} \hspace{-0.25cm}
  \includegraphics[width=0.31\textwidth, angle=0.5]{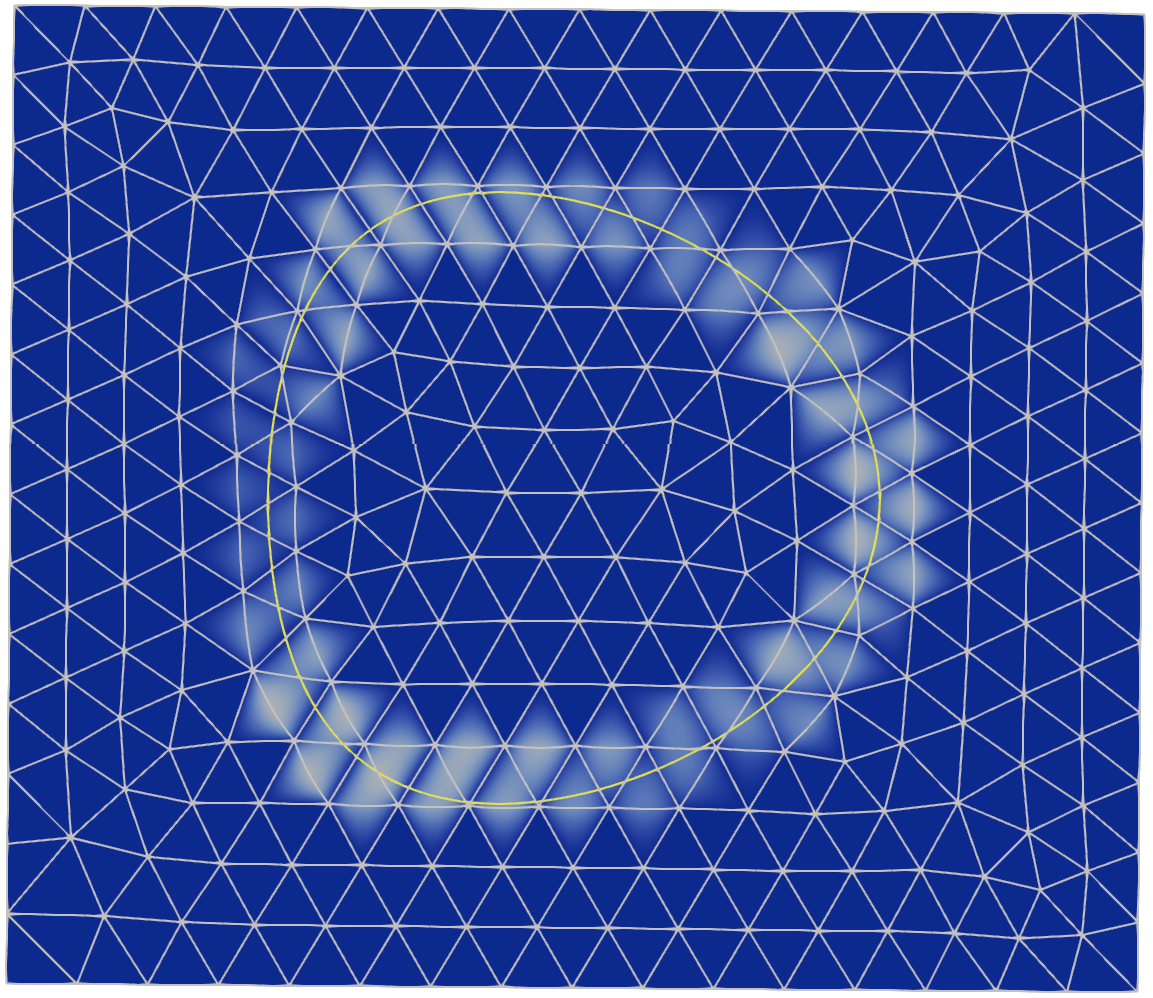}
 \vspace*{-0.25cm}
 \caption{Upper row left and middle: Two examples of blending functions $\protect\b$ at a fixed time for increasing values of $w_b$. Upper row right: Examples of the function $\pi_s$ involved in the definition of $\protect\b$ in terms of $\phi$, \cref{eq:def_b_example}. Bottom row left and middle: Absolute value of the deformation stemming from the smooth blending functions respectively above. Bottom row right: Absolute value of the deformation from the \protect\FE~blending for the same discretisation parameters and a small time step.} \vspace{-0.2cm}
 \label{fig:blending_def_examples} 
\end{figure}

We note that in view of the analysis, the blending width needs to be chosen sufficiently large ($w_b \gtrsim h$) so that all cut elements of a time slice have vanishing $b$, cf. first point of \cref{ass:blending}.
Numerically, we also include studies where cut elements are only partially contained in the region with vanishing $b$. 
\subsection{Geometrical Accuracy}
In the first numerical experiment, we want to study the geometrical accuracy of the suggested method. We measure numerically the distance between the numerical domain and the exact geometry on $\partial \Ohi{t}$. 
To this end, we take the maximal value of the exact levelset function on the discrete higher-order geometry $\partial \Ohi{t}$. As the chosen levelset function is a signed distance function up to a constant for an upper bound and a constant for a lower bound (stemming from the deformation of the coordinate system) 
this yields a proper distance measure that we denote by $\operatorname{dist}^\ast$ which is equivalent to the Euclidean distance.
In \cref{fig:geom_approx_conv} we observe for the chosen (equal) orders and refinements in space and time and with the smooth blending that
$\max_{t \in I_n} \mathrm{dist}^\ast(\partial \Ohi{t}, \partial \Omega(t)) \lesssim h^{k + 1} \simeq \Delta t^{k +1}$ which is in agreement with the theoretical predictions from \cref{lemma_Phi_h_Id_diff_small}.
We note that this study for the \FE~blending was already included in \cite{HLP2022} with similar results. We conclude that for the geometry approximation at the interface, both blending options work properly.

\begin{figure}
\centering
\vspace*{-0.25cm}
\begin{tikzpicture}[scale=0.825]
    \begin{semilogyaxis}[ xlabel=$i$, 
      title={$\max_{t \in [0,T]} \mathrm{dist}^\ast(\partial \Ohi{t}, \partial \Omega(t))$}, 
    title style={at={(0.5,0.8)},draw=black},
    legend entries ={ $k=1$, $k=2$, $k=3$, $k=4$, $k=5$, $k=6$}, legend style={anchor=north,legend columns=1, draw=none}, legend pos = south west, x label style={at={(axis description cs:0.97,0.06)},anchor=east},
     x tick label style={yshift=0cm,xshift=0.0em},
      y label style={at={(axis description cs:0.25,0.85)},anchor=east},
      ymax = 3e0,
      ymin = 1e-12,
      width=1.1\textwidth,
      height=0.45\textwidth,
     ]
     \addplot table [x index = 0, y index =5] {num_exp/out/conv_kite_DG_ks1_kt1_both_nref8_gamma0.05_sm2_interpolspecial_blend_p4_w0.1.dat};
     \addplot table [x index = 0, y index =5] {num_exp/out/conv_kite_DG_ks2_kt2_both_nref8_gamma0.05_sm2_interpolspecial_blend_p4_w0.1.dat};
     \addplot table [x index = 0, y index =5] {num_exp/out/conv_kite_DG_ks3_kt3_both_nref7_gamma0.05_sm2_interpolspecial_blend_p4_w0.1.dat};
     \addplot table [x index = 0, y index =5] {num_exp/out/conv_kite_DG_ks4_kt4_both_nref7_gamma0.05_sm2_interpolspecial_blend_p4_w0.1.dat};
     \addplot table [x index = 0, y index =5] {num_exp/out/conv_kite_DG_ks5_kt5_both_nref6_gamma0.05_sm2_interpolspecial_blend_p4_w0.1.dat};
     \addplot table [x index = 0, y index =5] {num_exp/out/conv_kite_DG_ks6_kt6_both_nref6_gamma0.05_sm2_interpolspecial_blend_p4_w0.1.dat};
     \addplot[gray, dashed, domain=0:7] {(1/2^(x+1)))^2};
     \addplot[gray, dashed, domain=0:7] {(1/2^(x+1)))^3};
     \addplot[gray, dashed, domain=0:6] {(1/2^(x+1)))^4};
     \addplot[gray, dashed, domain=0:6] {(1/2^(x+1)))^5};
     \addplot[gray, dashed, domain=0:5] {(1/2^(x+0.5)))^6};
     \addplot[gray, dashed, domain=0:5] {(1/2^(x+0.5)))^7};
    \end{semilogyaxis}
                  \node[scale=0.75] at (11.15,3.85) {$O(h^{k+1})= O(\Delta t^{k+1})$};
     \draw[scale=0.75, gray, dash=on 2.25pt off 2.25pt phase 0pt, line width=0.4*0.75pt] (12.05,5.125) -- (12.75,5.125);
   \end{tikzpicture} \vspace*{-0.3cm}
 \caption{Convergence of geometrical accuracy in terms of $\operatorname{dist}^\ast$ for simultaneous space and time refinements for the smooth blending on the kite geometry, $s = 4$, $w_b = 0.1$.} \vspace*{-0.25cm}
 \label{fig:geom_approx_conv}
\end{figure}
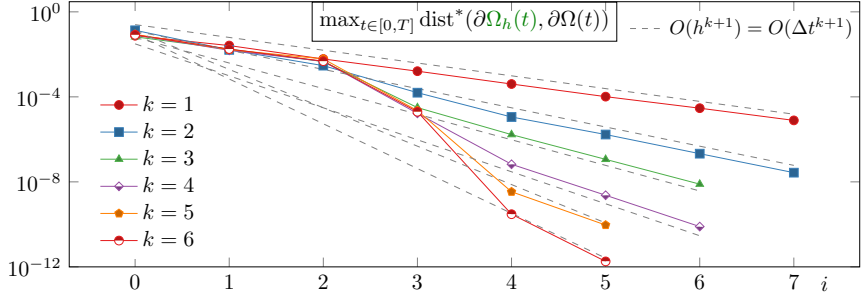


\subsection{Interpolation Quality}
As the next step, we solve a Ghost-penalty (\cite{burman2010ghost, burman15, preuss18, heimann20}) stabilized $L^2$ projection problem for a function $u = \cos (\pi \frac{r}{r_0}) \cdot \sin (\pi t)$. The discrete approximation $u_h \in \Vh{k_s,k_t}(\EQlin) \circ \Thehst^{-1}$ solves $a(u_h,v_h) = (u_h,v_h)_{\Qhn} + J(u_h,v_h) = (u, v_h)_{\Qhn}$ for all $v_h \in \Vh{k_s,k_t}(\EQlin) \circ \Thehst^{-1}$. where $J(\cdot,\cdot)$ is defined as in \cite{HLP2022} but without the $\frac{1}{h^2}$ scaling. We measure the numerical error in this refinement process in terms of the norms:
$ \| \cdot \|_{L^2(\Ohi{T})}$, and $\| \cdot \|_{L^2( \mathcal{Q}_h )}$
with $\mathcal{Q}_h = \bigcup_n \Qhn$. With these two $L^2$ norms, we choose two reasonable candidate norms for a first investigation, but of course also gradients in space or time derivatives could be included. (See e.g. \cite{preuss18})
We observe the results shown in \cref{num_res2}.
\begin{table}[htb!]\small \centering
 \begin{tabular}{l@{$\;$}l@{~~}l@{}l@{$\;$}l@{$\;$}l@{}l@{$\;$}l@{$\;$} l@{}l@{$\;$}l} \toprule
   &  & \multicolumn{3}{c}{FE blending} & \multicolumn{3}{c}{smooth blend., $w_b =0.1$} & \multicolumn{3}{c}{smooth blend., $w_b =0.2$} \\
  $k$ & $i$ & err: $L^2(T)$, & $L^2(\mathcal{Q}_h)$ & \scriptsize\texttt{(eoc)}\normalsize & err: $L^2(T)$, & $L^2(\mathcal{Q}_h)$ & \scriptsize\texttt{(eoc)}\normalsize & err: $L^2(T)$, & $L^2(\mathcal{Q}_h)$ & \scriptsize\texttt{(eoc)}\normalsize \\ \midrule
  2 & 0 & $1.24\!\cdot\!10^{-1}$ & $6.22\!\cdot\!10^{-2}$ &  & $1.41\!\cdot\!10^{-1}$ & $6.59\!\cdot\!10^{-2}$ &  & $1.51\!\cdot\!10^{-1}$ & $6.75\!\cdot\!10^{-2}$ & \\ 
 & 1 & $3.13\!\cdot\!10^{-2}$ & $1.42\!\cdot\!10^{-2}$ & \scriptsize\texttt{(2.1)}\normalsize & $3.56\!\cdot\!10^{-2}$ & $1.61\!\cdot\!10^{-2}$ & \scriptsize\texttt{(2.0)}\normalsize & $3.30\!\cdot\!10^{-2}$ & $1.57\!\cdot\!10^{-2}$ & \scriptsize\texttt{(2.1)}\normalsize \\ 
 & 2 & $3.85\!\cdot\!10^{-3}$ & $1.65\!\cdot\!10^{-3}$ & \scriptsize\texttt{(3.1)}\normalsize & $4.41\!\cdot\!10^{-3}$ & $1.73\!\cdot\!10^{-3}$ & \scriptsize\texttt{(3.2)}\normalsize & $3.09\!\cdot\!10^{-3}$ & $1.32\!\cdot\!10^{-3}$ & \scriptsize\texttt{(3.6)}\normalsize \\ 
 & 3 & $4.01\!\cdot\!10^{-4}$ & $1.67\!\cdot\!10^{-4}$ & \scriptsize\texttt{(3.3)}\normalsize & $3.39\!\cdot\!10^{-4}$ & $1.36\!\cdot\!10^{-4}$ & \scriptsize\texttt{(3.7)}\normalsize & $2.72\!\cdot\!10^{-4}$ & $1.21\!\cdot\!10^{-4}$ & \scriptsize\texttt{(3.4)}\normalsize \\ 
 & 4 & $4.25\!\cdot\!10^{-5}$ & $1.83\!\cdot\!10^{-5}$ & \scriptsize\texttt{(3.2)}\normalsize & $3.46\!\cdot\!10^{-5}$ & $1.55\!\cdot\!10^{-5}$ & \scriptsize\texttt{(3.1)}\normalsize & $3.11\!\cdot\!10^{-5}$ & $1.41\!\cdot\!10^{-5}$ & \scriptsize\texttt{(3.1)}\normalsize \\ 
 & 5 & $4.99\!\cdot\!10^{-6}$ & $2.17\!\cdot\!10^{-6}$ & \scriptsize\texttt{(3.1)}\normalsize & $4.25\!\cdot\!10^{-6}$ & $1.91\!\cdot\!10^{-6}$ & \scriptsize\texttt{(3.0)}\normalsize & $3.79\!\cdot\!10^{-6}$ & $1.74\!\cdot\!10^{-6}$ & \scriptsize\texttt{(3.0)}\normalsize \\ 
 & 6 & $6.13\!\cdot\!10^{-7}$ & $2.68\!\cdot\!10^{-7}$ & \scriptsize\texttt{(3.0)}\normalsize & $5.34\!\cdot\!10^{-7}$ & $2.41\!\cdot\!10^{-7}$ & \scriptsize\texttt{(3.0)}\normalsize & $4.77\!\cdot\!10^{-7}$ & $2.19\!\cdot\!10^{-7}$ & \scriptsize\texttt{(3.0)}\normalsize \\ 
 \midrule
  4 & 0 & $3.66\!\cdot\!10^{-2}$ & $2.16\!\cdot\!10^{-2}$ &  & $1.03\!\cdot\!10^{-1}$ & $3.78\!\cdot\!10^{-2}$ &  & $1.07\!\cdot\!10^{-1}$ & $4.34\!\cdot\!10^{-2}$ & \\ 
 & 1 & $1.96\!\cdot\!10^{-3}$ & $9.39\!\cdot\!10^{-4}$ & \scriptsize\texttt{(4.5)}\normalsize & $1.18\!\cdot\!10^{-2}$ & $5.88\!\cdot\!10^{-3}$ & \scriptsize\texttt{(2.7)}\normalsize & $2.15\!\cdot\!10^{-2}$ & $8.60\!\cdot\!10^{-3}$ & \scriptsize\texttt{(2.3)}\normalsize \\ 
 & 2 & $5.91\!\cdot\!10^{-5}$ & $2.61\!\cdot\!10^{-5}$ & \scriptsize\texttt{(5.2)}\normalsize & $3.78\!\cdot\!10^{-3}$ & $1.38\!\cdot\!10^{-3}$ & \scriptsize\texttt{(2.1)}\normalsize & $8.18\!\cdot\!10^{-4}$ & $4.62\!\cdot\!10^{-4}$ & \scriptsize\texttt{(4.2)}\normalsize \\ 
 & 3 & $7.13\!\cdot\!10^{-7}$ & $2.33\!\cdot\!10^{-7}$ & \scriptsize\texttt{(6.8)}\normalsize & $1.24\!\cdot\!10^{-4}$ & $3.52\!\cdot\!10^{-5}$ & \scriptsize\texttt{(5.3)}\normalsize & $3.48\!\cdot\!10^{-6}$ & $1.45\!\cdot\!10^{-6}$ & \scriptsize\texttt{(8.3)}\normalsize \\ 
 & 4 & $1.41\!\cdot\!10^{-8}$ & $4.90\!\cdot\!10^{-9}$ & \scriptsize\texttt{(5.6)}\normalsize & $4.82\!\cdot\!10^{-7}$ & $1.84\!\cdot\!10^{-7}$ & \scriptsize\texttt{(7.6)}\normalsize & $9.91\!\cdot\!10^{-8}$ & $4.14\!\cdot\!10^{-8}$ & \scriptsize\texttt{(5.1)}\normalsize \\ 
 & 5 & $3.38\!\cdot\!10^{-10}$ & $1.17\!\cdot\!10^{-10}$ & \scriptsize\texttt{(5.4)}\normalsize & $1.52\!\cdot\!10^{-8}$ & $5.61\!\cdot\!10^{-9}$ & \scriptsize\texttt{(5.0)}\normalsize & $2.90\!\cdot\!10^{-9}$ & $1.15\!\cdot\!10^{-9}$ & \scriptsize\texttt{(5.2)}\normalsize \\ 
 \midrule
  6 & 0 & $2.04\!\cdot\!10^{-2}$ & $8.75\!\cdot\!10^{-3}$ &  & $5.84\!\cdot\!10^{-2}$ & $2.27\!\cdot\!10^{-2}$ &  & $1.30\!\cdot\!10^{-1}$ & $4.23\!\cdot\!10^{-2}$ & \\ 
 & 1 & $7.71\!\cdot\!10^{-4}$ & $3.64\!\cdot\!10^{-4}$ & \scriptsize\texttt{(4.6)}\normalsize & $1.89\!\cdot\!10^{-2}$ & $7.86\!\cdot\!10^{-3}$ & \scriptsize\texttt{(1.5)}\normalsize & $8.72\!\cdot\!10^{-3}$ & $5.16\!\cdot\!10^{-3}$ & \scriptsize\texttt{(3.0)}\normalsize \\ 
 & 2 & $4.09\!\cdot\!10^{-5}$ & $8.95\!\cdot\!10^{-6}$ & \scriptsize\texttt{(5.3)}\normalsize & $2.22\!\cdot\!10^{-3}$ & $1.06\!\cdot\!10^{-3}$ & \scriptsize\texttt{(2.9)}\normalsize & $9.44\!\cdot\!10^{-4}$ & $2.48\!\cdot\!10^{-4}$ & \scriptsize\texttt{(4.4)}\normalsize \\ 
 & 3 & $1.21\!\cdot\!10^{-8}$ & $2.62\!\cdot\!10^{-9}$ & \scriptsize\texttt{(11.7)}\normalsize & $7.49\!\cdot\!10^{-5}$ & $2.49\!\cdot\!10^{-5}$ & \scriptsize\texttt{(5.4)}\normalsize & $9.27\!\cdot\!10^{-7}$ & $3.32\!\cdot\!10^{-7}$ & \scriptsize\texttt{(9.5)}\normalsize \\ 
 & 4 & $2.30\!\cdot\!10^{-11}$ & $4.94\!\cdot\!10^{-12}$ & \scriptsize\texttt{(9.1)}\normalsize & $1.22\!\cdot\!10^{-7}$ & $4.53\!\cdot\!10^{-8}$ & \scriptsize\texttt{(9.1)}\normalsize & $2.05\!\cdot\!10^{-8}$ & $7.48\!\cdot\!10^{-9}$ & \scriptsize\texttt{(5.5)}\normalsize \\ 
 \bottomrule
 \end{tabular}
\caption{Numerical results for stabilized $L^2$ projection problem on unfitted domain approximated isoparametrically. The function $u$ is smooth and the discretisation orders in space and time, and the respective geometry approximation orders are chosen to be $k$, and $i$ denotes simultaneous refinements in space and time. In the columns, different blending options are used: The \protect\FE~blending, and two variants of the smooth blending with $w_b\in\{0.1,0.2\}$ and $s=4$. Estimated orders of convergence are calculated based on the $L^2(\mathcal{Q}_h)$-norm. 
See \cite[Table 2]{arXiv} for $k\in\{1,3,5\}$, and \cite[Table 3]{arXiv} for $k=4,s=10$.} \vspace*{-0.7cm}
\label{num_res2}
\end{table}
As the function $u$ is smooth, and we chose the geometry approximation and discrete function space orders to be $k=k_s=k_t=q_s=q_t$, we expect e.g. for the $L^2(\mathcal{Q}_h)$-norm from \cref{interpollemma1} in combination of the boundedness of the according norms of $\hat u$ by means of \cref{interpollemma2} and \cref{boundedness_result_final} that for the smooth blending
\begin{equation*}
 \| u^e - u_h \|_{\mathcal{Q}_h} \lesssim h^{\min(q_s,k_s) +1}+ \Delta t^{\min(q_t, k_t)+1} \simeq h^{k+1} \simeq \Delta t^{k+1},
\end{equation*}
and for the \FE~blending
\begin{equation*}
 \| u^e - u_h \|_{\mathcal{Q}_h} \lesssim h^{\min(q_s,k_s) +1}+ \Delta t^{\min(q_t, k_t+1)} \simeq h^{k+1} + \Delta t^{k}.
\end{equation*}
Numerically, after a slightly different behavior in the pre-asymptotic regime, we observe that in all cases
\begin{equation*}
 \| u^e - u_h \|_{\mathcal{Q}_h} \lesssim h^{k+1} \simeq \Delta t^{k+1}.
\end{equation*}
This confirms the mathematical analysis for the smooth blending. For the \FE~blending, we obtain similar results, so that the theoretical loss of one order in time is not showing in the particular example. 
\begin{remark}[Approximation in the $L^2(T)$ norm]
In \cref{num_res2}, we have also included results for the respective numerical interpolation errors in the $L^2(T)$ norm.
Numerically, we observe the same convergence orders in this norms as with the $L^2(\mathcal{Q}_h)$ space-time norm, i.e. errors of order $\mathcal{O}(h^{k+1}) \simeq \mathcal{O}(\Delta t^{k+1})$. Our theoretical results would guarantee only $\mathcal{O}(h^{k+\frac{1}{2}}) \simeq \mathcal{O}(\Delta t^{k+\frac{1}{2}})$, cf. \cref{interpollemma1}. 
\end{remark}
\subsection{Boundedness of $\nabla \protect\Psii$ under general space-time refinements for the \protect\FE~blending}
In this section, we want to investigate numerically if specifically the refinement restriction of \cref{ass:dtlesssimh_feblend} is necessary in the case of the \FE~blending. To this end, we fix a large time step, $\Delta t= T$ (one time step) or $\Delta t = T/2$ (two time steps), for the kite geometry and refine the mesh with $h = 0.9 \cdot 0.5^{i_s}$ for $i=i_s=1,2,\dots$. Considering the proof of \cref{lemma_diff_Psi_id_small}, one would expect that potentially for fine meshes, a temporal error scaled by $\frac{1}{h}$ might result in an unbounded summand for e.g. the upper bound on $\| \nabla \Psii\|_\infty$. This, in turn, could become problematic for the interpolation procedure. To investigate whether we can observe this behavior in practice, we calculate a numerical approximation of $\Psii$ as follows: First, we use integration points of a third-order numerical integration on the cut elements $\ThbOne \times I_n$ as sample points to calculate $\d$ explicitly from its definition by a Newton search (based on the given levelset $\phii{}$). Next, we use these sampled points to derive an approximate discrete function $\tilde d$ by solving a \FE~interpolation problem. Finally, we combine this function with the exact $\G$, and apply the blending procedure (which implicitly follows from setting a discrete function on $\ThbOne$)
and calculate $\| \nabla \tilde \Psi \|_\infty$ of this function.
As a numerical compromise, we focus on one intermediate time level for each $\In{n}$ to construct $\tilde \Psi$ and calculate $\| \nabla \tilde \Psi \|_\infty$.
Two results are shown in \cref{fig:PsiTheta_bound_exps}. We observe that asymptotically, $\| \nabla \tilde \Psi \|_\infty$ scales as $\frac{1}{h}$, whereas the error of the $L^2$ interpolation problem remains constant. We conclude that the assumption \cref{ass:dtlesssimh_feblend} seems necessary to avoid unbounded spatial gradients of $\tilde \Psi$ (and hence $\Psi$), although the interpolation error might not deteriorate immediately. Let us stress that very fine spatial meshes were necessary to trigger the problem.
We note in passing that the behavior of $\| \nabla \Theh\|_\infty$ is different from $\Psii$, we observe a norm decay with $\mathcal{O}(h)$. 
\begin{figure}
\centering \vspace*{-0.25cm}
\begin{tikzpicture}[scale=0.75]
    \begin{semilogyaxis}[ xlabel=$i$, title={$\Delta t = T$},title style={at={(0.5,0.83)},draw=black},
    legend entries ={$\| \nabla \tilde \Psi \|_\infty$, $\| \nabla \Theh \|_\infty$, interp. err. $L^2(\mathcal{Q}_h)$}, legend style={anchor=north,legend columns=1, draw=none}, legend pos = south west, x label style={at={(axis description cs:0.95,0.05)},anchor=east},
     x tick label style={yshift=0cm,xshift=0.0em},
      y label style={at={(axis description cs:0.25,0.85)},anchor=east},
     ]

     \addplot table [x index = 0, y index =5] {num_exp/out/conv_kite_DG_ks2_kt1_space_nref8_gamma0.05_sm2_interpol_ro-10.dat};
     \addplot table [x index = 0, y index =6] {num_exp/out/conv_kite_DG_ks2_kt1_space_nref8_gamma0.05_sm2_interpol_ro-10.dat};
     \addplot table [x index = 0, y index =2] {num_exp/out/conv_kite_DG_ks2_kt1_space_nref8_gamma0.05_sm2_interpol_ro-10.dat};
     \addplot[gray, dashed, domain=3:7] {(1/2^(x+0.7)))};
     \addplot[gray, dotted, domain=3:7] {(1/2^(x-6)))^(-1)};
    \end{semilogyaxis}
    \node[scale=0.75, gray] at (6,1.25) {$O(h)$}; \node[scale=0.75, gray] at (5.5,4) {$O(\frac{1}{h})$};
   \end{tikzpicture}
   \begin{tikzpicture}[scale=0.75]
    \begin{semilogyaxis}[ xlabel=$i$, title={$\Delta t = T/2$},title style={at={(0.5,0.83)},draw=black},
    legend entries ={$\| \nabla \tilde \Psi \|_\infty$, $\| \nabla \Theh \|_\infty$, interp. err. $L^2(\mathcal{Q}_h)$}, legend style={anchor=north,legend columns=1, draw=none}, legend pos = south west, x label style={at={(axis description cs:0.95,0.05)},anchor=east},
     x tick label style={yshift=0cm,xshift=0.0em},
      y label style={at={(axis description cs:0.25,0.85)},anchor=east},
     ]

     \addplot table [x index = 0, y index =5] {num_exp/out/conv_kite_DG_ks2_kt1_space_nref8_gamma0.05_sm2_interpol_ro-9.dat};
     \addplot table [x index = 0, y index =6] {num_exp/out/conv_kite_DG_ks2_kt1_space_nref8_gamma0.05_sm2_interpol_ro-9.dat};
     \addplot table [x index = 0, y index =2] {num_exp/out/conv_kite_DG_ks2_kt1_space_nref8_gamma0.05_sm2_interpol_ro-9.dat};
     \addplot[gray, dashed, domain=3:7] {(1/2^(x+0.7)))};
     \addplot[gray, dotted, domain=3:7] {(1/2^(x-8)))^(-1)};
    \end{semilogyaxis}
    \node[scale=0.75, gray] at (6,1.25) {$O(h)$}; \node[scale=0.75, gray] at (5.5,3.5) {$O(\frac{1}{h})$};
   \end{tikzpicture} \vspace*{-0.25cm}
\caption{Numerical studies about $\| \nabla \tilde \Psi \|_\infty$ and $\| \nabla \protect\Theh\|_\infty$ for the \protect\FE~blending in the limit $h \to 0$ and $\Delta t$ large and fixed. In both cases, $\| \nabla \protect\Theh \|_\infty$ scales with $\mathcal{O}(h)$, and $\| \nabla \tilde \Psi \|_\infty$ with $O(\frac{1}{h})$. The correspondig interpolation error remains constant.} \vspace*{-0.7cm}
\label{fig:PsiTheta_bound_exps}
\end{figure}
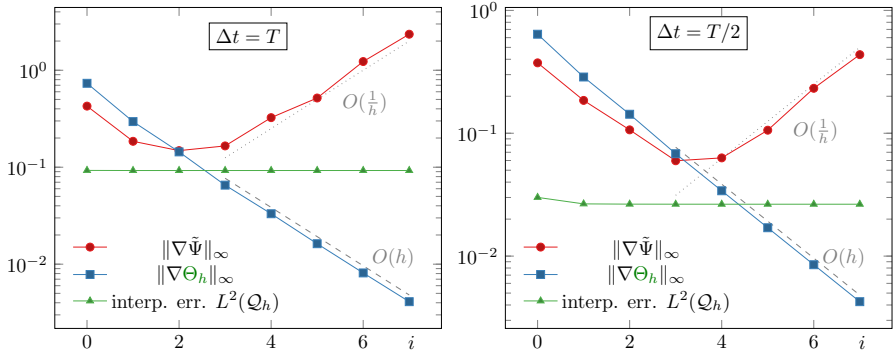


\section{Conclusion and Outlook}  \label{ref:concl}
We have presented a strategy to solve the problem of higher-order numerical integration on implicitly defined (smooth) space-time domains, arising within the context of unfitted space-time discretisations on moving domains. 
Crucial for the strategy are two space-time mesh deformations, a computable mesh deformation $\Thest$ and an ideal mesh deformation $\Psist$, which is not computationally feasible. The focus of this work is the detailed and rigorous analysis of the geometry approximation approach which revealed important properties such as proximity and boundedness which are important to understand the geometrical approximation and interpolation properties of \FE~methods on the deformed meshes.

A major output of this study is collected in \cref{lemma_diff_Theta_h_Psi_small} which gives bounds on the proximity of $\Psist$ and $\Thest$ (in different norms). Notable is the dependence of these results -- especially for the anisotropic cases where $\Delta t \not \simeq h$ -- on the chosen blending from deformations on active elements to the exterior. Only the \emph{smooth} blending yields optimal results here, whereas the \FE~blending can slightly suffer from a mesh-dependent decay of the mesh deformation. Proximity and the effect of the different blendings are also studied and validated numerically.

A second output of major importance in this study is the boundedness of $\Psist$ in higher-order derivatives which directly relates to the application of (optimal) space-time interpolation estimates. The corresponding results are collected in \cref{boundedness_result_final}.
Also here, we observe a dependence on the chosen blending.


Both the geometry approximation and interpolation results can now be taken into account in the rigorous mathematical analysis of unfitted space-time discretisations for different problems. In particular, we will apply the techniques to the DG discretisation in \cite{HLP2022} in upcoming work. Apart from this method for a bulk convection-diffusion problem, the structures proven in this paper can be applied similarly to other problems, such as surface problems (see e.g. \cite{sassreusken23,reusken2024analysis}) or multi-domain problems (see e.g. \cite{preuss18}).\footnote{To see how the corresponding geometry analysis for the merely spatial problem, \cite{LR_IMAJNA_2018}, finds application to a range of different problems, we refer the reader to \cite{L_GUFEMA_2017}, \cite{grande2018analysis}, and \cite{LO_ESAIM_2019}.}

\bibliographystyle{siamplain}
\bibliography{paper}

\appendix

\section*{Selected proofs}
In this section, we collect some proofs that are either tedious, but elementary or similar to other proofs presented in the main part or other literature.
\section{Selected proofs for \cref{lem:d}} \label{sec:more_proofsa} 
\subsection{Proof of \eqref{eq:da} for $(m_t,m_s)=(1,1)$} \label{app:damtms11}
\begin{proof}
We start to consider the first time derivative of \cref{eq_step_234} and bounding the left-hand side (term $A$) by $\mathcal{O}(h + \Delta t^{q_t})$:
\begin{align*}
 &\partial_t (\nabla\! \philin (x,t) \!-\! \nabla\! \phix(\ppi(x,t))) 
 =\underbrace{\partial_t \!\nabla\! \philin (x,t) - \partial_t \!\nabla\! \phix(x,t)}_{\lesssim h + \Delta t^{q_t} \text{by} \cref{diff_grad_phi_philin_dt}} + \underbrace{\partial_t \!\nabla\! \phix(x,t) - \partial_t \!\nabla\! \phix(\ppi(x,t))}_{\lesssim |\phix|_{3,\infty, \U} d \lesssim h^2 + \Delta t^{q_t+1}}
\end{align*}
Now, we have to apply the product rule on all the right-hand-side terms. We start with the first two summands (related to terms $B$ and $C$):
\begin{align*}
 &\partial_t (\nabla \d \G(x,t)^T \nabla \phix(\ppi(x,t)) + \d \nabla \G^T \nabla \phix(\ppi(x,t))) && \\
 =& (\partial_t \nabla \d) \G(x,t)^T \nabla \phix(\ppi(x,t)) + \nabla \d (\partial_t \G (x,t)^T) \nabla \phix(\ppi(x,t)) && (\hspace*{-0.25cm}\text{\hphantom{IIV}I}) + (\hspace*{-0.25cm}\text{\hphantom{IV}II})\\
 & + \nabla \d \G(x,t)^T \partial_t \nabla \phix(\ppi(x,t)) + (\partial_t \d) \nabla \G^T \nabla \phix(\ppi(x,t)) && (\hspace*{-0.25cm}\text{\hphantom{V}III}) + (\hspace*{-0.25cm}\text{\hphantom{II}IV})\\
 &+ \d (\partial_t \nabla \G^T) \nabla \phix(\ppi(x,t)) + \d \nabla \G^T  (\partial_t \nabla \phix(\ppi(x,t))). && (\hspace*{-0.25cm}\text{\hphantom{III}V}) + (\hspace*{-0.25cm}\text{\hphantom{II}VI})
\end{align*}
For I, we observe that the argument in \cref{eq_step_238} can be applied here as well to yield
\begin{equation} \nonumber
 |\text{I}| \simeq (\partial_t \nabla \d) (1 + h^2 + \Delta t^{q_t+1})
\end{equation}
For II, 
we argue with \eqref{eq:da} for $(m_s,m_t)=(1,0)$ to see
\begin{align*}
  |\text{II}| = | \nabla \d (\partial_t \G)^T (\nabla \phix) \!\circ\! \ppi | &\lesssim \|\nabla \d\| \| \partial_t \nabla \phix\|_\infty \left(\| \nabla \phix \|_\infty \! +\! h^2 \!+\! \Delta t^{q_t+1}\right) \label{eq_step_243-new} 
  \lesssim h + \Delta t^{q_t+1} 
\end{align*}
Regarding the third term, we argue as follows:
\begin{align}
 |\text{III}| =& |\nabla d \G(x,t)^T \partial_t \nabla \phix(\ppi(x,t))| \lesssim \| \nabla d \| | \phix |_{1,\infty,\U} \left( \partial_t \nabla \phix(x,t) + \mathcal{O}(h^2 + \Delta t^{q_t+1})\right) \nonumber \\
 &\lesssim \| \nabla d\|  |\phix|_{1,\infty,\U} |\phix|_{2,\infty,\U} \lesssim h + \Delta t^{q_t +1} \text{ by } \cref{eq:da}. \nonumber
\end{align}
For the fourth term, we combine the argument of \cref{eq_step_239} with \cref{eq_step_244} to obtain
\begin{align}
 |\text{IV}| \lesssim |\partial_t d| |\phix|_{2,\infty}(1+ h^2 + \Delta t^{q_1 +1}) \lesssim h^2 + \Delta t^{q_t} \nonumber
\end{align}
For V, we give the following adaptation of the argument in \cref{eq_step_239}:
\begin{align} \nonumber
 |\text{V}| &= |d (\partial_t \nabla \G^T) \nabla \phix(\ppi(x,t)))| \lesssim |d| |\phix|_{3,\infty, \U} (1 + \mathcal{O}(h^2 + \Delta t^{q_t+1})) \lesssim h^2 + \Delta t^{q_t +1} 
\end{align}
by \eqref{eq:da}.
Lastly, for VI, we have
\begin{align}
 |\text{VI}| &= |d \nabla \G^T  (\partial_t \nabla \phix(\ppi(x,t))))| \lesssim |d| |\phix|_{2,\infty,\U} (\partial_t \nabla \phix(x,t) + h^2 + \Delta t^{q_t+1}) \nonumber \\
 &\lesssim |d| |\phix|_{2,\infty,\U}^2 \lesssim h^2 + \Delta t^{q_t +1} \text{ by } \eqref{eq:da}.
\end{align}
Next, we investigate the time derivatives from the two remaining summand on the right-hand side of \cref{eq_step_234}:
\begin{align*}
  & |\partial_t ( \nabla b (\philin - \phix)(x,t) + b (\nabla \philin - \nabla \phix)(x,t))| \\
 \leq& |(\partial_t \nabla b) (\philin - \phix)| + |\nabla b (\partial_t \philin - \partial_t \phix)| + |\partial_t b (\nabla \philin - \nabla \phix)| + |b (\partial_t \nabla \philin - \partial_t \nabla \phix)| \\
 \lesssim&~ h + \Delta t^{q_t},
\end{align*}
by the known techniques/ previous results.
Taking all these results together, we obtain
\begin{equation}
| (\partial_t \nabla d) (1 + h^2 + \Delta t^{q_t + 1})| \lesssim h + \Delta t^{q_t},
\end{equation}
which finishes the proof of \cref{eq:da} for $(m_t,m_s)=(1,1)$.
\end{proof}

\subsection{Proof of \eqref{eq:db} for $m_t\geq 1$} \label{app:eqdbmtgt1}
\begin{proof}
We split the cases $m_t=1$ and $m_t>1$. 

\underline{Proof of \eqref{eq:db} for $m_t=1$}:\\
We exploit the chain rule on \eqref{char:pi} similarly as in the proof of \eqref{eq:da}:
\begin{align}
 & |\partial_t \philin  - (\partial_t \phix) \circ \ppi | \\
 & \leq |(\partial_t \d) \!\!\!\!\! \underbrace{\G^T (\nabla \phix) \circ \ppi}_{\overset{\cref{eq_step_238}}{\simeq} 1 + h^2 + \Delta t^{q_t + 1}} \!\!\!\!\!|  + |\d (\partial_t \G)^T (\nabla \phix) \circ \ppi \!| \label{eq_step_241}
 + |\partial_t \b \!\! \underbrace{(\philin \!-\! \phix)}_{\lesssim h^2 + \Delta t^{q_t+1}}\!\! | \!+\! | \b \underbrace{(\partial_t \philin \!-\! \partial_t \phix)}_{\lesssim h^2 + \Delta t^{q_t}}\!| \nonumber 
\end{align}
For the l.h.s. there obviously holds
\begin{align*}
  |\partial_t \philin - (\partial_t \phix) \circ \ppi | &\leq \!\!\!\!\!\!\underbrace{ |\partial_t \philin - \partial_t \phix|}_{\lesssim h^2 + \Delta t^{q_t} \textnormal{ by } \eqref{diff_grad_phi_philin_dt}} \!\!\!\!\!\! + \quad \underbrace{|(\partial_t \phix) - (\partial_t \phix) \circ \ppi|}_{\lesssim \| \partial_t \nabla \phix\|_{\infty} \d \lesssim h^2 + \Delta t^{q_t +1}} \lesssim h^2 + \Delta t^{q_t}.
\end{align*}
For the remaining summand, we argue with \eqref{eq:da} to see
\begin{align}
  |\d (\partial_t \G)^T (\nabla \phix) \circ \ppi| &\lesssim |\d| \| \partial_t \nabla \phix\|_\infty \left(\| \nabla \phix \|_\infty  + h^2 + \Delta t^{q_t+1}\right) \label{eq_step_243} 
  \lesssim h^2 + \Delta t^{q_t+1} 
\end{align}
Taking these results together, we obtain
\begin{equation}
 |\partial_t \d | (1+ h^2 + \Delta t^{q_t + 1}) \lesssim h^2 + \Delta t^{q_t}, \label{eq_step_244}
\end{equation}
which implies the claim in \cref{eq:db} for $m_t = 1$.

\underline{Proof of \eqref{eq:db} for $m_t>1$}:\\
In regards to $2 \leq m_t \leq q_t+1 $, we proceed similarly as in the proof 
of \eqref{eq:db} 
for $m_t=1$. First, we note that for the time derivatives of $ \b ( \phix - \philin)$, it holds with \eqref{diff_grad_phi_philin_dt} and \cref{ass:blending}
\begin{align*}
 \big| \partial_t^{m_t} \left( \b( \phix - \philin) \right) \big| &= \Big|\sum_{j=0}^{m_t} \binom{m_t}{j} (\partial_t^j \b) (\partial_t^{m_t - j} (\phix - \philin)) \Big| 
  \lesssim h^2 + \Delta t^{q_t + 1 - m_t}
\end{align*}
With $\phix \circ \ppi = \philin + \b (\phix-\philin)$ and hence 
$\partial_t^{m_t} (\phix \circ \ppi) = \partial_t^{m_t} \philin + \partial_t^{m_t} (\b (\phix-\philin) )$ this implies together with another application of \eqref{diff_grad_phi_philin_dt}
\begin{equation*}
 |\partial_t^{m_t} (\phix \circ \ppi - \phi)| \lesssim |\partial_t^{m_t} (\phi^{\text{lin}} - \phix)| + h^2 + \Delta t^{q_t + 1 - m_t} \lesssim h^2 + \Delta t^{q_t + 1 - m_t}.
\end{equation*}
Now we want to exchange $(\partial_t^{m_t} \phix)$ with $(\partial_t^{m_t} \phix) \circ \ppi$ and note that $\| \nabla \partial_t^{m_t} \phi \|_\infty \lesssim 1$, so that with \cref{eq:da} for $m_t=m_s=0$ we obtain
\begin{equation}
 |\partial_t^{m_t} (\phix \circ \ppi) -  (\partial_t^{m_t} \phix) \circ \ppi| \lesssim h^2 + \Delta t^{q_t + 1 - m_t}. \label{eq:dtmtphicircpi}
\end{equation}
We exploit this in now showing the result \cref{eq:db} for $2 \leq m_t \leq q_t+1$ by induction. So assume that \cref{eq:db} holds for $0,\dots,m_t-1$. By induction (over $m_t$) we can easily calculate the following equation
\begin{align*}
 \partial_t^{m_t} (\phix \circ \ppi) - (\partial_t^{m_t} \phix) \circ \ppi = \sum_{j=0}^{m_t-1} \partial_t^{j} \left[ (\nabla \partial_t^{m_t-1-j} \phix)\circ \ppi \cdot ( \partial_t (\d \G)) \right].
\end{align*}
Next, we want to separate the summand for $j=m_t-1$ and the remaining terms where we note that with \eqref{eq:dtmtphicircpi} we already have a proper bound for the l.h.s. of the equation. So, let's consider the summands for $j<m_t-1$: $\partial_t^{j} \left[ (\nabla \partial_t^{m_t-1-j} \phix)\circ \ppi \cdot ( \partial_t (\d \G)) \right]$. According to the Leibniz rule, we will consider the products of terms for $j_1+j_2=j$ where the $j_1$-th time derivative acts on $(\nabla \partial_t^{m_t-1-j} \phix)\circ \ppi$ and the ones where 
the $j_2$-th time derivative acts on $\partial_t (\d \G)$. For the former ones, we note that repeated application of the chain rule yields
\begin{align*}
|\partial_t^{j_1} ( (\nabla \partial_t^{m_t-1-{j_1}} \phix)\circ \ppi )|
& \!\lesssim\!  |(\nabla \partial_t^{m_t-1} \phix) \circ \ppi | + \big| \sum_{k=1}^{j_1} (\partial_t^{m_t-1-k} D^{k+1} \phix) \circ \ppi ~ \sum_{k=1}^{j_1} \partial_t^{k} (\d \G) \big|
\\[-3ex]
& \!\lesssim\! 1 + \big| \sum_{k=1}^{j_1} \partial_t^{k} (\d \G)  \big|
\end{align*}
For the latter ones we have $\partial_t^{j_2} \partial_t (\d \G) = \partial_t^{j_2+1} (\d \G)$
so that the products of the two terms are bounded by $ \partial_t^{j_2+1} (\d \G) + \sum_{k=1}^{j_1} \partial_t^{k} (\d \G) \partial_t^{j_2+1} (\d \G) \lesssim \partial_t^{m_t-1} (\d \G)$. With the regularity of $G$ the crucial remaining terms are the derivates on $d$ for which we can apply the induction hypothesis so that the summands for $j<m_t-1$ are bounded by
$h^2 + \Delta t^{q_t + 2 - m_t}$.
We hence find
\begin{equation}
 |\partial_t^{m_t-1} \left[ (\nabla \phix)\circ \ppi \cdot ( \partial_t (d G)) \right] | \lesssim h^2 + \Delta t^{q_t + 1 - m_t}.
\end{equation}
Now, applying the Leibniz rule again we obtain 
the summand 
$(\nabla \phix) \circ \ppi \cdot \partial_t^{m_t-1} (dG)$ and the remaining summands have the same bounds as the previously treated terms. Splitting the remaining term further as
$(\nabla \phix) \circ \ppi \cdot \partial_t^{m_t-1} (dG) = (\nabla \phix) \circ \ppi \cdot G ~~ \partial_t^{m_t-1} d + R$ with the remainder term $R$, we notice that $R$ is again bounded as before. With 
$(\nabla \phix) \circ \ppi \cdot G \simeq 1$ we hence finally have
\begin{equation*}
  | \partial_t^{m_t} d | \lesssim h^2 + \Delta t^{q_t + 1 - m_t}.
\end{equation*}
which implies the claim (by induction).
\end{proof}

\subsection{Proof of \eqref{eq:dc}} \label{app:eqdc}
\begin{proof}
  Fix $T \in \ThG$, $\alpha_s \in \mathbb{N}_0^d$ and $\alpha_t \in \mathbb{N}_0$ such that $|\alpha_s| \leq q_s+1, \alpha_t \leq q_t+1$. Then, we aim to show
  \begin{equation}
   \| D^{(\alpha_s, \alpha_t)} d\|_{\infty,T \times \In{n}} \lesssim 1.
  \end{equation}
  To this end, we introduce the function $F(x,t,y) = \phix(x + y \G(x,t), t) - \philin(x,t) + \b(x,t) (\philin(x,t) - \phix(x,t))$. Then the function $\d(x,t) = y(x,t)$ solves the implicit equation $F(x,t,y(x,t)) = 0$.
 
  By the implicit function theorem, then we obtain for all $\alpha \in \mathbb{N}_0^{d+1}$ with $|\alpha| = 1$
  \begin{equation}
   D^\alpha d = - \underbrace{\left( \frac{\partial F}{\partial y}(x,t,\d(x,t)) \right)^{-1}}_{:=A(x,t)} \cdot D^\alpha_{(x,t)} F(x,t,\d(x,t)) = -A \cdot D^\alpha_{(x,t)} F(x,t,\d(x,t)).
  \end{equation}
  Hence, for all $\alpha \in \mathbb{N}_0^{d+1}$ with $0 < |\alpha| \leq q_s + q_t +2$ (the boundedness of the zeroth and even first derivatives of $d$ follow from the even stronger bounds in the previous equations), we find $\alpha_1, \alpha_r \in \mathbb{N}_0^{d+1}$ s.t. $|\alpha_1| = 1$ and $\alpha = \alpha_1 + \alpha_r$, so that we can conclude,
  \begin{align}
   D^\alpha d &= D^{\alpha_r}( D^{\alpha_1} d ) = - D^{\alpha_r} ( A \cdot D^{\alpha_1}_{(x,t)} F(x,t,\d(x,t)) )\\
   &= -\sum_{\nu \leq \alpha_r} \binom{\alpha}{\nu} (D^\nu A) \cdot D^{\alpha-\nu} ( D^{\alpha_1}_{(x,t)} F(x,t,\d(x,t)) )
  \end{align}
  by the application of a Leibniz formula for multi-index derivatives. We are left with the task of showing boundedness of $D^\alpha A$ for all $\alpha \in \mathbb{N}_0^{d+1}$ s.t. $|\alpha| \leq q_s + q_t +1$ and of $D^\alpha F$ for $\alpha \in \mathbb{N}_0^{d+1}$ with $|\alpha| \leq q_s + q_t + 2$.
  
  Concerning the boundedness of $A$ itself, we start with the following observation in regard to its denominator:
  \begin{align*}
    \frac{\partial F}{\partial y}(x,t,\d(x,t)) &= \nabla \phix (x+ \d(x,t) \G(x,t), t) \nabla \phix (x,t) = \| \nabla \phix (x,t) \|_2^2 + \mathcal{O}(h^2 + \Delta t^{q_t +1})
  \end{align*}
  Hence, for sufficiently small meshes and time-steps, this expression is not only $\lesssim 1$ (as would be also implied by the boundedness of the derivatives of $F$ shown below) but even $\simeq 1$. We can conclude that $A \lesssim 1$ and even higher negative powers, $(\frac{\partial F}{\partial y}(x,t,\d(x,t)))^{-k} \lesssim 1$ for $k\in \mathbb{N}$. This is helpful to show the general boundedness of $D^\alpha A$, as we can write $A = i \circ \frac{\partial F}{\partial y}(x,t,\d(x,t))$, where $i\colon x \mapsto 1/x$. Then $D^\alpha A$ can be calculated by a higher-dimensional Faà di Bruno formula such as \cite[Lemma 3]{CIARLET1972217}, leaving us with a sum of combinatorial factors and (higher) derivatives of $i$ and $\frac{\partial F}{\partial y}(x,t,\d(x,t))$. Then the boundedness follows from the previous observations with the bounds on (derivatives of) $F$.
  
  Now in regards to $D^\alpha F$, we discuss the summands individually in the order: 1) $\philin(x,t)$, 2) $\b(x,t) (\philin(x,t) - \phix(x,t))$, 3) $\phix(x + y \G(x,t), t)$.
 
  Regarding 1): We note that $D^\alpha \philin(x,t) = 0$ if a higher spatial derivative than the first is considered. Else, $\| D^\alpha \philin(x,t) \| \leq \| D^\alpha \phix(x,t) \| - \| D^\alpha \philin(x,t)  - D^\alpha \phix(x,t)\|$, where the first summand is bounded by $\phix \in C^{q_s+q_t+2}(\U)$ and the second can be bounded by \cref{diff_grad_phi_philin_dt}.
 
  Regarding 2): From the previous point, we take the result $\| D^\alpha \philin \|_\infty \lesssim 1$ for all $\alpha$ with $|\alpha|<q_s+q_t+2$. Moreover, $D^\alpha b$ and $D^\alpha \phix$ are bounded similarly by the assumptions $b,\phix \in C^{k_s + k_t +2}$. This suffices to show
  \begin{equation}
   \| D^\alpha  (\b (\philin - \phix))\|_\infty \lesssim 1,
  \end{equation}
  as all terms that appear in a Leibniz formula for multi-index derivatives as given above are bounded.
 
  Regarding 3): So it remains to consider the summand $(x,t,y) \mapsto \phix(x+y \G(x,t),t)$, which means $D^\alpha_{(x,t)} \phix(x+y \G(x,t),t)$, where the evaluation point $y = \d(x,t)$ will be chosen. We define the helper function $\pi\colon (x,t) \mapsto (x + y \G(x,t),t)$, so that we have to bound $D^\alpha_{(x,t)} \phix \circ \pi$. To this end, we again apply \cite[Lemma 3]{CIARLET1972217}, so that bounds on the corresponding functions $\phix$ and $\pi$ are sufficient. In regards to $\phix$, the assumption $\phix \in C^{q_t + q_s +2}(\U)$ suffices. For $\pi$ we note that as $\G(x,t) = \nabla \phix(x,t)$, the previously mentioned assumption implies that $\G(x,t)\in C^{q_t + q_s +1}$, so that the result follows.
\end{proof}

\section{Further selected proofs} \label{sec:more_proofsb}
\subsection{Proof of \eqref{eq:grad-PsidtG-PsiG}}
\label{app:gradPsidtacc}
\begin{proof}
We prove the bound at $\ti{i}$, $\Vert \nabla (\PsidtGi{i} - \PsiGi{i}) \Vert_{\infty,\OG} \lesssim h^{q_s} + \Delta t^{q_t+1}$. There holds 
  \begin{align*}
  |&\nabla (\PsidtGi{i} - \PsiGi{i})|
  = 
  |\nabla (\ddti{i} \Gdti{i} - \di{i} \Gi{i})| \\
  & = 
  |\Gdti{i} - \Gi{i}| |\ddti{i}| + |\Gi{i}| |\nabla(\ddti{i} - \di{i})|
  + |\ddti{i} - \di{i}| |\nabla \Gi{i}|
  + |\di{i}| |\nabla (\Gdti{i} - \Gi{i})|
  \end{align*}
  From \eqref{diff_phid_phi} (with $m_s=1$ and $m_s=2$), \eqref{eq:ddt}, \eqref{eq:PsidtG-PsiG}, \eqref{eq:cnphi} and \eqref{eq:d} we can bound the first, third and fourth summand by the desired bound. Further we have $|\Gi{i}| \lesssim 1$ due to \eqref{eq:cnphi} so that it only remains to bound $|\nabla (\ddti{i} - \di{i})|$.
  
  Again, we start with \eqref{ddef} and \eqref{ddtdef} and fix a time $\ti{i}$, $i=0,..,q_t$ and $x \in \OG$.
     \begin{align*}
      (\phii{i} \circ \PsiGi{i})(x)  &= R(x) \coloneqq (1- \bi{i}(x)) \philini{i}(x) + \bi{i}(x) \phii{i}(x) \text{ and } \\
      (\phidti{i} \circ \PsidtGi{i})(x)  &= R_{\Delta t}(x) \coloneqq (1- \bi{i}(x)) \philini{i}(x) + \bi{i}(x) \phidti{i}(x). 
   \end{align*}
   where we defined $R$ and $R_{\Delta t}$ for later use.
    For $x \in \OG$ we set $y = \PsiGi{i}(x)$ and $y_{\Delta t} = \PsidtGi{i}(x)$
    and obtain from the chain rule, recalling $\Gi{i} = \nabla \phii{i}$ and some elementary calculus 
  \begin{align*}
    \nabla& R(x) = \nabla (\phii{i} \circ \PsiGi{i})(x) \\ 
    & = \nabla \phii{i}(y) + (\Gi{i}(x) \cdot \Gi{i}(y)) \nabla \di{i}(x)  + \di{i}(x) \nabla \Gi{i}(x) \nabla \phii{i}(y) \\
    \nabla& R_{\Delta t}(x) = \nabla (\phidti{i} \circ \PsidtGi{i})(x) \\
    & = \nabla \phidti{i}(y_{\Delta t}) + (\Gdti{i}(x) \cdot \Gdti{i}(y_{\Delta t})) \nabla \ddti{i}(x) + \ddti{i}(x) \nabla \Gdti{i}(x) \nabla \phidti{i}(y_{\Delta t})
  \end{align*}
  Subtracting both equations and moving the terms involving $\nabla \di{i}(x)$ and $\nabla \ddti{i}(x)$ to one side we obtain the equation $A=B-C-D$ with the expressions 
  \begin{align*}
    A & \coloneqq (\Gi{i}(x) \cdot\Gi{i}(y)) \nabla \di{i}(x) -
    (\Gdti{i}(x)\cdot\Gdti{i}(y_{\Delta t})) \nabla \ddti{i}(x), \\
    B &\coloneqq \nabla (R - R_{\Delta t})(x), \qquad \qquad C \coloneqq \nabla \phii{i}(y) - \nabla \phidti{i}(y_{\Delta t}), \\ D & \coloneqq (\di{i}(x) \nabla \Gi{i}(x) \nabla \phii{i}(y) - \ddti{i}(x) \nabla \Gdti{i}(x) \nabla \phidti{i}(y_{\Delta t}))
  \end{align*}  
  We will start splitting the term $A$ further with
  \begin{align*}
    A&= (\Gi{i}(x) \!\cdot\!\Gi{i}(y)) \nabla \di{i}(x) \!-\!
    (\Gdti{i}(x)\!\cdot\!\Gdti{i}(y_{\Delta t})) \nabla \ddti{i}(x) \\
    & = \underbrace{(\Gi{i}(x) \!\cdot\!\Gi{i}(y)) \nabla (\di{i} - \ddti{i})(x)}_{\eqqcolon A_1} - 
    \underbrace{(\Gi{i}(x) \!\cdot\!\Gi{i}(y)-\Gdti{i}(x)\!\cdot\!\Gdti{i}(y_{\Delta t})) \nabla \ddti{i}(x)}_{\eqqcolon A_2}
  \end{align*}
  For $A_1$ we note that $|y-x|=|\PsiGi{i}(x)-x| = |(\di{i} \Gi{i})(x)|$ and hence with \eqref{eq:da} $|y-x|\lesssim h^2 + \Delta t^{q_t +1}$. Further we have $|\nabla \Gi{i}|=|D^2 \phii{i}|\lesssim 1$ from \eqref{eq:cnphi} so that
  $\Gi{i}(y) = \Gi{i}(x) + \mathcal{O}(h^2 + \Delta t^{q_t +1})$.
  This yields 
  $$
  \Gi{i}(x) \cdot \Gi{i}(y) = \| \nabla \phii{i}(x) \|^2 ( 1 + \mathcal{O}(h^2 + \Delta t^{q_t +1})) \stackrel{\eqref{eq:cnphi}}{\simeq} 1,
  $$
  so that we can rearrange the balance $A = B - C - D$ to 
  $$
  |\nabla(\di{i} - \ddti{i})(x)| \lesssim | A_2 | + | B | + | C | + | D |
  $$
  We will now bound each of these terms separately with $\lesssim h^{q_s} + \Delta t^{q_t +1}$ which will yield the desired result.
  \\ \underline{\underline{$A_2$:}}
  With $|y-y_{\Delta t}|=|(\PsiGi{i}-\PsidtGi{i})(x)|$ and the bound from 
  \eqref{eq:PsidtG-PsiG} we have
  $$
  \Gdti{i}(y_{\Delta t}) = \Gdti{i}(y) + \mathcal{O}(h^{q_s+1}+\Delta t^{q_t+1}) \cdot \|D \Gdti{i} \|_{\infty,\OG}
  $$
  where $\|D \Gdti{i} \|_{\infty,\OG} \lesssim 1$ with \eqref{diff_phid_phi} and $m_s=1$ and \eqref{eq:cnphi}. From \eqref{diff_phid_phi} we also obtain $\Gdti{i}(y) = \Gi{i}(y) + \mathcal{O}(h^{q_s} + \Delta t^{q_t+1})$ so that together with \eqref{eq:cnphi} we have 
  $$
  \Gdti{i}(x) \cdot \Gdti{i}(y_{\Delta t}) = \Gi{i}(x) \cdot \Gi{i}(y) + \mathcal{O}(h^{q_s}+\Delta t^{q_t+1}) 
  $$
  With \eqref{eq:ddt} (for the derivative) we then make the rough estimate $A_2 \lesssim h^{q_s}+ \Delta t^{q_t+1}$. 
  \\ \underline{\underline{$B$:}}
  A bound for $B$ follows easily with \cref{ass:blending} and \eqref{diff_phid_phi}:
  $$
  |\nabla(R-R_{\Delta t})(x)| \lesssim \underbrace{|\nabla \bi{i}(x)|}_{\lesssim 1} \underbrace{|(\phii{i}-\phidti{i})(x)|}_{\lesssim h^{q_s+1} + \Delta t^{q_t+1}}
  + \underbrace{|\bi{i}(x)|}_{\lesssim 1} \underbrace{|\nabla(\phii{i}-\phidti{i})(x)|}_{\lesssim h^{q_s} + \Delta t^{q_t+1}} \lesssim h^{q_s} + \Delta t^{q_t+1}
  $$
  \\ \underline{\underline{$C$:}}
  For $C$ we again apply \eqref{diff_phid_phi} with $m_s=1$ to obtain
  $$
  |\nabla \phii{i}(y) \!-\! \nabla \phidti{i}(y_{\Delta t})|
  \lesssim \underbrace{|\nabla (\phii{i} - \phidti{i})(y)|}_{\lesssim h^{q_s}+\Delta t^{q_t+1}}
  + \underbrace{|\nabla \phidti{i}(y) \!-\! \nabla \phidti{i}(y_{\Delta t})|}_{
    \lesssim \| D^2 \phidti{i} \|_{\infty,\Ubar} |y-y_{\Delta t}|}
  $$
  Now, we have $\| D^2 \phidti{i} \|_{\infty,\Ubar} \lesssim \| D^2 \phii{i} \|_{\infty,\Ubar} \lesssim 1$ due to \eqref{diff_phid_phi} with $m_s=2$ and \eqref{eq:cnphi} and $|y-y_{\Delta t}| \lesssim h^{q_s+1} + \Delta t^{q_s+1}$ with 
  \eqref{eq:PsidtG-PsiG}, so that also $|C|$ has the desired upper bound.
  \\ \underline{\underline{$D$:}}
  Telescoping $a_1 \bOne c_{1,1} - a_2 \bTwo c_{2,2} = (a_1-a_2) \bOne c_{1,1} + a_2 (\bOne-\bTwo) c_{1,1} + a_2 \bTwo (c_{1,1}-c_{1,2}) + a_2 \bTwo (c_{1,2}-c_{2,2})$ we obtain
  \begin{align*}
  (\di{i}(x) \nabla \Gi{i}(x) & \nabla \phii{i}(y) \!-\! \ddti{i}(x) \nabla \Gdti{i}(x) \nabla \phidti{i}(y_{\Delta t})) = D_1 + D_2 + D_3 + D_4 \text{ with }\\
  &D_1 = (\di{i}-\ddti{i})(x) \nabla \Gi{i}(x) \nabla \phii{i}(y) \\ 
  &D_2 = \ddti{i}(x) \nabla (\Gi{i} - \Gdti{i})(x) \nabla \phii{i}(y)  \\
  &D_3 = \ddti{i}(x) \nabla \Gdti{i}(x) (\nabla \phii{i}(y) -  \nabla \phidti{i}(y)) \\
  &D_4 = \ddti{i}(x) \nabla \Gdti{i}(x) (\nabla \phidti{i}(y) -  \nabla \phidti{i}(y_{\Delta t}))
  \end{align*}
  With \eqref{eq:di-ddti} and \eqref{eq:cnphi} we find $|D_1| \lesssim h^{q_s+1} + \Delta t^{q_t+1}$. Next, with \eqref{eq:ddt} and \eqref{diff_phid_phi} (with $m_s=2$) and \eqref{eq:cnphi} we get 
  $|D_2| \lesssim h^2 \cdot (h^{q_s-1}+\Delta t^{q_t+1}) \cdot 1$ which suffices for the desired bound. Next, with \eqref{eq:ddt}, \eqref{eq:cnphi} and \eqref{diff_phid_phi} (with $m_s=1$) we also have $|D_3| \lesssim h^2 \cdot 1 \cdot (h^{q_s}+\Delta t^{q_t+1})$ which again suffices. 
  Finally, with \eqref{eq:ddt}, \eqref{eq:cnphi}, 
  $\| D^2 \phidti{i} \|_{\infty,\Ubar} \lesssim \| D^2 \phii{i} \|_{\infty,\Ubar} \lesssim 1$ due to \eqref{diff_phid_phi} with $m_s=2$ and \eqref{eq:cnphi} 
  and $|y-y_{\Delta t}| \lesssim h^{q_s+1} + \Delta t^{q_s+1}$ with 
  \eqref{eq:PsidtG-PsiG},
  we have $|D_4| \lesssim h^2 \cdot 1 \cdot (h^{q_s+1}+\Delta t^{q_t+1})$ which concludes the proof.
\end{proof}

\subsection{Proof of \cref{G_Gh_diff_lemma} for \texorpdfstring{$G_h = P_h^\Gamma \nabla \phi_h$}{projected normal field}}
\label{app:proofG_Gh_diff_lemma}
\begin{proof}
  Let $\Gh = \PhG \nabla \phih$. We will make use of an inverse inequality to deal with the spatial derivative of $\PhG \nabla \phih$. To circumvent using the inverse inequality on terms that involve temporal resolution quantities 
  we apply a triangle inequality introducing $\PhG \nabla \phidt$ and $\nabla \phidt$. Exploit smoothness of $\phidt$ so that $\PhG \nabla \phidt = \Is{q_s} \nabla \phidt$ we obtain
  \begin{align*}
   & \| D^{m} (\PhG \nabla \phih - \nabla \phix)\|_{\infty, \ThGQ} \\
   \leq & \underbrace{\| D^{m} \PhG \nabla (\phih\!-\! \phidt) \|_{\infty, \ThGQ}}_{=:A}  + \underbrace{\| D^{m} (\Is{q_s} \nabla \phidt\!-\!\nabla \phidt)\|_{\infty, \ThGQ}}_{=:B} + \underbrace{\| D^{m} \nabla (\phidt\!-\!\phix) \|_{\infty, \ThGQ}}_{=:C} 
  \end{align*}
 For $A$ we apply an inverse inequality, exploit continuity of $\PhG$ and \eqref{diff_phid_phi}:
 \begin{equation*}
 A \lesssim h^{-m} \| \PhG \nabla (\phih\!-\! \phidt) \|_{\infty, \ThGQ}   
 \lesssim h^{-m} \| \nabla (\phih\!-\! \phidt) \|_{\infty, \ThGQ}   
 \lesssim h^{q_s-m} 
 \end{equation*}
 For $B$ we apply standard interpolation estimates 
 \begin{equation*}
   B \lesssim h^{q_s-m} \| D^{q_s+1} \phidt \|_{\infty, \ThGQ}   
 \lesssim h^{q_s-m} 
 \end{equation*}
 where the boundedness (by a constant) of $D^{q_s+1} \phidt$ follows from 
 \eqref{diff_phid_phi} with $m_s=q_s+1$ and the regularity of $\phi$.
 Finally, for $C$ we can directly apply \eqref{diff_phid_phi} with $m_s=m+1$ to get $C \lesssim h^{q_s-m}$. This concludes the proof.
 \end{proof}

 \subsection{Proof of \cref{diff_phi_phihn_spaceD_special}}\label{app:diff_phi_phihn_spaceD_special}
 \begin{proof}
 We apply a several-variable version of Taylor's theorem with Lagrange remainder on $\ET \phihi{i} - \phii{i}$ where we recall that $\ET \phihi{i} \in \mathcal{P}^{q_s}$: We obtain the existence of $\xi \in (0,1)$ s.t.
  \begin{align*}
 |D^m (\ET \phihi{i}& - \phii{i})(y)| = \Big| {\textstyle \sum_{r=0}^{q_s-m}} {(r!)}^{-1} D^{r+m}(\ET \phihi{i} - \phii{i})(x)(\overbrace{y-x,\dots,y-x}^{r\text{ times}})  \\
  & -  (q_s\!+\!1\!-\!m)!^{-1} D^{q_s+1}\phii{i}(x + \xi(y-x)) (\underbrace{y-x, \dots, y-x}_{(q_s+1-m)\text{ times}}) \Big| \\
  \lesssim& {\textstyle \sum_{r=0}^{q_s}} |D^{r}(\phihi{i} - \phii{i})|_{\infty,T} |y-x|^r + |y-x|^{q_s+1-m}.
\end{align*}
We have $|D^{r}(\phihi{i} - \phii{i}) |_{\infty,T} \lesssim h^{q_s + 1 - r} + \Delta t^{q_t + 1}$ by \cref{diff_phi_phihn_spaceD} and $|y-x|\lesssim h$ so that  \eqref{eq:diff_phi_phihn_spaceD_special:1} follows. 
The second equation follows along the same lines noting that $\ET$ only acts on the spatial variable and that 
\eqref{diff_phi_phihn_spaceD} also provides the corresponding bounds for the temporal derivatives for any $t\in\In{n}$. Finally, for the third equation, we can apply the same approach to get 
\begin{align*}
 |D^m (\ET \phihi{i} - \phidti{i})(y)| 
  \lesssim {\textstyle \sum_{r=0}^{q_s}} |D^{r}(\phihi{i} - \phidti{i})|_{\infty,T} h^r + h^{q_s+1-m} \lesssim h^{q_s+1-m}
\end{align*}
with \eqref{diff_phid_phih}. 
\end{proof}

\subsection{Proof of the second estimate in \eqref{d_h_lemma_eq2}} \label{app:d_h_lemma_eq2}
\begin{proof}{
  We treat the second estimate in \cref{d_h_lemma_eq2}. For this, we take the gradient of both sides of \cref{def_dhi}.
  We start with the r.h.s. term and take $x$ as argument for $\philini{i}$, $\bi{i}$ and $\phihi{i}$, respectively:
  \begin{align*}
  &(1-\bi{i}) \nabla \philini{i} + \bi{i} \nabla \phihi{i}
  - \nabla \bi{i} \philini{i} + \nabla \bi{i} \phihi{i} 
   \\ &= \nabla \philini{i} + \bi{i} \nabla ( \phihi{i} - \philini{i} ) + \nabla \bi{i} (\phihi{i} - \philini{i}) 
  \stackrel{\text{Cor.}\ref{diff_phihn_philin}}{=} \nabla \philini{i} + \mathcal{O}(h),
 \end{align*}
 where in the last step we used $\bi{i} \in [0,1]$ and $|\nabla \bi{i}|\lesssim 1$ and \eqref{eq:diff_phihn_philin} from \cref{diff_phihn_philin} with $m=0$ and $m=1$, respectively.
 It is now of interest to compare this term with the result of taking the gradient on the other side of \cref{def_dhi}, denoting $y_h := x+ (\dhi{i} \Ghi{i})(x)$
 \begin{align*}
 \underbrace{\nabla \philini{i}(x)}_{A} + \mathcal{O}(h) = & \underbrace{\nabla (\ET \phihi{i})(y_h)}_{B} + \nabla \dhi{i}(x) \underbrace{(\nabla (\ET \phihi{i})(y_h)) \cdot \Ghi{i}(x)}_{C} \\
 & + \underbrace{\dhi{i}(x) (\nabla \Ghi{i}(x)) \cdot \nabla (\ET \phihi{i})(y_h)}_{D}
 \end{align*}
 We will first show that $|A-B| = \mathcal{O}(h+ \Delta t^{q_t+1})$.
 We apply \eqref{eq:diff_phi_phihn_spaceD_special:1} with $m=1$ from \cref{diff_phi_phihn_spaceD_special}
 to obtain $\nabla (\ET \phihi{i})(y_h) = \nabla \phii{i}(y_h) + \mathcal{O}(h^{q_s} + \Delta t^{q_t+1})$. 
 With $|y_h - x| \lesssim |\dhi{i}(x)| |\Ghi{i}(x)| \lesssim h^2 + \Delta t^{q_t+1}$,
 which follows from the first estimate in \cref{d_h_lemma_eq2},
 \eqref{diff_phi_phihn_spaceD} 
 and $\|D^2 \phi\|_{\infty,U} \lesssim 1$ from \eqref{eq:cnphi}
 we have 
 \begin{equation}
 \nabla (\ET \phihi{i})(y_h) = \nabla \phii{i}(x) + \mathcal{O}(h)
 = \nabla \philini{i}(x) + \mathcal{O}(h + \Delta t^{q_t+1})
 \label{eq:temp1}
 \end{equation}
 where we exploited 
 \eqref{diff_grad_phi_philin_dt}  with $(m_s,m_t)=(1,0)$.
 Now, 
 we achieved $|A-B| = \mathcal{O}(h+ \Delta t^{q_t+1})$.
 Next, we consider $D$. We know already $|\dhi{i}(x)|\lesssim h^2 + \Delta t^{q_t+1}$ and with $(m_s,m_t)=(2,0)$ in \eqref{diff_phi_phihn_spaceD} also
 $| \nabla \Ghi{i}| = |\nabla \Gi{i}| + \mathcal{O}(h^{q_s-1}+\Delta t^{q_t+1}) \lesssim 1$ and $|\nabla(\ET \phihi{i})(y_h)| = |\nabla \phii{i}(x)| + \mathcal{O}(h)\lesssim 1$ so that $|D| \lesssim h^2 + \Delta t^{q_t+1}$. 
 As a preliminary result, we can summarize $C \nabla \dhi{i}(x) = \mathcal{O}(h+ \Delta t^{q_t+1})$. 
 
 It remains to treat $C$. To this end, we use the first equality of \eqref{eq:temp1} again but additionally combine it with \eqref{G_Gh_diff_lemma_eq2} to get
 $\nabla (\ET \phihi{i})(y_h) = \Ghi{i}(x) + \mathcal{O}(h+ \Delta t^{q_t+1})$. 
 Hence, we have $C = |\Ghi{i}(x)|^2 + \mathcal{O}(h+ \Delta t^{q_t+1})$. As $|\Ghi{i}(x)|^2$ is bounded from above and below away from zero we can hence divide by $C$ and obtain 
 $\dhi{i}(x) = \mathcal{O}(h+ \Delta t^{q_t+1}).$
  }
\end{proof}

\subsection{Proof of \eqref{d_H_lemma_eq3}} \label{sec:proof:d_H_lemma_eq3}
\begin{proof}
  {
  We proceed similarly as in the proof of \eqref{eq:db} for $m_t>1$. 
  Recall the defining equation for $\dH$ from \eqref{def_dH} and defining 
  $\pi_H(x,t) := (x + \dH(x,t) \GH(x,t),t)$ we have 
  \begin{equation}
    \ET \phiH \circ \pi_H = (1-\b) \philin + \b \phiH = \philin + \b (\phiH - \philin). \tag{\ref{def_dH}'}
   \end{equation}
   For the time derivatives of $ \b ( \phiH - \philin)$, it holds with 
   \eqref{diff_phiH_phih}, \eqref{eq:phidt-philin} and \cref{ass:blending}
\begin{align*}
 \partial_t^{m_t} \left( \b( \phiH - \philin) \right) &= \sum_{j=0}^{m_t} \binom{m_t}{j} (\partial_t^j \b) (\partial_t^{m_t - j} (\phiH - \philin))  
  \lesssim h^2 + \Delta t^{q_t + 1 - m_t}
\end{align*}
With 
$\partial_t^{m_t} (\ET \phiH \circ \pi_H) = \partial_t^{m_t} \philin + \partial_t^{m_t} (\b (\phiH-\philin) )$ this implies together with another application of \eqref{diff_grad_phi_philin_dt} and \eqref{diff_ETphiH_phi}
\begin{equation*}
 | \partial_t^{m_t} (\phix \circ \pi_H - \phix) | \lesssim | \partial_t^{m_t} (\philin - \phix)| + h^2 + \Delta t^{q_t + 1 - m_t} \lesssim h^2 + \Delta t^{q_t + 1 - m_t}.
\end{equation*}
We are now in an almost identical setting as in the proof of 
\eqref{eq:db} for $m_t>1$ in \cref{app:eqdbmtgt1} before \eqref{eq:dtmtphicircpi}. We can proceed as in the second half of that proof with the only change that $\pi$ is replaced by $\pi_H$ and hence $\d\G$ is replaced by $\dH\GH$. 
}
\end{proof}

\subsection{Proof of \cref{Psi_H_Psi_diff_lemma}} \label{app:Psi_H_Psi_diff_lemma}
\begin{proof}
  We start with a simple splitting of the difference:
  \begin{align*}
    \| & \partial_t (\PsiHG - \PsiG) \|_{\infty,\QGn} 
    = \| \partial_t (\dH \GH - \d \G ) \|_{\infty,\QGn} 
    \\
    &\leq \| \partial_t (\d (\GH - \G )) \|_{\infty,\QGn} 
    + \| \partial_t ((\dH - \d) \GH) \|_{\infty,\QGn} 
    \\
    & \leq \| (\GH - \G ) \|_{\infty,\QGn}  \|\partial_t \d \|_{\infty,\QGn}    
      + \| \d \|_{\infty,\QGn}  \|\partial_t (\GH - \G ) \|_{\infty,\QGn} 
      \\
    &\hphantom{\leq} 
    + \| \dH - \d \|_{\infty,\QGn}  \| \partial_t \GH \|_{\infty,\QGn} 
    + \| \GH \|_{\infty,\QGn}  \|\partial_t (\dH - \d)  \|_{\infty,\QGn}          
     = \!{\!I\!} + {\!II\!} + {\!III\!} + {\!IV\!}
  \end{align*}
  We treat the terms $I$, $II$, $III$ and $IV$ one after another, where sufficient estimates for the first three terms are comparably easily compiled and $IV$ is more involved. To reduce the complexity of the proof we only treat the case where $G_h = \nabla \phi_h$ and  $G_H = \nabla \phi_H$, but the proof can be adapted to the case where $G_h = \PhG \nabla \phi_h$ and  $G_H = \PhG \nabla \phi_H$.

  For $I$ we obtain with
  \eqref{diff_phiH_phi} (with $m_s=1$, $m_t=0$) and \eqref{eq:db} the bound 
  $I \lesssim (h^{q_s} + \Delta t^{q_t+1}) (h^2 + \Delta t^{q_t})$ which suffices for the desired bound $h^{q_s+1} + \Delta t^{q_t}$. 
  For $II$ we use  \eqref{diff_phiH_phi} (with $m_t=1$, $m_s=1$) and \eqref{eq:da} and to get $II \lesssim (h^{q_s} + \Delta t^{q_t}) (h^2 + \Delta t^{q_t+1})$. With \eqref{delta_t_smallness} this also yields a sufficient bound. 
  For $III$ we have $\|\partial_t \GH \|_{\infty,\QGn} \leq \|\partial_t \G \|_{\infty,\QGn} + \|\partial_t (\G - \GH) \|_{\infty,\QGn} \lesssim 1$ with 
  \eqref{diff_phiH_phi} (with $m_t=1$, $m_s=1$) and \eqref{eq:cnphi}
  and bound
  \begin{equation*}
  \| \dH - \d \|_{\infty,\QGn} \leq 
  \| \dH - \ddh \|_{\infty,\QGn}
 + \| \ddh - \It{q_t} \d \|_{\infty,\QGn}
 + \| \It{q_t} \d - \d \|_{\infty,\QGn} 
  = III_a + III_b + III_c
  \end{equation*}
  where $\It{q_t} \d$ is the Lagrange interpolant of $\d$ at the time nodes $\ti{i}$, $i=0,\dots,q_t$ (note the regularity for $\d$ provided in \cref{lem:d}).
  For $III_a$ we have $III_a \lesssim \Delta t^{q_t+1}$ from the previous lemma. For $III_c$ we have
  $\| \It{q_t} \d - \d \|_{\infty,\QGn} \lesssim \Delta t^{q_t+1} \| \partial_t^{q_t+1} \d \|_{\infty,\QGn} \lesssim \Delta t^{q_t+1}$ with \eqref{eq:db}. The expression in $III_b$ is discrete in time and it hence suffices to bound terms on the time nodes $i = 0, \dots, q_t$, i.e $\| \dhi{i} -  \di{i} \|_{\infty,\OG}$. Hence, with \eqref{eq:di-ddti} and \eqref{ddt_d_hi_diff_bound} we obtain the bound $III \lesssim h^{q_s+1} + \Delta t^{q_t+1}$ which is sufficient for the desired bound.
  Finally, for $IV$ we have $\| \GH \|_{\infty,\QGn} \leq \| \G \|_{\infty,\QGn} + \| \GH - \G \|_{\infty,\QGn} \lesssim 1$ 
  with \eqref{diff_phiH_phi} (with $m_t=0$, $m_s=1$) and \eqref{eq:cnphi}. 
  
  It remains to establish a bound $\| \partial_t (\dH - \d) \|_{\infty,\QGn} \lesssim h^{q_s+1} + h^{q_t}$. 
    To this end we follow similar ideas as in the proof of \eqref{eq:db} for $m_t=1$.  
  From the definitions of $\dH$ and $\d$ in 
  \eqref{def_dH} and \eqref{ddef} we have 
  \begin{align*}
    \ET \phiH (x + (\dH \GH)(x,t),t) & = (1-\b(x,t)) \philin(x,t) + \b(x,t) \phiH(x,t) \\ 
    \phix(x + (\d \G)(x,t), t) & = (1- \b(x,t)) \philin(x,t) + \b(x,t) \phix(x,t) 
\intertext{Subtracting both equations (and smuggling in $\phix(x + (\dH \GH)(x,t),t)$) yields after dropping the arguments $(x,t)$ for $\dH \GH$ and $\d \G$
} 
\underbrace{\phix(x + \dH \GH, t) - \phix(x + \d \G, t)}_{=:A(x,t)} & = \underbrace{(\phix - \ET \phiH)(x + \dH \GH,t)}_{=:B(x,t)}  + \underbrace{\b(x,t) (\phiH - \phix)(x,t)}_{=:C(x,t)}.
 \\[-6ex]
  \end{align*}
Now deriving $A(x,t)$ after time yields  \vspace*{-0.5cm}
\begin{align*} 
  \partial_t& A(x,t) = \partial_t(\dH \GH - \d \G) \!\cdot\! \overbrace{\nabla \phi(x,t)}^{=\G} - \overbrace{(\partial_t \phi(x + \d \G, t) - \partial_t \phi(x + \dH \GH, t))}^{=:D(x,t)} \\
  &= \G \!\cdot\! ( \partial_t (\dH - \d) \GH + \partial_t \d \!\cdot\! (\GH \!-\! \G)) + \partial_t(\GH\!-\!\G) \!\cdot\! \d + \partial_t \GH \!\cdot\! (\dH  \! - \! \d) - D(x,t) \\
  &= \G\!\cdot\!\GH \partial_t (\dH - \d) + \underbrace{\partial_t \d \GH \!\cdot\! (\GH \!-\! \G)}_{=:A_2(x,t)} + \underbrace{\partial_t(\GH\!-\!\G) \!\cdot\! \d}_{=:A_3(x,t)} + \underbrace{\partial_t \GH \!\cdot\! (\dH  \! - \! \d)}_{=:A_4(x,t)} - D(x,t) \\[-5ex]
\end{align*}
which allows to isolate the desired expression $\partial_t(\dH -\d)$ and we obtain
\begin{align*}
  |\partial_t (\dH - \d) (x,t)| & \leq \frac{1}{\G\!\cdot\!\GH} \left( \sum_{i=2}^4 |A_i(x,t)| + |\partial_t B(x,t)| + |\partial_t C(x,t)| + |D(x,t)|  \right) \\
\end{align*}
There is $ \G \cdot \GH = \| \G\|^2 - \G \cdot (\G - \GH) \geq  \| \G\| ( \| \G\|  - c (h^{q_s} + \Delta t^{q_t+1}))\gtrsim 1$ (for $h$ and $\Delta t$ sufficiently small) and we can hence bound the terms $A_i$, $i=2,3,4$ and $\partial_t B$, $\partial_t C$ and $D$ one after another.
For $A_2$, $A_3$ and $A_4$ we observe that the necessary bounds follow along the lines of the bound for I, II and III above.

We have that $|\partial_t B|$ is directly in the form of the l.h.s. of \eqref{diff_ETphiH_phi} and hence has a suitable bound.
Next, we note $|\partial_t C| \leq  |\partial_t \b| |\phiH - \phix| + |\b| |\partial_t (\phiH - \phix)|$ where $|\b|, |\partial_t \b| \lesssim 1$ and $|\phiH - \phix| , |\partial_t (\phiH - \phix)| \lesssim h^{q_s+1} + \Delta t^{q_t}$ by \eqref{diff_phiH_phi}.
Finally, we bound $|D|$:
\begin{align*}
  |D(x,t)| \lesssim \|\partial_t \nabla \phix\|_{\infty} \| \PsiHG - \PsiG \|_{\infty} \lesssim \Delta t^{q_t+1} + h^{q_s+1}
\end{align*}
where the last step follows from \eqref{eq:PsiHG-PsihG} and \eqref{eq:PsihG-PsiG}. Putting it all together yields the claim.
 \end{proof}

 \subsection{Proof of \eqref{Theta_h_Psi_diff_ho_boundA} and \eqref{Theta_h_Psi_diff_ho_boundB} for \protect\FE~blending} \label{app:Theta_h_Psi_diff_ho_bounds:FEblend}
\begin{proof}
 In the case of the \FE~blending, the result on $\OG$ follow directly from \cref{Theta_h_Psi_Gamm_diff_bound_Dr_dt} 
 and \cref{Theta_h_Psi_Gamm_diff_bound_first_grad}. It remains to show the estimates for $\OGp \backslash \OG$. We start to tackle \cref{Theta_h_Psi_diff_ho_boundA} with \cref{dt_E_commutation_lemma}
  \begin{align*}
   \Delta t^{m_t} h^{m_s} \| D^{m_s} \partial_t^{m_t} (\Theh \!-\! \Psii)& \|_{\infty, \ThGQS} \!=\! \Delta t^{m_t} h^{m_s} \| D^{m_s}  \mathcal{E}^{\partial \OG} \partial_t^{m_t}\! \left( \ThehG \!-\! \PsiG  \right) \!\|_{\infty, \ThGQS}\\
   \hspace*{-1cm}\overset{\cref{ext_bnd_lemma_lo}\&\cref{ext_bnd_lemma_ho}}{\lesssim}\hspace*{-1cm}
    & \max_{F \in \Facets(\partial \OG)} \sum_{r=m_s}^{q_s+1} \Delta t^{m_t} h^{r} \| D^r \partial_t^{m_t} \left( \ThehG - \PsiG \right) \|_{\infty,F \times \In{n}}\\
   & ~~~~~~~~~~~~~~~~~~~~~~~~ + \Delta t^{m_t} \|\partial_t^{m_t} \left( \ThehG - \PsiG \right)\|_{\infty, \Verts(\partial \OG) \times \In{n}} \\
   \lesssim ~ & h^{q_s + 1} + \Delta t^{q_t + 1} \quad \textnormal{by} \quad \cref{Theta_h_Psi_Gamm_diff_bound_Dr_dt}.
  \end{align*}
  \cref{Theta_h_Psi_diff_ho_boundB} follows from \cref{Theta_h_Psi_diff_ho_boundA}
  exploiting \cref{ass:dtlesssimh_feblend}.
\end{proof}   

\subsection{Proof of \cref{interpollemma2}} \label{app:interpollemma2}
\begin{proof}
  We start with $\| u^e \circ \Psist \|_{\Hrs{0}{\ell_t}(\EQlin)}$ and apply the setting of \cite[Theorem 2.1]{faa_di_bruno_constantine}. In the following, we will use the notation of this paper without further explanation and refer the reader to \cite{faa_di_bruno_constantine} for these. We identify $d+1$ in this text with $d$ in \cite{faa_di_bruno_constantine}, $u^e \circ \Psist$ with $h$, $u^e$ with $f$ and the components of $\Psist$ with $g^{(1)}, \dots, g^{(d+1)}$. Then we investigate the constituent expressions of the norm, which are $D^{\boldsymbol \nu} (u^e \circ \Psist)$ at some point $(x,t)$ for $\boldsymbol \nu = (0,\dots,0,s_t)$ with $s_t \leq \ell_t$. Moreover, $n= s_t$ and we observe that all the expressions $f_{\boldsymbol \lambda}$ in \cite[Theorem 2.1]{faa_di_bruno_constantine} are bounded by $\| u^e \|_{H^{\ell_t}(\Psist(\EQlin))}$ at $(x,t)$. In regards to the contribution from $\Psist$, we note that the indices $\mathbf{l}_j$ must give $\boldsymbol \nu$ as the sum weighted with $|\mathbf{k}_j| > 0$. Hence, $\mathbf{l}_j \leq \boldsymbol \nu$. So only bounded contributions from $\| \Psist\|_{\Hrs{0}{\ell_t}(\EQlin)}$ appear.

 Next, we proceed similarly for $\| u^e \circ \Psist \|_{\Hrs{\ell_s}{0}(\EQlin)}$. Let $\boldsymbol \nu = (\nu_1,\dots, \nu_d, 0)$ be some appropriate derivation multi-index, i.e. $n = \sum_{i=0}^d \nu_i \leq \ell_s$. This means that again all the terms $f_{\boldsymbol \lambda}$ are bounded by $\| u^e \|_{H^{\ell_s}(\Psist(\EQlin))}$ at the point of consideration. For the derivatives of $\Psist$, we again note that the multi-indices $\mathbf{l}_j$ must sum up to $\boldsymbol \nu$ with the according weighting and hence all these are bounded by terms from $\| \Psist\|_{\Hrs{\ell_s}{0}(\EQlin)} \lesssim 1$.

 In relation to $\| \partial_t (u^e \circ \Psist) \|_{\Hrs{\ell_s}{0}(\EQlin)}$ we have to consider a multi-index $\boldsymbol \nu = (\nu_1, \dots, \nu_d, 1)$, so that $n = 1+ \sum_{i=1}^d \nu_i \leq \ell_s +1$. This causes $\| u^e \|_{H^{\ell_s+1}( \Psist(\EQlin))}$ to be an appropriate upper bound for all the $f_{\boldsymbol \lambda}$. In relation to $\mathbf{l}_j \leq \boldsymbol \nu$, all corresponding summands are contained in $\| \Psist\|_{\Hrs{\ell_s}{0}(\EQlin)} + \| \partial_t \Psist\|_{\Hrs{\ell_s}{0}(\EQlin)}$.

 For $ \| \nabla (u^e \circ \Psist) \|_{\Hrs{0}{\ell_t}(\EQlin)}$, the argument proceeds similarly with an multi-index such as $\boldsymbol \nu = (1,\dots,0,\ell_t)$, $n \leq \ell_t + 1$.
\end{proof}

\end{document}